\documentclass[12pt]{article}

\usepackage{graphicx}
\usepackage{amssymb}
\usepackage{amsmath}
\usepackage{enumerate}
\usepackage{amsfonts}
\usepackage{amsthm}
\usepackage[all,import]{xy}
\usepackage{xcolor}
\usepackage{url}
\makeatletter
\def\url@leostyle{%
  \@ifundefined{selectfont}{\def\UrlFont{\sf}}{\def\UrlFont{\small\ttfamily}}}
\makeatother
\numberwithin{equation}{section}
\usepackage[colorlinks=true, pdfstartview=FitV, linkcolor=black, 
            citecolor=black, urlcolor=black]{hyperref}
\usepackage{bookmark}
\usepackage[titletoc]{appendix} %For some reason, putting the bookmark and appendix packages after hyperref gets rid of the \endcsname error.
\theoremstyle{definition}
\newtheorem{prop}{Proposition}[section]
\newtheorem{proposition}[prop]{Proposition}
\newtheorem{theorem}[prop]{Theorem}
\newtheorem{lemma}[prop]{Lemma}

\newtheorem{corollary}[prop]{Corollary}

\newtheorem{defn}[prop]{Definition}
\newtheorem{example}[prop]{Example}
\newtheorem{notation}[prop]{Notation}
\newtheorem{remark}[prop]{Remark}
\newtheorem{image}[prop]{Figure} % Added for indexing figures more easily.

\newcommand{\nc}{\newcommand}
\nc{\DMO}{\DeclareMathOperator}	

\nc{\newnotation}{\nomenclature}
\nc{\wrap}{\cW}
\nc{\Cob}{\mathsf{Cob}}
\nc{\mul}{\mathsf{Mul}}
\nc{\fat}{\mathsf{fat}}
\nc{\cob}{\mathsf{Cob}}
\nc{\coh}{\mathsf{Coh}}
\nc{\idem}{\mathsf{Idem}}
\nc{\sets}{\mathsf{Sets}}
\nc{\near}{\mathsf{near}}
\nc{\sing}{\mathsf{Sing}}
\nc{\symp}{\mathsf{Symp}}
\nc{\perf}{\mathsf{Perf}}
\nc{\ssets}{\mathsf{sSets}}
\nc{\cmpct}{\mathsf{cmpct}}
\nc{\compact}{\mathsf{cmpct}}
\nc{\pwrap}{\mathsf{PWrap}}
\nc{\coder}{\mathsf{Coder}}
\nc{\bimod}{\mathsf{Bimod}}
\nc{\grmod}{\mathsf{GrMod}}
\nc{\spaces}{\mathsf{Spaces}}
\nc{\pwrms}{\mathsf{PWrFuk}_{M,S}}
\nc{\pwrmf}{\mathsf{PWrFuk}_{M,F}}
\nc{\pwrapmf}{\mathsf{PWrFuk}_{M,F}}
\nc{\fuk}{\mathsf{Fukaya}}
\nc{\infwr}{\mathsf{InfWr}}
\nc{\fukaya}{\mathsf{Fukaya}}
\nc{\autml}{\mathsf{Aut}_{M,\Lambda}}
\nc{\fukml}{\mathsf{Fukaya}_{M,\Lambda}}
\nc{\fukmle}{\mathsf{Fukaya}_{M,\Lambda,\epsilon}}
\nc{\fukmod}{\wrfukcompact(M)\modules}
\nc{\lag}{\mathsf{Lag}}
\nc{\lagm}{\lag_M}
\nc{\lago}{\lag^o}
\nc{\lagml}{\lag_{M,\Lambda}} % For when I get lazy.
\nc{\lagmle}{\lag_{M,\Lambda,\epsilon}}
\nc{\fun}{\mathsf{Fun}}
\nc{\vect}{\mathsf{Vect}}
\nc{\chain}{\mathsf{Chain}}
\nc{\wrfuk}{\mathsf{WrFukaya}}
\nc{\wrfukcompact}{\mathsf{WrFukaya}_{\mathsf{cmpct}}}
\nc{\pwrfuk}{\mathsf{PWrFukaya}}
\nc{\inffuk}{\mathsf{InfFuk}}
\nc{\pwrfukml}{\mathsf{PWrFukaya}_{M,\Lambda}}
\nc{\inffukml}{\mathsf{InfFuk}_{M,\Lambda}}
\nc{\nattrans}{\mathsf{NatTrans}}
\nc{\corres}{\mathsf{Corres}}
\nc{\fukep}{\fukaya_\Lambda(M,\epsilon)}
\nc{\fukepop}{\fukaya_\Lambda(M,\epsilon)^{\op}}
\nc{\lagep}{\lag_\Lambda(M,\epsilon)}
\DMO{\cyl}{cyl} % Cylindrical
\nc{\dbcoh}{D^b\mathsf{Coh}}
\nc{\corr}{\mathsf{Corr}}

\nc{\cat}{\mathsf{Cat}}
\nc{\Cat}{\mathsf{Cat}}
\nc{\ainfty}{\mathsf{A}_\infty}
\nc{\inftycat}{\mathcal{C}\!\operatorname{at}_\infty}
\nc{\Ainftycat}{\mathcal{C}\!\operatorname{at}_{A_\infty}}
\nc{\ainftycat}{\mathcal{C}\!\operatorname{at}_{A_\infty}}
\nc{\stablecat}{\mathcal{C}\!\operatorname{at}_\infty^{\Ex}}

\DMO{\im}{im}
\DMO{\inj}{inj}
\DMO{\fib}{fib}
\DMO{\conf}{Conf}
\DMO{\chains}{Chains}
\DMO{\cochains}{Cochains}
\DMO{\cone}{Cone}
\DMO{\ran}{Ran}
\DMO{\rot}{Rot}
\DMO{\leg}{Leg}
\DMO{\imm}{imm}
\DMO{\adj}{adj}
\DMO{\cube}{Cube}
\DMO{\back}{back}
\DMO{\front}{front}
\DMO{\flow}{Flow}
\DMO{\floer}{Floer}
\DMO{\maps}{Maps}
\DMO{\exact}{exact}
\DMO{\Decomp}{Decomp}
\DMO{\decomp}{Decomp}
\DMO{\collar}{collar}
\DMO{\yoneda}{Yoneda}
\DMO{\hamspace}{Ham}
\DMO{\sympspace}{Symp}
\DMO{\holomaps}{Holomaps}
\DMO{\comp}{Comp}
\DMO{\crit}{Crit}
\DMO{\test}{{test}}
\DMO{\sign}{sign}
\DMO{\topp}{top}
\DMO{\indx}{Index}
\DMO{\Break}{Break} % Partitions
\DMO{\zero}{zero} %Zero
\DMO{\ob}{Ob}
\DMO{\gr}{Gr} % Grassmanian
\DMO{\Gr}{Gr} % Grassmanian
\DMO{\cl}{Cl} % Clifford Algebra
\DMO{\grlag}{GrLag}
\DMO{\Pin}{Pin}
\DMO{\Graph}{Graph}
\DMO{\pin}{Pin}
\DMO{\gap}{Gap}
\DMO{\Ex}{Ex}
\DMO{\id}{id}
\DMO{\End}{End}
\DMO{\sym}{Sym} 
\DMO{\aut}{Aut}
\DMO{\DK}{DK} %Dold-Kan
\DMO{\poly}{poly} % Polynomial deRham forms
\DMO{\diff}{Diff}
\DMO{\coll}{coll}
\DMO{\dist}{dist} %Distance function
\DMO{\coker}{coker} %Cokernel
\nc{\kernel}{\ker} %Kernel
\DMO{\sspan}{span}
\DMO{\hocolim}{hocolim}	
\DMO{\holim}{holim}
\DMO{\sk}{sk}

\DMO{\ho}{ho}
\DMO{\fin}{fin}
\DMO{\ret}{Ret}
\DMO{\ham}{Ham}
\DMO{\con}{con}
\DMO{\leaf}{leaf}
\DMO{\supp}{supp}
\DMO{\edge}{edge}
\DMO{\colim}{colim}
\DMO{\edges}{edges}
\DMO{\Image}{image}
\DMO{\roots}{roots}
\DMO{\height}{height}
\DMO{\finmod}{FinMod}
\DMO{\leaves}{leaves}
\DMO{\planar}{planar}
\DMO{\vertices}{vertices}

\nc{\lagg}{\lag^{\cG}}
\nc{\iso}{\mathsf{Iso}}
\nc{\Set}{\mathsf{Set}}
\nc{\ass}{\mathsf{ \bf Ass}}
\nc{\Mod}{\mathsf{Mod}}
\nc{\modules}{\mathsf{Mod}}
\nc{\sset}{\mathsf{sSet}}
\nc{\liou}{\mathsf{Liou}}
\nc{\poset}{\mathsf{Poset}}
\nc{\trno}{T^*\RR^n_{\geq 0}}
\nc{\spectra}{\mathsf{Spectra}}
\nc{\tensorfin}{\tensor^{\fin}}
\nc{\lagptg}{\lag_{pt,pt}^{\cG}}
\nc{\Fin}{\mathcal{F}\mathsf{in}}
\nc{\lagnl}{\lag_{N,\Lambda}}
\nc{\lagmlg}{\lag_{M,\Lambda}^{\cG}}
\nc{\lagsplit}{\lag^{\mathsf{split}}}
\nc{\lagktimes}{(\lag^{\dd k})^\times}
\nc{\lagplanar}{\lag^{\times,\planar}}

\nc{\smsh}{\wedge}
\nc{\un}{\underline}
\nc{\xto}{\xrightarrow}
\nc{\xra}{\xto}
\nc{\tensor}{\otimes}
\nc{\del}{\partial}
\nc{\dd}{\diamond}
\nc{\tri}{\triangle}
\nc{\bb}{\Box}
\nc{\into}{\hookrightarrow}
\nc{\onto}{\twoheadrightarrow}
\nc{\contains}{\supset}
\nc{\transverse}{\pitchfork}
\nc{\uncirc}{\underline{\circ}}
\nc{\Jbar}{\overline{J}}
\nc{\Fbar}{\overline{F}}
\nc{\delbar}{\overline{\del}}
\nc{\thetabar}{\overline{\theta}}
\nc{\omegabar}{\overline{\omega}}

\nc{\colldiff}{\diff^{\del}} 
\nc{\trbar}{\overline{T^*\RR}}
\nc{\tr}{T^*\RR}
\nc{\tsa}{Ts\cA}
\nc{\tsb}{Ts\cB}
\nc{\cmbar}{\overline{\cM}}
\nc{\crbar}{\overline{\cR}}

\nc{\vece}{ {\vec \epsilon}}	
\nc{\vecd}{ {\vec \delta}}
\nc{\ov}{\overline}
\DMO{\op}{op}
\nc{\opp}{ ^{\op}}

\nc{\hiro}{\textcolor{blue}}

\nc{\eqn}{\begin{equation}}
\nc{\eqnn}{\begin{equation}\nonumber}
\nc{\eqnd}{\end{equation}}
\nc{\enum}{\begin{enumerate}}
\nc{\enumd}{\end{enumerate}}

\def\cA{\mathcal A}\def\cB{\mathcal B}\def\cC{\mathcal C}\def\cD{\mathcal D}
\def\cG{\mathcal G}\def\cH{\mathcal H}
\def\cJ{\mathcal J}\def\cL{\mathcal L}
\def\cM{\mathcal M}\def\cP{\mathcal P}
\def\cR{\mathcal R}
\def\cW{\mathcal W}

\def\BB{\mathbb B}\def\CC{\mathbb C}

\def\RR{\mathbb R}\def\SS{\mathbb S}

\def\ZZ{\mathbb Z}

\title{The Fukaya category pairs with Lagrangian cobordisms}
\author{Hiro Lee Tanaka}
\begin{document}

\maketitle

\begin{abstract}
Fix a suitably convex, exact symplectic manifold M. We consider the stable $\infty$-category Lag(M) of non-compact Lagrangians whose (higher) morphisms are (higher) Lagrangian cobordisms between them. We show that this $\infty$-category pairs with the Fukaya category Fuk(M) of compact branes. In fact, we also show that there is a subcategory of Lag(M) which pairs with the wrapped Fukaya category of M. This is a first step in a project to enrich wrapped Fukaya categories over cobordism spectra. As a corollary, we show that cobordant compact branes are equivalent in the Fukaya category. We will also mention several other applications (without proof) of the $\infty$-categorical appraoch: One can realize Seidel's representation as a $\pi_0$-level consequence of a map of spaces; stable cobordism groups of non-compact branes map to Floer cohomology groups; some of Biran-Cornea's results can be recovered from the colored planar operad associated to the s-dot constructions of each category; and there is an $E_\infty$ map of spectra $\cL \to H\ZZ$ from exact Lagrangian cobordisms in Euclidean space to the integers.
\end{abstract}

\tableofcontents

% Activate the following line by filling in the right side. If for example the name of the root file is Main.tex, write
% "...root = Main.tex" if the chapter file is in the same directory, and "...root = ../Main.tex" if the chapter is in a subdirectory.
 
%!TEX root = _pairing.tex

\section{Introduction}
Fix an exact symplectic manifold $M$, and assume it is suitably convex at infinity (see Section~\ref{section.M}). This convexity gives a good restriction on what kinds of {\em non-compact} submanifolds to consider.

In previous work with David Nadler~\cite{nadler-tanaka}, we constructed from this data an $\infty$-category of Lagrangian cobordisms. We will call it $\lag_\Lambda(M)$ in the present work. The dependence on $\Lambda$ will be explained shortly.

The objects of $\lag_\Lambda(M)$ are certain Lagrangian submanifolds of $M \times T^*\RR^m$ for any $m\geq 0$. They are generally non-compact, but closed as a subset. A morphism between two objects is a certain kind of Lagrangian cobordism between them, and what makes $\lag_\Lambda(M)$ an $\infty$-category (as opposed to an ordinary category) is that we have specified what it means for morphisms to have homotopies between them, and for homotopies to have further homotopies between them. In our setting, a (higher) homotopy is given by a higher-dimensional Lagrangian cobordism between Lagrangian cobordisms. If one likes, one may think of such higher cobordisms as manifolds with corners, whose boundary edges, corners, and higher-codimension strata are collared by lower cobordisms.

As an example, if $L \subset M$ is a Lagrangian, the Lagrangian $L \times \RR \subset M \times T^*\RR$ is an example of the trivial cobordism, and serves as an identity morphism for $L$. The Lagrangian $L \times \RR^2 \subset M \times T^*\RR^2$ is a trivial higher cobordism; it is a trivial homotopy from the identity to itself. Though these examples don't show it, we emphasize that higher cobordisms may have non-trivial topology, and may be non-compact so long as they are eventually conical.

A key feature of the construction is that, in fact, for any specified subset $\Lambda \subset M$,  one can construct an $\infty$-category $\lag_\Lambda(M)$ whose cobordisms avoid $\Lambda$ in an appropriate way.\footnote{See Definition~\ref{defn.non-characteristic}.} In this paper, we consider two choices of $\Lambda$: When $\Lambda$ equals the skeleton $\sk(M)$ of $M$ (or a tubular neighborhood thereof), and when $\Lambda$ equals all of $M$ itself.
 
The main theorem of our previous work~\cite{nadler-tanaka} is that $\lag_\Lambda(M)$ is in fact a {\em stable} $\infty$-category for any $\Lambda$. This means that---even without performing any kind of formal completion---$\lag_\Lambda(M)$ enjoys rich algebraic properties. One can form direct sums of objects, take mapping cones and kernels of morphisms, and naturally construct a ``shift'' functor in much the same way that one can shift a chain complex up or down in grading. In fact, on objects, the shift functor for $\lag_\Lambda(M)$ is equivalent to the shift functor for $\ZZ$-graded Fukaya categories. Another consequence is that the homotopy category of $\lag_\Lambda(M)$ is naturally triangulated, while $\lag_\Lambda(M)$ itself is enriched over spectra (in the sense of stable homotopy theory).

Moreover, in~\cite{nadler-tanaka} we conjectured that this $\infty$-category of Lagrangian cobordisms is actually equivalent to the wrapped Fukaya category of $M$ after a change of coefficients:
	\eqn\label{eqn.conjecture}
	D^\pi\lag_{\sk(M)}(M) \tensor_{\cL} \ZZ \simeq D^\pi \wrap(M).
	\eqnd
Let me explain this conjecture. 
\enum
	\item
		First, $\wrap(M)$ is the wrapped Fukaya category of $M$. $D^\pi \wrap(M)$ is the pre-triangulated $A_\infty$-category one obtains by Karoubi-completing $\wrap(M)$ to make it well-behaved algebraically. As a result, one can take direct sums of objects, take mapping cones of morphisms, and split idempotent morphisms. (These are formal properties also enjoyed by categories of modules over a ring.) 
	\item
		Likewise, $D^\pi\lag_{\sk(M)}(M)$ is an idempotent completion of $\lag_{\sk(M)}(M)$. As we mentioned, $\lag_{\sk(M)}(M)$ already enjoys good algebraic properties, but there is no proof that its idempotents split, so we complete it.
	\item
		The conjecture further claims that each $\lag_{\sk(M)}(M)$ is an $\infty$-category which is linear over a ring spectrum called $\cL$. (Though this statement is only true in the world of spectrally enriched categories, a reader will not lose intuition by imagining that $\lag(M)$ is some dg category linear over some coefficient ring $\cL$ in the dg sense.) For the lefthand side of~(\ref{eqn.conjecture}) to make sense, one must also show that $\cL$ has an $E_\infty$ ring map\footnote{The reader will not lose too much intuition by viewing an $E_\infty$ ring map as a commutative ring map; just be aware that the former encodes many higher homotopies  {\em realizing} the respect between the commutativities of the domain and codomain rings. (This is for the same reason that an $A_\infty$ ring map requires higher homotopies realizing how associativity is respected.)} to the integers, so that one may tensor $\lag_{\sk(M)}(M)$ with $\ZZ$ to turn it into an $\infty$-category linear over $\ZZ$. We prove in~\cite{tanaka-symm} both these things. We emphasize that---just as $\wrap(M)$ is $\ZZ$-linear for many choices of $M$---the ring spectrum $\cL$ is independent of choice of $M$. It is a universal ring spectrum over which all of our theories are linear.
	\item 
		The conjecture is that after this change of coefficients, there is a natural equivalence\footnote{Here, ``natural,'' as with other parts of this paper, has a specific categorical meaning. It means the conjecture predicts a natural transformation between the assignments $(M \mapsto \lag(M))$ and $(M \mapsto \fuk(M)$. } between the ($\ZZ$-linearized) $\infty$-category of Lagrangian cobordisms, and the $\infty$-category given by the wrapped Fukaya category.
\enumd

This paper is a first in a series of papers that proves that (i) this conjecture even makes sense, and (ii) the conjecture is true in the case where $M$ is a point. This may sound silly, but this implies the result for when $M$ is $T^*\RR^m$ for any $m$, and this brings us plenty closer to proving the conjecture for whenever $M$ is a cotangent bundle by using a local-to-global property for both sides of the conjecture. (The local-to-global principle for the wrapped side has been conjectured by Kontsevich, and a proof has been announced by Ganatra-Pardon-Shende~\cite{shende-paris-talk}.) Moreover, the case of $M=pt$ studies precisely the space of exact, non-characteristic cobordisms in $T^*\RR^\infty$---a convex symplectic analogue of classical Pontrjagin-Thom theory. As a consequence, we'll see some geometric results along the way.

For more motivation for this conjecture, see Section~\ref{section.conjecture}.

\subsection{Main results}
We let $\wrap = \wrap(M)$ denote the wrapped Fukaya category of $M$. Let $\wrap\Mod$ denote the dg-category of contravariant functors from $\wrap$ to $\chain_\ZZ$. 

\begin{theorem}\label{theorem-wrapped}
There exists a functor of $\infty$-categories
	\eqnn
		\Xi: \lag_M(M) \to \wrap\Mod.
	\eqnd
If $L \subset M \times T^*\RR^m$ is an object of $\lag_M(M)$, then $\Xi(L)$ is the module
	\eqnn
		X \mapsto WF^*( X \times \RR^m, L).
	\eqnd
\end{theorem}

Now let $\wrap^\compact := \wrap^\compact(M)$ denote the full subcategory of $\wrap$ consisting of Lagrangians which are compact. 

\begin{theorem}\label{theorem-compact}
There exists a functor
	\eqnn
		\Xi: \lag_{\sk M} (M) \to \wrap^\compact\Mod
	\eqnd
with the same effect on objects.
\end{theorem}

In what follows, we will simply write $\lag$ instead of $\lag_M(M)$ or $\lag_{\sk (M)} (M)$ when a statement is true for both categories. Likewise, we will simply write $\fuk\Mod$ rather than $\wrap\Mod$ or $\wrap^\compact\Mod$ when a statement is true for both categories.

Here is a sample application of the above functor. One can prove that any compact Lagrangian cobordism in $\lag(M)$ is an equivalence. By the Yoneda Lemma, we conclude:

\begin{corollary}
Fix $L_0, L_1 \in \wrap^\compact(M)$ and assume $Q$ is a compact Lagrangian cobordism from $L_0$ to $L_1$. Then $Q$ induces an equivalence $L_0 \simeq L_1$  in $\fuk(M)$.
\end{corollary}

This result was independently obtained by Biran and Cornea in the (more general) monotone setting as well. It shows that any natural notion of characteristic classes associated to Lagrangian cobordism theory must also preserve Floer theory. Since any Hamiltonian isotopy gives rise to a Lagrangian cobordism of the above sort, this is another proof that Hamiltonian isotopies give rise to equivalences in the Fukaya category. 

\begin{remark}
Note that by the tensor-hom adjunction, one obtains a functor
	\eqnn
	\lag \times \fuk^{\op} \to \chain_\ZZ.
	\eqnd
This is what we mean by the ``pairing'' in the title of this paper. 
\end{remark}

\begin{remark}\label{remark.partial-wrapping}
At this point it is natural to wonder whether one can define an analogous pairing between $\lag_\Lambda(M)$ and some wrapped category that depends on $\Lambda$, where the two theorems above are extreme examples. The answer seems to be yes, but to the author's knowledge, this requires a set-up of partially wrapped categories that seems a little bit beyond the technology developed in Sylvan's thesis~\cite{sylvan-thesis}. So we leave this for future work. However, it is worth noting that changes in $\Lambda$ are compatible with Sylvan's stop-removal theorem. For example, if $\Lambda ' \subset \Lambda$, one can show that a natural inclusion
	\eqnn
		\lag_\Lambda(M) \to \lag_{\Lambda'}(M)
	\eqnd
inverts the class $W$ of cobordisms which avoid $\Lambda \setminus \Lambda'$ at {\em positive $\infty$}. In other words, one has an induced localization functor
	\eqnn
		\lag_\Lambda(M) \to \lag_\Lambda(M)[W^{-1}] \to \lag_{\Lambda'}(M)
	\eqnd
but we do not know whether the last arrow is an equivalence.
\end{remark}

\subsection{Relation to other works, and future directions}
As mentioned above, we view this paper as a first in a series. For some context, we briefly outline some future directions, indicating relations to other works:
\enum
	\item (Lifting the Seidel representation)
		The fact that we deal with {\em higher} cobordisms means that we can detect elements of $\pi_k$ in the space of Hamiltonian automorphisms of $M$. Concretely: Since Hamiltonians act by automorphisms of $\lag$ and of $\fuk$, one has a continuous map $\ham \to \aut\lag$ and $\ham \to \aut \fuk$. Of course, the based loop space (at the identity) of $\aut$ of a category has a map to the Hochschild cohomology of that category---hence, in the compact $M$ case, to quantum cohomology of $M$. This was utilized in the 1-categorical setting by Charette-Cornea~\cite{charette-cornea}, and one can utilize the $\infty$-categorical constructions here to give a higher lift of the Seidel representation, which induce maps from the higher homotopy groups of $\ham$ to quantum cohomology (and to the analogue of quantum cohomology for the wrapped case)~\cite{tanaka-seidel-rep}.
	\item (Exactness of $\Xi$)
		Another benefit of higher cobordisms is that---as proven in~\cite{nadler-tanaka}---cobordisms have their own mapping cones. While other works~\cite{biran-cornea, biran-cornea-2, mak-wu} see ex post facto that cobordisms induce mapping cones in the Fukaya category, higher cobordisms show that mapping cones exist a priori in the $\infty$-category of Lagrangian cobordisms. One can show that $\Xi$ in fact respects mapping cones (i.e., is an {\em exact} functor)~\cite{tanaka-exact}. This has at least two consequences:
			\enum
				\item 
					$\Xi$ induces a map of spectra between $\hom_\lag$ and $\hom_\fuk$. In particular, the higher cobordism groups of non-compact branes map to Floer cohomology groups additively.
				\item
					The construction in~\cite{biran-cornea, biran-cornea-2} can be recovered as a map at the level of the s-dot construction. (See below.)
			\enumd
	\item (Higher algebra of cobordisms)
		Finally, we prove in later work that when $M$ is a point, $\Xi$ lifts to a symmetric monoidal functor from $\lag(pt)$ to $\ZZ\Mod$~\cite{tanaka-symm}. This proof in fact gives an infinite-loop space model for the sphere spectrum and its ring structure, while also showing that the lefthand side of the conjecture~(\ref{eqn.conjecture}) makes sense.
\enumd

\begin{remark}[Other functors]
Of course, the present functor is only a pairing. We content ourselves with this because our immediate goals are to understand and construct the $E_\infty$ map $\cL \to H\ZZ$. There are in fact better functors that do not utilize test objects and staircases like this paper, but which require much more analytical work involving open-closed maps; this will be a subject of future work.
\end{remark}

\begin{remark}[Immersions]
One can construct an $(\infty,2)$-category of symplectic manifolds where $\hom(M_0,M_1) = \lag(\ov{M_0},M_1)$, provided one is content with immersed objects and immersed cobordisms. Note there are no technical obstructions to setting up such an $(\infty,2)$-category; this is contrast to the Fukaya side, where at present arbitrarily immersed branes do not yet seem to allow for a general enough Floer theory the author understands. Finally, if such a theory were to exist, it seems the conjecture~(\ref{eqn.conjecture}) is much more within reach, for the simple reason that one could perform a surgery along intersection points to create a morphism (i.e., a cobordism) in an immersed version of $\lag$. For work in this direction, which for compact $M$ utilizes the fully faithful embedding of Lagrangian correspondences into bimodules over Fukaya categories, see~\cite{mak-wu}.
\end{remark}

\begin{remark}[Non-exact setting]
The reader will note that very few of the methods employed in defining $\lag$, and in proving the existence of $\Xi$, actually rely on exactness. Here, as usual, exactness is only used to remove obstructions and guarantee Gromov compactness. In particular, we expect that Ritter-Smith's generalization of the wrapped category to the monotone setting~\cite{ritter-smith} will also apply to construct a functor $\lag \to \fuk(M)\Mod$, where one then must set $\Lambda = M$. 
\end{remark}

\begin{remark}[Other works]
For computations in the monotone setting, see Haug~\cite{haug}. Also see the work of Lara-Simone Suarez~\cite{suarez} for a generalization of~\cite{tanaka-h-cobordism}.
\end{remark}

\begin{remark}[Relation to Biran-Cornea]
Independently of~\cite{nadler-tanaka} and the present work, Biran-Cornea~\cite{biran-cornea-2} also observed algebraic structures present in Lagrangian cobordisms. Notably, \cite{biran-cornea-2} showed that a formally constructed category of Lagrangian cobordisms (with multiple ends) has a functor to a category consisting of diagrams of exact sequences (a.k.a. ``triangular decompositions'') in the Fukaya category. 

Though~\cite{biran-cornea-2} put it a different way, the formal structure on both the categories appearing in~\cite{biran-cornea-2} is that of a planar colored operad. Concretely: There is a color for every class of object, and the composition operations are parametrized over a plane (i.e., are planar) because all of the cobordisms in~\cite{biran-cornea-2} live over a plane, while all triangular decompositions are drawn on a sheet of paper (a plane). One can then straightforwardly show that the functor constructed in~\cite{biran-cornea-2} naturally defines a map of colored planar operads.

By work of Dyckerhoff-Kapranov~\cite{dyckerhoff-kapranov}, any exact functor between stable $\infty$-categories results in a map of colored planar operads, where each colored operad keeps track of simplices in the Waldhausen s-dot construction. So a natural guess is that $\Xi$ induces exactly the map of planar colored operads constructed by~\cite{biran-cornea-2}. In particular, the map in~\cite{biran-cornea, biran-cornea-2} from the 0th cobordism group to $K_0(\fuk)$ is precisely the functor on Grothendieck groups induced by the functor $\Xi$. The main obstacle to showing this is extending some of the analytical details for the wrapped category to the monotone setting as established in~\cite{ritter-smith}. This will be the subject of later work.
\end{remark}

\subsection{Some remarks on the conjecture (\ref{eqn.conjecture})}\label{section.conjecture}
\begin{remark}[Motivation]
It has long been anticipated that the wrapped category (and in fact, a putative {\em partially} wrapped category) should have a purely topological characterization. One important point of emphasis is that this topological characterization should involve no mention of holomorphic disks. Here are some example results and conjectures that indicate the spirit:
	\enum
		\item
			A theorem of Abouzaid that for $M = T^*Q$ a cotangent bundle with $Q$ connected and Spin, $D^\pi\wrap(M) \simeq \Omega Q \Mod$. That is, (the idempotent completion of) the wrapped Fukaya category is equivalent to representations of the based loop space. Note that, if one chooses the trivial locally constant cosheaf of stable categories on $Q$ where at each point of $Q$ one assigns modules over some base ring $R$, the global section of this cosheaf would indeed give rise to representations of $\Omega Q$ over the ring $R$.
		\item
			Nadler and Nadler-Zaslow's work on microlocal sheaves, and on wrapped versions thereof (characterized as perfect objects in an Ind-completion of microlocal sheaves)~\cite{nadler-wrapped-microlocal}. Here, because of the very nature of sheaves, Nadler's microlocal objects inherently have a local-to-global computation, meaning one need only understand local data and gluing data to compute the microlocal category of a whole. In all known examples, this category is equivalent to the wrapped Fukaya category, and a proof of this in general has been announced by Ganatra-Pardon-Shende~\cite{shende-paris-talk}.
		\item 
			Kontsevich's conjecture that there is a locally constant cosheaf of stable $\infty$-categories on any skeleton $\Lambda$, whose global sections recovers the partially wrapped Fukaya category of $M$ with respect to $\Lambda$. This would be a consequence of  Ganatra-Pardon-Shende.
		\item
			Dyckerhoff-Kapranov's work on topological Fukaya categories by taking colimits of categories on ribbon graphs (so by construction, this satisfies a local-to-global principle). As with Nadler's microlocal categories, with $\ZZ$ coefficients, this is equivalent to the wrapped category in known examples.
	\enumd
One observation about each of the ideas above is that specifying $\Omega Q\Mod$, or microlocal sheaves, or a locally constant cosheaf, is secondary to choosing a base ring $R$. This raises the question: Is there a notion of a Fukaya category over rings that are more general than $\ZZ$-linear rings? Of course, any unital ring in the usual sense is a $\ZZ$-algebra---so what could be more general? This question may invoke phrases like ``the field with one element'' to many, but we have a specific realm in mind: The realm of ring spectra, where the base ring is not $\ZZ$, but the sphere spectrum (over which $\ZZ$ itself is a ring).

Just as (a) the sphere spectrum can be recovered from the usual framed theory of cobordisms via Pontrjagin-Thom theory, and (b) the $\ZZ$-linearization of the sphere spectrum recovers usual homology over $\ZZ$, the conjecture (\ref{eqn.conjecture}) anticipates that (a) the geometry of non-compact exact Lagrangians in $T^*\RR^\infty$ specifies a ring spectrum $\cL$, and that (b) the $\ZZ$-linearization of the $\cL$-linear theory is wrapped Floer theory.

Of course, the most general ring in homotopy theory is the sphere spectrum. We know very little about the unit map $\SS \to \cL$.
\end{remark}

{\bf Acknowledgments.}
I'd like to thank Kevin Costello, John Francis, and David Nadler for guiding and supporting me during the preparation of this work, which is a part of my thesis. 
I would also like to thank
Paul Biran and Octav Cornea
for fruitful discussions at ETH Zurich and at IAS. 
Finally, I would like to thank 
Mohammed Abouzaid,
Denis Auroux, 
Nate Bottman,
Kenji Fukaya,
Sheel Ganatra, 
Yong-Geun Oh,
James Pascaleff, 
Tim Perutz,
Paul Seidel,
and
Zack Sylvan,
for helpful discussions, which may have long been forgotten,  on minutiae of this work.

This work was conducted while I was supported chiefly by an NSF Graduate Research Fellowship. At various points I was also supported by 
a Presidential Fellowship via Northwestern University's Office of the President, 
the Mathematical Sciences Research Institute, 
the Center for Geometry and Physics at Pohang,
Harvard University, and by the National Science Foundation under Award No. DMS-1400761. This research was also supported in part by Perimeter Institute for Theoretical Physics. Research at Perimeter Institute is supported by the Government of Canada through the Department of Innovation, Science and Economic Development and by the Province of Ontario through the Ministry of Research and Innovation. I would like to thank all these institutions for their support.

% Activate the following line by filling in the right side. If for example the name of the root file is Main.tex, write
% "...root = Main.tex" if the chapter file is in the same directory, and "...root = ../Main.tex" if the chapter is in a subdirectory.
 
%!TEX root = _pairing.tex

\section{Overview of the proof}\label{section.outline}

\begin{remark}[Prerequisites]
We assume the reader is familiar with the framework of $\infty$-categories, and of $A_\infty$-categories. No technical knowledge of either is used in this paper, with the exception of the result from~\cite{tanaka-non-strict}, which one can take as a black box. (See Section~\ref{section.degeneracy}.) 

We also assume the reader is familiar with either the telescoping (colimit) construction of wrapped cochains as in~\cite{abouzaid-seidel}, or with the construction of wrapped cochains using eventually quadratic Hamiltonians as in~\cite{abouzaid-geometric}.
\end{remark}

\begin{notation}
In what follows, we let $E = F=\RR$ be the standard Euclidean line, and $E^n = F^n = \RR^n$ standard Euclidean $n$-space. Since Euclidean space will appear in two different ways, we hope this notation will sort out the purpose of the different Euclidean coordinates. Roughly speaking:

{\em $E^n$ are the directions in which one {\em stabilizes} to create more room. The $F^N$ are the directions in which {\em cobordisms propagate}.}

We will let $E \subset T^*E$ denote the zero section, and we will identify $E^\vee$---the $\RR$-linear dual of $E$---with the cotangent fiber of $T^*E$ at the origin. So
	\eqnn
		E^\vee \cong T^*_0 E.
	\eqnd
Likewise for $F$, and $E^n$, and $F^N$. 
\end{notation}

We now explain the construction of the functor $\Xi$. To do this, recall first that $\lag(M)$ itself is constructed as a union. Namely, for each $n \in \ZZ_{\geq 0}$, one has an $\infty$-category $\lag^{\dd n}(M)$. Its objects are branes $Y$ in the exact symplectic manifold $M \times T^*E^n$. Its morphisms are Lagrangian submanifolds of $M \times T^*E^n \times T^*F$, with collaring conditions rendering them cobordisms. Its $N$-morphisms are Lagrangian submanifolds of $M \times T^*E^n \times T^*F^N$, again with collaring conditions rendering the branes higher cobordisms.

Now, there is a natural way to take branes $Y \subset M \times T^*E^n$ and produce branes in $M \times T^*E^{n+1}$: form the direct product with $E^\vee$. This likewise produces, for any cobordism inside $M \times T^*E^n \times T^*F^N$, a cobordism inside $M \times T^*E^{n+1} \times T^*F^N$. Because this process preserves all collaring conditions, it yields a functor
	\eqnn
		\times E^\vee: \lag^{\dd n}(M) \to \lag^{\dd n+1}(M).
	\eqnd
And $\lag(M)$, by definition, is the increasing union (or colimit) of this $\ZZ_{\geq0}$-indexed diagram.
	\eqnn
		\lag(M) := \bigcup_{n \geq 0} \lag^{\dd n}(M).
	\eqnd
So, by definition, to construct $\Xi$, it suffices to construct a functor $\lag^{\dd n}(M) \to \fun_{A_\infty}(\fuk(M)^{\op}, \chain_\ZZ)$ for each $n$, and then prove that the diagram
	\eqn\label{eqn.stabilization-criterion}
		\xymatrix{
			\lag^{\dd n}(M) \ar[r]^-{\Xi} \ar[d]_{ \times E^\vee} & \fun_{A_\infty}(\fuk(M)^{\op},\chain_\ZZ) \\
			\lag^{\dd n+1}(M) \ar[ur]_-{\Xi}
		}
	\eqnd
commutes.

We outline this construction now, by explaining the following:

\enum
	\item
		The construction of $\Xi$ on the objects and morphisms of $\lag^{\dd 0}$. The key here is explaining how one obtains higher natural transformations in $\fuk\Mod$ from higher cobordisms.
	\item
		The construction of $\Xi$ on $\lag^{\dd n}$ for $n \geq 1$, and why the diagram (\ref{eqn.stabilization-criterion}) commutes.
\enumd

We refer the reader to Section~\ref{section.geometry} for the details on the wrappings, along with proofs of transversality and compactnesss. We refer also to Section~\ref{section.algebra} for details on how one constructs a functor from an $\infty$-category to a dg-category.

\subsection{$\Xi$ on $\lag^{\dd 0}$}
When $n=0$, $\Xi$ is easily described on objects. To any brane $Y \in \ob \lag^{\dd 0}$, one assigns the module that $Y$ represents as an object of the Fukaya category.\footnote{Recall our convention that $\fuk$ is a stand-in for both $\wrap$ and $\wrap^\compact$; the distinction will not be relevant until we begin to discuss the kinds of cobordisms we allow in $\lag$.} That is,
	\eqnn
		Y \mapsto CW^*(-,Y) \in \fuk\Mod.
	\eqnd
Now let $Y_{01} \subset M \times T^*F$ be a Lagrangian cobordism from $Y_0$ to $Y_1$. One must now assign a map of modules---i.e., a natural transformation---from $CW^*(-,Y_0)$ to $CW^*(-,Y_1)$. To do this, one employs a construction which was developed in~\cite{nadler-tanaka} to turn cobordisms into objects of $\lag$; a variant using another language, but studying more or less the same moduli space of disks, was also developed independently by Biran and Cornea~\cite{biran-cornea}. We describe the construction in pictures here.

\begin{figure}
		\[
			\xy
			\xyimport(8,8)(0,0){\includegraphics[width=1.5in]{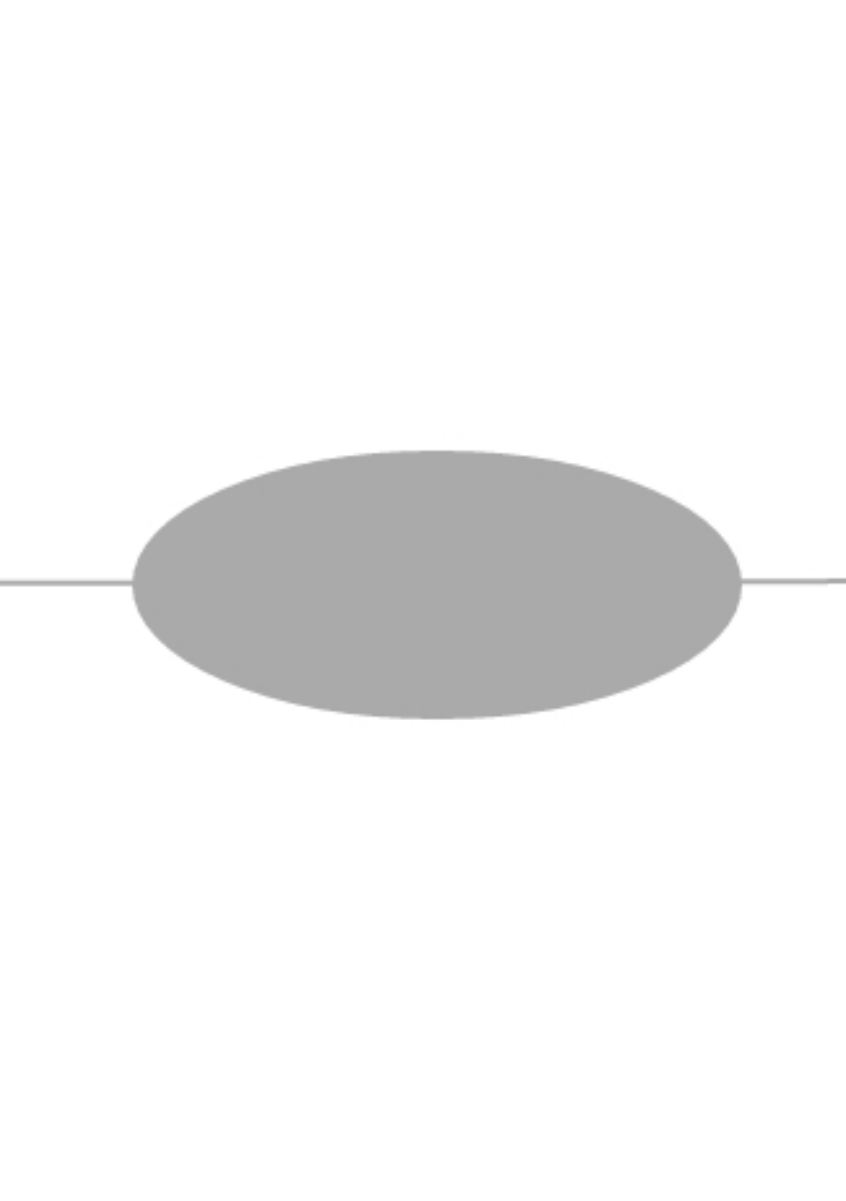}}
			%%Grid, for help.
%		 	,(2,1)*+{\bullet2},(3,1)*+{\bullet3},(4,1)*+{\bullet4},(5,1)*+{\bullet5},(6,1)*+{\bullet6},(7,1)*+{\bullet7},(1,1)*+{\bullet1},(1,2)*+{\bullet2},(1,3)*+{\bullet3},(1,4)*+{\bullet4},(1,5)*+{\bullet5},(1,6)*+{\bullet6},(1,7)*+{\bullet7}
				%%End Helping grid
			,(4,0)*+{\text{(a)}}
			\endxy
			\qquad
			\xy
			\xyimport(8,8)(0,0){\includegraphics[width=1.5in]{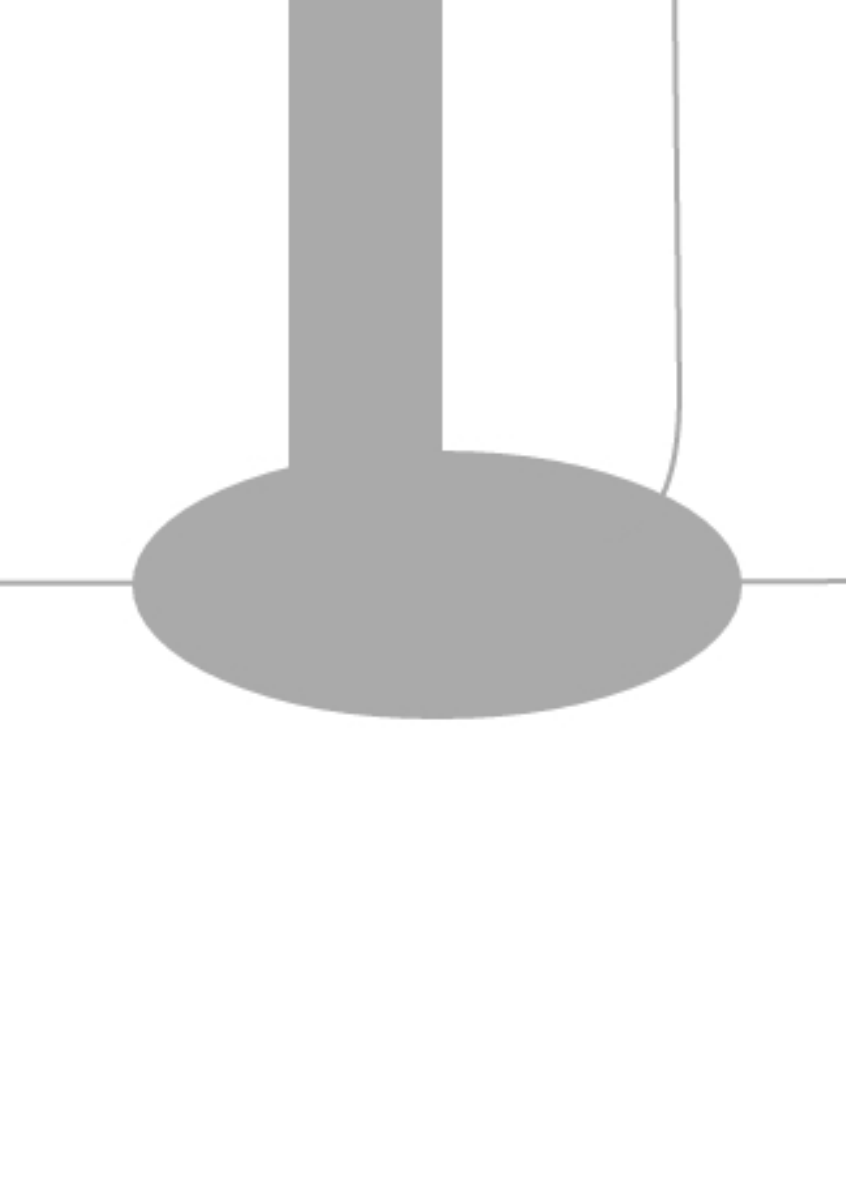}}
			%%Grid, for help.
%		 	,(2,1)*+{\bullet2},(3,1)*+{\bullet3},(4,1)*+{\bullet4},(5,1)*+{\bullet5},(6,1)*+{\bullet6},(7,1)*+{\bullet7},(1,1)*+{\bullet1},(1,2)*+{\bullet2},(1,3)*+{\bullet3},(1,4)*+{\bullet4},(1,5)*+{\bullet5},(1,6)*+{\bullet6},(1,7)*+{\bullet7}
				%%End Helping grid
			,(4,0)*+{\text{(b)}}
			\endxy
			\qquad
			\xy
			\xyimport(8,8)(0,0){\includegraphics[width=1.5in]{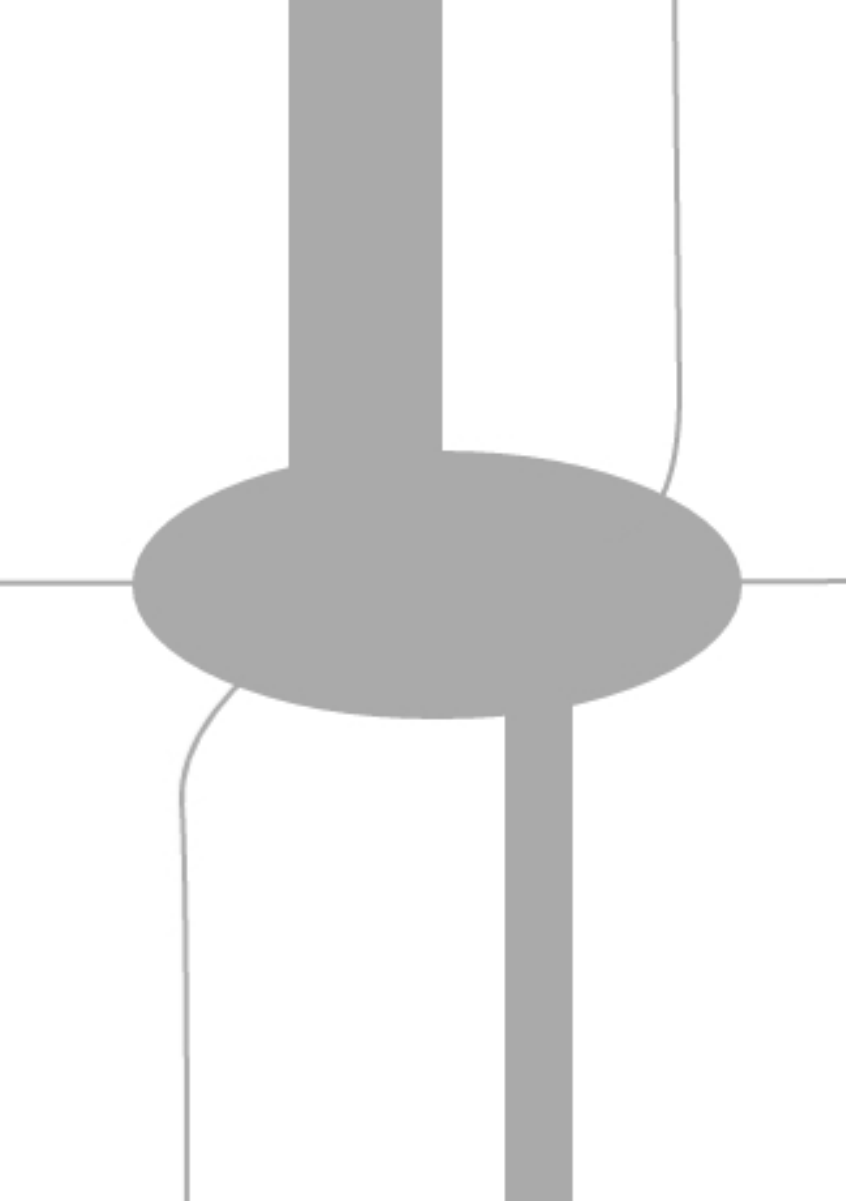}}
			%%Grid, for help.
%		 	,(2,1)*+{\bullet2},(3,1)*+{\bullet3},(4,1)*+{\bullet4},(5,1)*+{\bullet5},(6,1)*+{\bullet6},(7,1)*+{\bullet7},(1,1)*+{\bullet1},(1,2)*+{\bullet2},(1,3)*+{\bullet3},(1,4)*+{\bullet4},(1,5)*+{\bullet5},(1,6)*+{\bullet6},(1,7)*+{\bullet7}
				%%End Helping grid
			,(4,0)*+{\text{(c)}}
			\endxy
			\]
\begin{image}\label{figure.cobordisms}
Three different cobordisms in $M \times T^*F$, depicted by the image of their projection to $T^*F$. (a) depicts a bounded cobordism, which has no unbounded component in the $F^\vee$ direction. (b) depicts a cobordism which only goes to $+\infty$ in the $F^\vee$ direction, while (c) depicts a cobordism with components going to both $+\infty$ and $-\infty$ in the $F^\vee$ direction.
\end{image}
\end{figure}

In Figure~\ref{figure.cobordisms} are depicted three typical cobordisms. We focus on types (b) and (c). In Figure~\ref{figure.cobordisms}(b), one sees a cobordism that does not appproach $-\infty$ in the $F^\vee$ direction. We say that the cobordism {\em avoids $M$}. Every morphism in $\lag_M(M)$ is of type (a) or (b). 

In Figure~\ref{figure.cobordisms}(c), the cobordism does go off to $-\infty$ in the $F^\vee$ direction. But let us assume that it does so in a way which never intersects a tubular neighborhood of $\sk(M) \subset M$. (See Definition~\ref{defn.non-characteristic} for details.) Every morphism in $\lag_{\sk(M)}(M)$ is of type (a), (b), or (c). 

By definition of a map of modules, one must now construct---for any sequence of objects $X_0,\ldots,X_k$ in $\fuk$---maps
	\eqn\label{eqn.module-operations}
		CW^*(X_k, Y_0) \tensor \ldots \tensor CW^*(X_0,X_1) \to CW^*(X_0,Y_1).
	\eqnd
One does this as follows: We first choose certain curves $\gamma_i \subset T^*F$, and then count holomorphic disks in $M \times T^*F$ with boundaries on the Lagrangians
	\eqnn
		B(Y), \qquad X_0 \times \gamma_0, \qquad \ldots, \qquad X_k \times \gamma_k.
	\eqnd
These are pictured in Figure~\ref{figure.cobordism-disk}. 

\begin{figure}
		\[
			\xy
			\xyimport(8,8)(0,0){\includegraphics[width=1.5in]{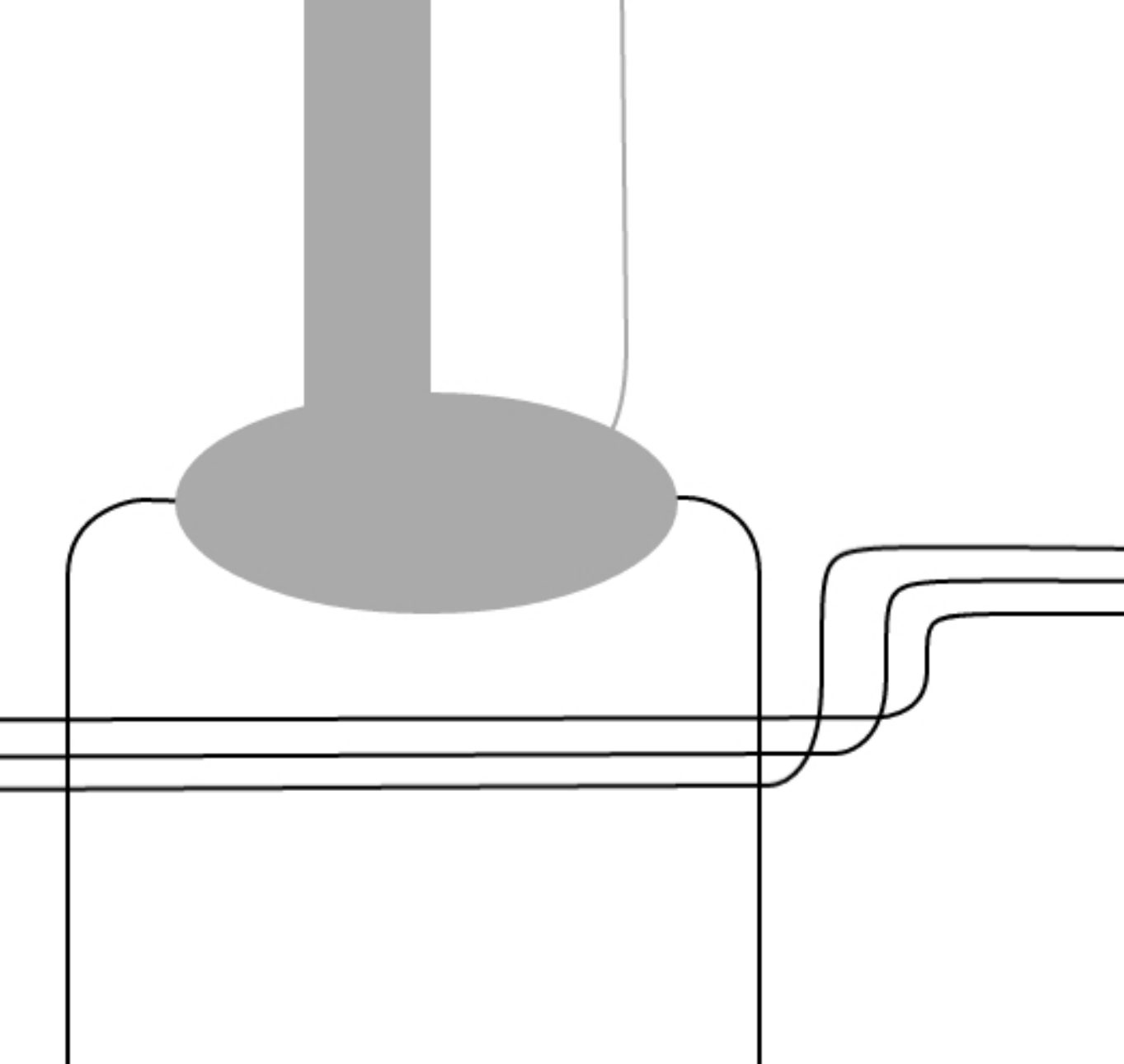}}
			%%Grid, for help.
%		 	,(2,1)*+{\bullet2},(3,1)*+{\bullet3},(4,1)*+{\bullet4},(5,1)*+{\bullet5},(6,1)*+{\bullet6},(7,1)*+{\bullet7},(1,1)*+{\bullet1},(1,2)*+{\bullet2},(1,3)*+{\bullet3},(1,4)*+{\bullet4},(1,5)*+{\bullet5},(1,6)*+{\bullet6},(1,7)*+{\bullet7}
				%%End Helping grid
			,(2,0)*+{\text{(a)}}
			\endxy
			\qquad
			\xy
			\xyimport(8,8)(0,0){\includegraphics[width=1.5in]{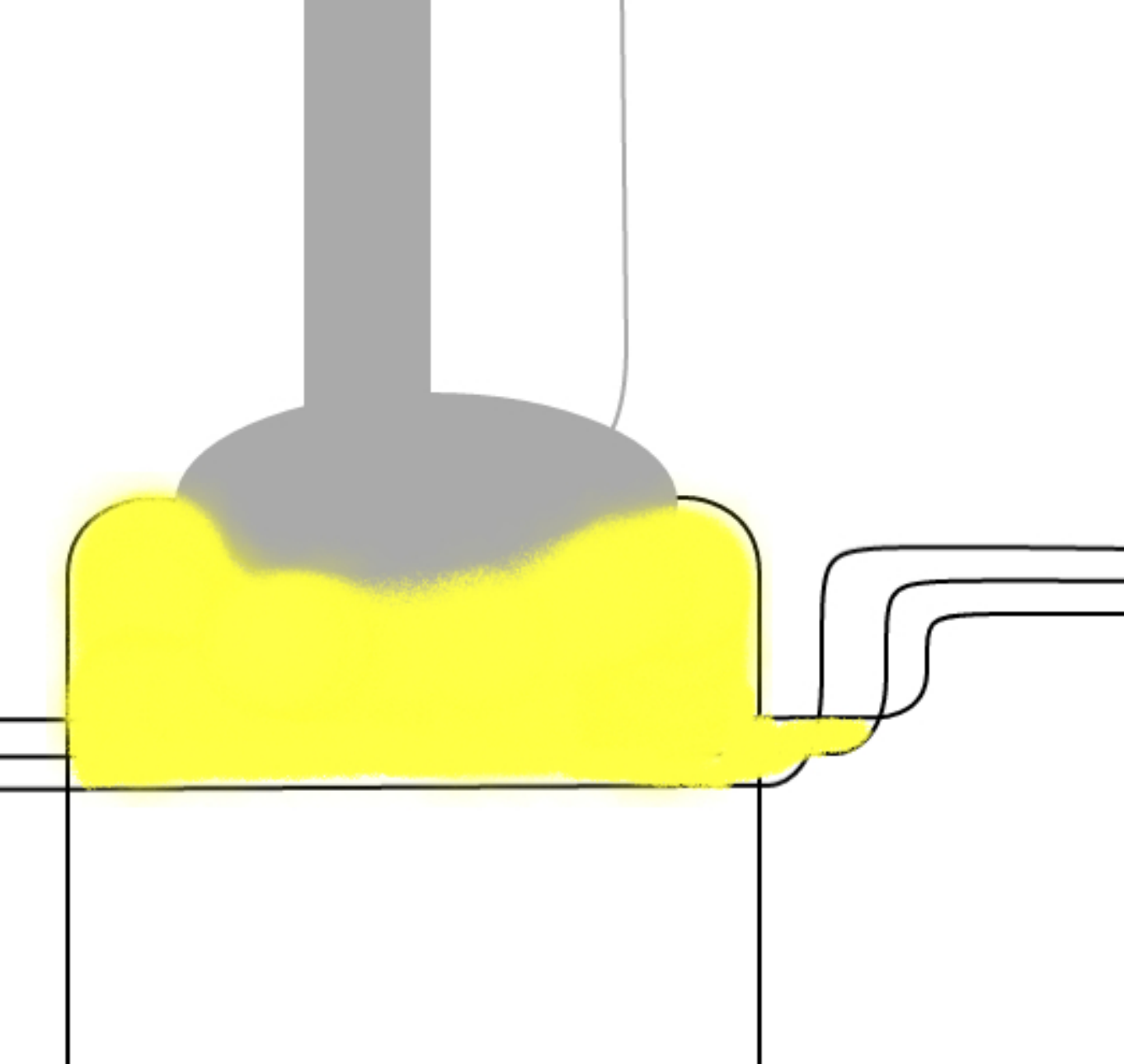}}
			%%Grid, for help.
%		 	,(2,1)*+{\bullet2},(3,1)*+{\bullet3},(4,1)*+{\bullet4},(5,1)*+{\bullet5},(6,1)*+{\bullet6},(7,1)*+{\bullet7},(1,1)*+{\bullet1},(1,2)*+{\bullet2},(1,3)*+{\bullet3},(1,4)*+{\bullet4},(1,5)*+{\bullet5},(1,6)*+{\bullet6},(1,7)*+{\bullet7}
				%%End Helping grid
			,(2.5,0)*+{\text{(b)}}
			\endxy			
			\qquad
			\xy
			\xyimport(8,8)(0,0){\includegraphics[width=1.5in]{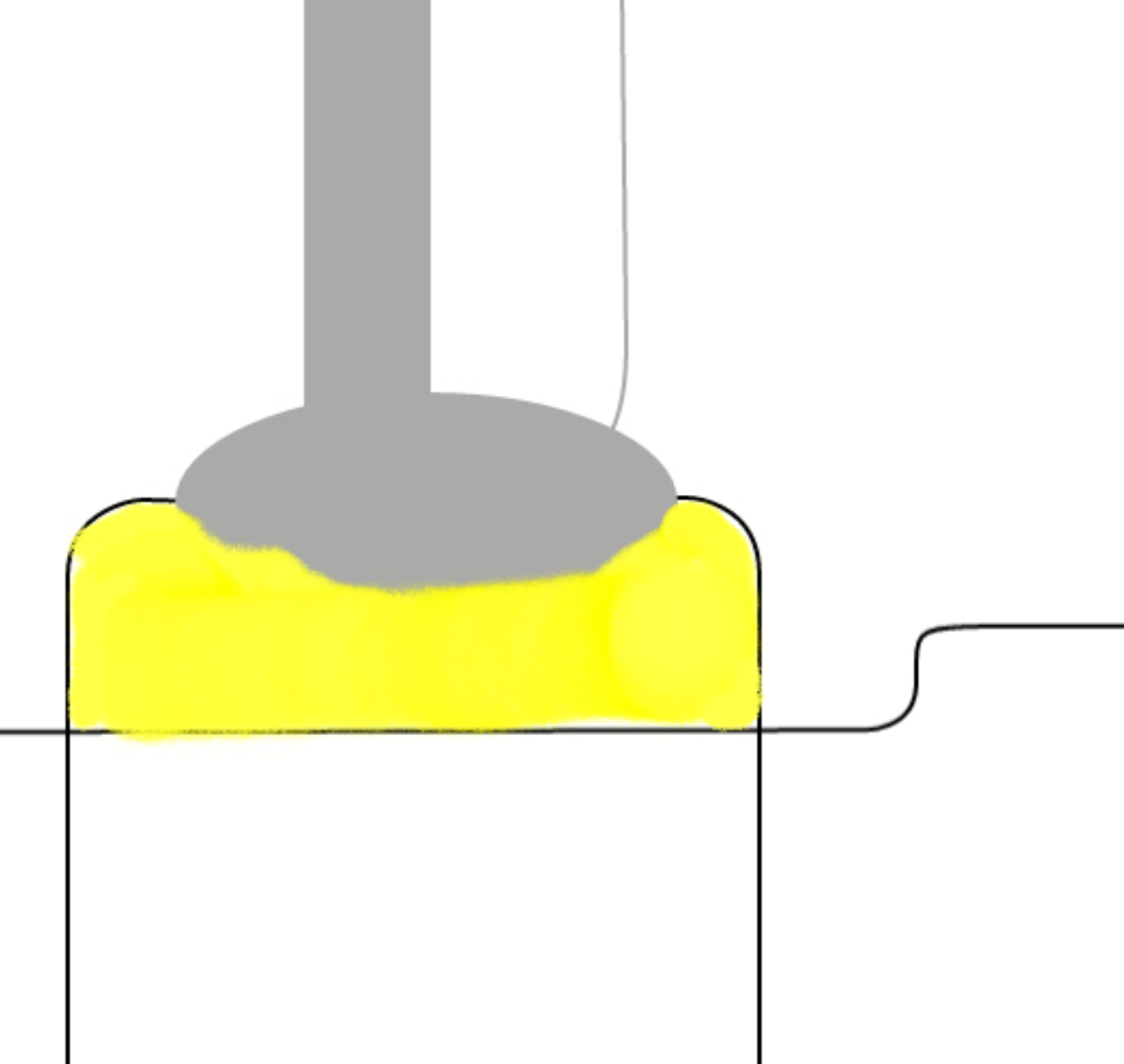}}
			%%Grid, for help.
%		 	,(2,1)*+{\bullet2},(3,1)*+{\bullet3},(4,1)*+{\bullet4},(5,1)*+{\bullet5},(6,1)*+{\bullet6},(7,1)*+{\bullet7},(1,1)*+{\bullet1},(1,2)*+{\bullet2},(1,3)*+{\bullet3},(1,4)*+{\bullet4},(1,5)*+{\bullet5},(1,6)*+{\bullet6},(1,7)*+{\bullet7}
				%%End Helping grid
			,(2.5,0)*+{\text{(c)}}
			\endxy			
			\]
\begin{image}\label{figure.cobordism-disk}
In (a) is depicted $B(Y)$, together with branes $X_i \times \gamma_i$, all projected to $T^*F$. Note that $B(Y)$ is obtained from the cobordism $Y$ by ``dragging the tails of $Y$'' toward $-\infty$ in the $F^\vee$ direction. In yellow in (b) we depict the image of a typical disk with boundary on the cobordism and on the $X_i \times \gamma_i$, again projected to $T^*F$. (c) is the special case when we test against a single $X \times \gamma$.
\end{image}
\end{figure}

\begin{example}
As an example when $k = 0$, one has a chain map
	\eqn\label{eqn.k=0}
		CW^*(X, Y_0) \to CW^*(X,Y_1)
	\eqnd
simply by counting disks from $X \cap Y_0$ to $X \cap Y_1$ as depicted in Figure~\ref{figure.cobordism-disk}(c). 
\end{example}

It is the content of Theorem~\ref{theorem.edges} that the operations in~(\ref{eqn.module-operations}) indeed define maps of $A_\infty$ modules.

\begin{remark}\label{remark.e}
The crux behind the success of this construction is being able to enumerate the generators of $CW^*(X \times \gamma, B(Y))$. (See also~(\ref{eqn.intersections-N}) below.) We explain this here, as this is the main reason we must consider different cobordism categories when pairing against different Fukaya categories.

When pairing against the entire wrapped category $\wrap$, we know that $\Lambda = M$ implies our cobordisms must look like Figure~\ref{figure.cobordisms}(b)---there are no parts of $B(Y)$ which go off to $-\infty$ in the $F^\vee$ direction except for the bits collared by $Y_0$ and $Y_1$. As a consequence, one can always choose $\gamma$ to have negative enough $F^\vee$ coordinate to guarantee
	\eqn\label{eqn.intersections}
		B(Y) \cap X \times \gamma = X \cap Y_0 \bigcup X \cap Y_1.
	\eqnd
These intersection points, of course, generate the chain complexes in~(\ref{eqn.k=0}). Equation (\ref{eqn.intersections-N}) below is a generalization to the higher $N$ case.

In the case where $\Lambda = \sk(M)$, we fix a real number $e>0$ once and for all. Then we only construct $\Xi$ on those $X$ which are within $e$ of $\sk(M)$, and on those cobordisms which go off to $-\infty$ in the $F^\vee$ direction in a way that avoids an $e$-neighborhood of $\sk(M)$. This guarantees that, once again, one can choose $\gamma$ such that the intersection set is exactly as in (\ref{eqn.intersections}). 
\end{remark}

\begin{remark}
Our restrictions to such $X$, and to such cobordisms, are of no consequence----every object of $\wrap_\compact$ is equivalent to such an $X$ by flowing via the Liouville flow. Likewise, every cobordism in $\lag_{\sk(M)}(M)$ is equivalent to a cobordism as described above.
\end{remark}

We have described how one takes morphisms in $\lag^{\dd 0}$ and creates morphisms in $\wrap\Mod$ (or $\wrap_\compact\Mod$). Now what does one do to higher morphisms? Let $Y \subset M \times T^*F^N$ be an $N$-morphism. Again, we perform a construction $B(Y)$, and count holomorphic disks with boundary on
	\eqn\label{eqn.N-boundaries}
		B(Y), \qquad
		X_0 \times \gamma_0^N,\qquad
		 \ldots,\qquad
		 X_k \times \gamma_k^N.
	\eqnd
We detail the $B$ construction in Section~\ref{section.B}.

This time, let us be more careful with grading. If the higher cobordism $Y$ is collared by objects $Y_0,\ldots,Y_N$ of $\lag^{\dd0}$, then one can compute that
	\eqn\label{eqn.intersections-N}
		CW^*(X \times \gamma^N, B(Y))
		\cong
		\bigoplus_{0 \in I \subset [N]} CW^*(X, Y_{\max I})[ |I|-1]
	\eqnd
as a graded abelian group. Here, $I$ runs through all subsets of $[N] = \{0,\ldots,N\}$ containing $0$. It is the content of Theorem~\ref{theorem.cubes} that the disk counts with boundary on (\ref{eqn.N-boundaries}) indeed yield higher homotopies between natural transformations. 

\begin{example}
Consider the case $N=2$ See Figure~\ref{figure.B-2-simplex},  where $B(Y)$ is pictured roughly by its projection to the zero section $F^2 \subset T^*F^2$. We assume $k=0$ for simplicity. Then the cochain complex $CF^*(X \times \gamma^2, B(Y))$ can be pictured as follows:
	\eqnn
		\xymatrix{
			CW^*(X, Y_2)[-1] \ar[rr]^{\id[-1]} 
				&& CW^*(X,Y_2)[-2] \\
			CW^*(X, Y_0) \ar[u]^{\Xi(Y_{02})[-1]} \ar[rr]^{\Xi(Y_{01})[-1]}
				\ar[urr]^{\Xi(Y)[-1]}
				&& CW^*(X, Y_1)[-1] \ar[u]_{\Xi(Y_{12})[-1]}
		}
	\eqnd
Here, $Y_{ij}: Y_i \to Y_j$ are the cobordisms collaring $Y$ along its faces. Not pictured are the self-differentials of the complexes $CW^*(X_i,Y)$. The fact that the Floer differential squares to zero precisely exhibits that $\Xi(Y)$ is a homotopy between
	\eqnn
		\Xi(Y_{02}) \qquad \text{and} \qquad \Xi(Y_{12}) \circ \Xi(Y_{01}).
	\eqnd
\end{example}

\subsection{$\Xi$ on $\lag^{\dd n}$, and stabilization}
Now we explain how we construct $\Xi$ on $\lag^{\dd n}$. 

By definition, an object of $\lag^{\dd n}$ is a brane $Y \subset M \times T^*E^n$ which, aside from the usual brane structures, must satisfy the following condition:
	\eqn\label{eqn.compact-image}
	\text{
	{\em When $Y$ is projected to $E^n$, it has compact image.}
	}
	\eqnd
For instance, a brane of the form $Y' \times E^n$ with $Y' \subset M$ is not allowed. (On the other hand, a brane of the form $Y' \times (E^n)^\vee$ is allowed.)

Moreover, any object of $\lag^{\dd n}$ can be isotoped by a compactly supported Hamiltonian on $M \times T^*E^n$ so that
	\eqn\label{eqn.transverse}
	\mbox{
	{\em $Y$ is transverse to $M \times E^n$.}
	}
	\eqnd
(and hence to $M \times (E^n + p)$ for any $p \in (E^n)^\vee$ small enough). So it suffices to define $\Xi$ on the full subcategory of $\lag^{\dd n}$ spanned by objects satisfying (\ref{eqn.transverse}).

To do this, one chooses curves $\beta_0,\ldots,\beta_k \subset T^*E$ which are very close to the zero section $E$, and we count holomorphic disks with boundary on
	\eqnn
		Y, X_0 \times \beta_0^n,\ldots, X_k \times \beta_k^n.
	\eqnd
See Figure~\ref{figure.betas}, where we depict the cases $k=0$ and $k >0$, with $n=1$. Note that if we choose $\beta_i$ to be close enough to the zero section $E^n$, the condition (\ref{eqn.transverse}) guarantees we do not have to perturb in the $T^*E^n$ direction to assure that $Y$ is transverse to each of the $X_i \times \beta_i$---only a perturbation in the $M$ component is necessary.

When computing the Hamiltonian chords between these branes, we will utilize a Hamiltonian which vanishes near the zero section of $T^*E^n$---this has the effect of essentially wrapping only in the $M$ components. 

Now, using the same constructions as for $\lag^{\dd 0}$ for cobordisms, one obtains a functor
	\eqnn
		\lag^{\dd n} \to \wrap\Mod
	\eqnd
for all $n$.

The choice of Hamiltonian chords and almost complex structures also explains why the diagram (\ref{eqn.stabilization-criterion}) commutes: If $Y \in \ob \lag^{\dd 0}$ and $X \in \ob \wrap$, one has a natural isomorphism
	\eqnn
		CW^*(X,Y) \cong CW^*(X \times \beta, Y \times E^\vee).
	\eqnd
And in fact, with $\beta_0,\ldots,\beta_k$ chosen appropriately, the $A_\infty$ operations 
	\eqnn
		CW^*(X_k,Y) \tensor \ldots \tensor CW^*(X_0,X_1) \to CW^*(X_0,Y)
	\eqnd
are identical to the stabilized operations
	\eqnn
		CW^*(X_k \times \beta_k, Y \times E^\vee) \tensor \ldots \tensor CW^*(X_0 \times \beta_0, X_1 \times \beta_1) \to CW^*(X_0 \times \beta_0, Y \times E^\vee).
	\eqnd
More generally, if $Y \in \ob \lag^{\dd n}$, the same argument goes through replacing the objects $X$ with $X \times \beta^n$. This is the content of Lemma~\ref{lemma.stabilization}.

\begin{figure}
		\[
			\xy
			\xyimport(8,8)(0,0){\includegraphics[width=1.5in]{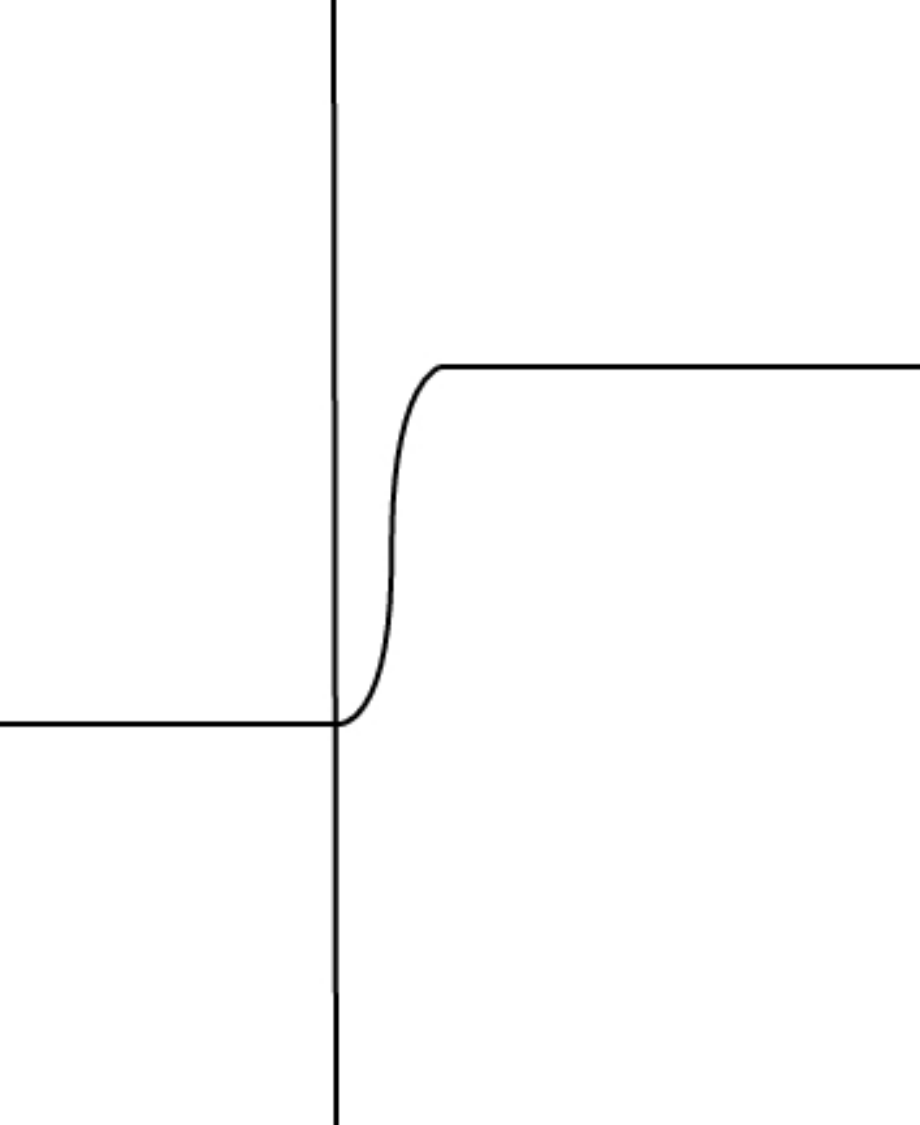}}
			%%Grid, for help.
%		 	,(2,1)*+{\bullet2},(3,1)*+{\bullet3},(4,1)*+{\bullet4},(5,1)*+{\bullet5},(6,1)*+{\bullet6},(7,1)*+{\bullet7},(1,1)*+{\bullet1},(1,2)*+{\bullet2},(1,3)*+{\bullet3},(1,4)*+{\bullet4},(1,5)*+{\bullet5},(1,6)*+{\bullet6},(1,7)*+{\bullet7}
				%%End Helping grid
			,(2,0)*+{\text{(a)}}
			\endxy
			\qquad
			\xy
			\xyimport(8,8)(0,0){\includegraphics[width=1.5in]{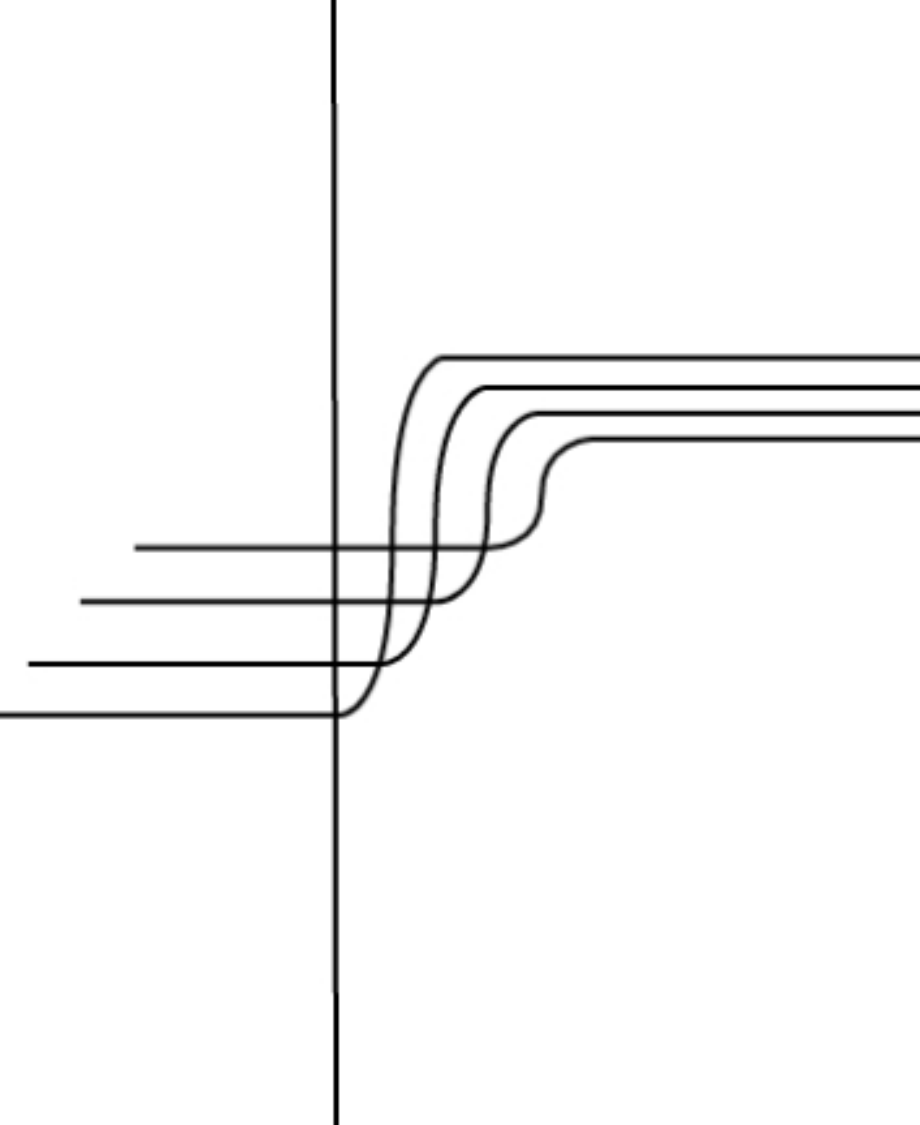}}
			%%Grid, for help.
%		 	,(2,1)*+{\bullet2},(3,1)*+{\bullet3},(4,1)*+{\bullet4},(5,1)*+{\bullet5},(6,1)*+{\bullet6},(7,1)*+{\bullet7},(1,1)*+{\bullet1},(1,2)*+{\bullet2},(1,3)*+{\bullet3},(1,4)*+{\bullet4},(1,5)*+{\bullet5},(1,6)*+{\bullet6},(1,7)*+{\bullet7}
				%%End Helping grid
			,(4,0)*+{\text{(b)}}
			\endxy					
			\]
\begin{image}\label{figure.betas}
In (a) is depicted a picture of $Y \times E^\vee$---a vertical line---and an image of $X \times \beta$. Though the scale is difficult to convey, one should imagine that $\delta$ is a curve very close to the zero section $E \subset T^*E$. Clearly, with appropriate choice of Floer data, $CW^*(X \times \beta, Y \times F^\vee$) recovers $CW^*(X, Y)$. In (b) is depicted $Y\times F^\vee$ along with a collection of $X_i \times \beta$. This configuration of $\beta$ ensures that disks with boundary on these branes precisely recovers the disks with boundary on the (unstabilized) collection of branes $X_0, \ldots, X_k, Y$. 
\end{image}
\end{figure}

\begin{remark}
Some care needs to be taken in choosing the $\beta_i$ and computing the $A_\infty$ operations when one wishes to work with a Novikov parameter---i.e., when one wants to take stock of the areas of holomorphic disks. Then, obviously, one could choose the curves $\beta$ to bound tiny disks, or less tiny disks. As usual, one needs negative exponentials of the Novikov variable to be able to show that the $A_\infty$ operations are equivalent when changing the areas of these disks. In this paper, we of course work in the exact setting, so none of this matters---we leave this remark just as an indication of what we need to deal with in future works.
\end{remark}

\begin{remark}
Philosophically, the functor $\Xi$ can be described as follows: For every $N$-simplex of $\lag$, which is represented by a Lagrangian submanifold $Y \subset M \times T^*E^n \times T^*F^N$, we construct a brane $B(Y)$ . Then we putatively take the composite functor
	\eqn
		\xymatrix{
			\fuk(M)^{\op} \ar[rr]^-{\times E^n \times F^N} 
				&&\wrap(M \times T^*E^n \times T^*F^N)^{\op} \ar[rr]^-{\hom(-,B(Y))}
				&&\chain_\ZZ.
		}	
	\eqnd
The claim is that analyzing this $\mu^d$ terms of this module gives rise to the higher natural transformations between the vertices of $Y$. One {\em outcome} of our main theorems is that the particular formulae for this natural transformation may differ depending on how one chooses wrapping functions and other parameters, but that they are all  equivalent in a coherent way. We use the adjective ``putative'' because, as mentioned in Remark~\ref{remark.wrapping-non-compact}, we do not define here a general notion of wrapped Floer cohomology in the setting of symplectic manifolds with non-compact skeleta. Rather, to prove the theorem itself, we construct a concrete model---using $\beta$ and $\gamma$---of what the above (putative) diagram should output, and we do not construct all the $A_\infty$ operations of the wrapped category of $M \times T^*\RR^{k+N}$. One advantage of constructing only the putative module maps is that we do not need to compute Floer cohomology between different {\em cobordisms}, only between the branes we test against a single cobordism. For example, one never needs to wrap a cobordism. This simplifies our exposition considerably. 

However, we do point out that by applying the $B$ construction to all cobordisms, one can easily set up a full Floer theory where one can wrap, and define the wrapped Floer cohomology between, cobordisms. This is because, once one applies the $B$ construction, every cobordism becomes a brane for which we have define wrapped Floer cohomology using the methods of this paper.
\end{remark}

\subsection{Transversality and compactness for moduli spaces, and comparing $M$ to $M \times T^*\RR$}\label{section.overview-transversality}
In this paper, the usual issues of finding a large enough family of perturbation data to achieve transversality, while ensuring Gromov compactness, arise. We must also choose these perturbation data in a way where we can compare the moduli of disks in $M \times T^*\RR$ to that in $M$. So issues similar to the work in~\cite{seidel-lefschetz-i} arise. We resolve all these issues as follows:
 
\enum
	\item
		To achieve compactness, we use $H$ and $J$ which are, outside some compact region $K \subset E^n \times F^N$, translation-invariant. Specifically, outside of $T^*(E^n \times F^N)|_K$, $H$ and $J$ are invariant under the translation action of $E^n \times F^N$ (we stress this is only in the horizontal direction). We also employ $J$ which are direct sum almost-compex structures in well-chosen regions. The asymptotic behavior of $H$ is up to the framework the reader wants: In one framework, $H$ is quadratic outside of some neighborhood of the skeleton with bounded distance from the skeleton (note that when the skeleton is non-compact, this does not imply that $H$ is quadratic away from a compact set). This is the framework introduced in~\cite{abouzaid-geometric}. Or, one can choose a sequence of linear $H_i$ with increasing slope, as in the colimit definition of wrapped cochains. This is the original framework of~\cite{abouzaid-seidel}, and for many applications, the most versatile.
		
		Usual convexity arguments prove that families of disks must remain close to the skeleton (i.e., compactness in the ``vertical'' direction), while translation-invariance allows us to apply an Uhlenbeck-style argument to ensure Gromov compactness in the $E^n \times F^N$ direction (i.e., in the ``horizontal'' direction).
	\item
		That these choices allow for transversality is an application of results from~\cite{abouzaid-seidel} which show that, if one knows a priori that there are regions of $M \times T^*E^n \times F^N$ that necessarily contain the images of the boundaries of holomorphic disks, then it suffices to choose perturbation data that have freedom in these regions to obtain tranversailty. These regions are contained in $T^*(E^n \times F^N)|_K$.
	\item
		Finally, we must compare disks in $M$ to those in $M \times T^*E^n \times T^*F^N$. We only do this for certain configurations of curves in $\RR^2$, whose moduli of holomorphic disks admit good geometric descriptions. (The ``staircase configurations'' which we define in Definition~\ref{defn.staircases}.) The hardest part of this is analyzing the moduli space for non-convex configurations of curves, but we choose a convenient configuration (called a staircase configuration) and use the coarse moduli space of the compactified stack of holomorphic disks to prove the geometric results we need.
\enumd

\subsection{We we can be degenerate about degeneracies}\label{section.degeneracy}
We detail the algebraic preliminaries in Section~\ref{section.algebra}, but we want to emphasize one technical point. Our main theorems assert that there exists a functor $\lag \to \fuk\Mod$ with a certain effect on objects. Recall that a functor between $\infty$-categories $\cC \to \cD$ is a map of simplicial sets. By definition, this means that we must assign simplices of $\cC$ to simplices of $\cD$ in such a way that all face {\em and} degeneracy maps are respected. 

What we verify in this paper are two things: (i) All face maps are respected, and (ii) the degeneracy map $s_0$ is respected at the level of objects. (One can prove that all the $s_0$ maps are respected, for any dimension of simplex; but in fact, it seems that the other degeneracy maps are not respected.)

So why is this enough? There are philosophical reasons, then there is the technical proof of it. We refer the reader to~\cite{tanaka-non-strict} for both. The main result of loc. cit. states that so long as (i) and (ii) above are satisfied, $\Xi$ indeed defines a functor of $\infty$-categories. The only caveat is that one may have to modify $\Xi$ to a new assignment $\Xi'$, but such a modification exists functorially, and for any simplex $A$ of $\cC$, $\Xi(A)$ is always homotopic to $\Xi'(A)$, so no computation or content is lost. Further, one can always choose $\Xi'$ to equal $\Xi$ on objects, and the main theorems follow.

% Activate the following line by filling in the right side. If for example the name of the root file is Main.tex, write
% "...root = Main.tex" if the chapter file is in the same directory, and "...root = ../Main.tex" if the chapter is in a subdirectory.
 
%!TEX root = _pairing.tex

\section{Algebraic preliminaries}\label{section.algebra}
We assume that the reader is familiar with $A_\infty$-categories and $\infty$-categories. As usual, $[n]$ stands for the poset $\{0 < 1 < \ldots < n\}$.
We fix a base ring $k$, which will usually be $\ZZ$.

Let us clarify the statement of the main theorems. $\fukaya\Mod$ is a dg category. There is a standard way to render it an $\infty$-category: by taking its dg nerve $N(\fukaya\Mod)$. So by a functor from $\lag$ to $\fukaya\Mod$, we really mean a map of simplicial sets:
	\eqnn
	\Xi:
	\lag \to N(\fukaya\Mod).
	\eqnd
One upshot of this section is to arrive at Lemma~\ref{lemma.goal}. This Lemma states the formulas that $\Xi$ must satisfy to respect face maps.

\subsection{The simplicial nerve for modules}
\label{section.nerve}

\subsubsection{The dg nerve}
Recall that in~\cite{higher-algebra}, Lurie defines an $\infty$-category out of any dg category. While he uses homological grading, the same formula produces an $\infty$-category for a dg category with cohomological grading (i.e., when the differential has degree $+1$). We recall the construction here.

Throughout, $|K|$ denotes the cardinality of a finite set $K$, while $|f|$ denotes the degree of an element $f$ of a cochain complex.

\begin{defn}[The dg nerve]\label{defn.dg-nerve}
Let $C$ be a small dg-category, cohomologically graded. Then the {\em dg nerve} of $C$ is a simplicial set $N(C)$ defined as follows: An $N$-simplex is given by a collection
	\eqnn
		\left( (X_i)_{0 \leq i \leq N} , \{f_K\}_{K \subset [N], |K| \geq 2} \right).
	\eqnd
Here,
	\enum
		\item
			$(X_i) = (X_0,\ldots,X_N)$ is a sequence of objects of $C$,
		\item
			$K$ runs through all subsets of $[N] = \{0,1,\ldots,N\}$ of cardinality at least 2, and 
		\item
			Each $f_K$ is an element of the abelian group $\hom_C(X_{K_{\min}}, X_{K_{\max}})$ of degree $2-|K|$.
	\enumd
This data must satisfy the following condition:
	\begin{itemize}
		\item
			For each $K \subset [N]$ with $|K| \geq 2$, writing $K = \{i_0 < \ldots < i_{|K|-1}\}$, we have that
    	\eqn\label{eqn.dg-nerve}
    	d f_K
    			= 
    			\sum_{1 \leq j \leq |K|-2} (-1)^j f_{K \setminus \{i_j\}}
    			+ \sum_{K = K' \wedge K''} (-1)^{|K'|} f_{K''} \circ f_{K'}.
    	\eqnd
	\end{itemize}
Here, the notation $K = K' \wedge K''$ is shorthand---the summation is over all decompositions of $K$ into two subsets $K'$ and $K''$ such that (a) $K' \cap K'' = \{i\}$ is a set of exactly one element, and (b) $K''$ equals the set of all elements of $K$ greater than or equal to $i$, while $K'$ equals the set of all elements of $K$ less than or equal to $i$.

If $\alpha: [N] \to [N']$ is a morphism of posets, then the induced function $N(C)_{N'} \to N(C)_N$ takes an $N'$-simplex $( (X_i), (f_K) )$ to the $N$-simplex given by
	\eqnn
		( (X_{\alpha(i)})_{0 \leq i \leq N}, \{g_J\})
	\eqnd
where
	\eqnn
		g_J = 
			\begin{cases}
			f_{\alpha(J)} & \text{if $\alpha|_J$ is injective} \\
			\id_{X_j} & \text{if $J=\{i,i'\}$ with $\alpha(i)=\alpha(i')=j$} \\
			0 & \text{otherwise.}
			\end{cases}
	\eqnd
\end{defn}

\begin{remark}
The nerve construction was generalized to $A_\infty$ categories in our previous work~\cite{tanaka-thesis} and independently in Faonte's work~\cite{faonte,faonte-2}. These latter two sources in particular give a satisfying framework for nerves of $A_\infty$-categories. Discussions of units appear in~\cite{tanaka-thesis}. As one expects, the nerve has a characterization that is more transparent: The $N$-simplices of $N(C)$ consist of all $A_\infty$-functors from $k[\Delta^N]$ to $C$, where $k[\Delta^N]$ is a cofibrant $A_\infty$ category over $k$ modeling the $N$-simplex.
\end{remark}

\begin{example}
A 0-simplex of $N(C)$ is simply an object of $C$. A 1-simplex is a choice of two objects $X_0, X_1$, and a degree 0, closed morphism from $X_0$ to $X_1$. A 2-simplex is a choice of morphisms $f_{ij} : X_i \to X_j$, and a choice of a degree -1 element realizing a homotopy from $f_{02}$ to the composition $f_{12} \circ f_{01}$.
\end{example}

\subsubsection{The category of modules}
For any $A_\infty$ category $\cA$, the module category
	\eqnn
	\cA\Mod := \fun_{A_\infty}(\cA^{\op},\chain_\ZZ)
	\eqnd
is an $A_\infty$-category with $\mu^{k \geq 3} = 0$. We recall the definition here, which the reader may also find in~\cite{seidel-book}. 
In what follows, we abbreviate some signs by the symbols
	\eqnn
	\spadesuit_c = |a_1| + \ldots + |a_c| + c
	\eqnd
and
	\eqnn
	\heartsuit = |a_{c+1}| + \ldots + |a_{d-1}| + |x| - d + c + 1.
	\eqnd

{\bf Objects of $\cA\Mod$.} An object is a collection of data
	\eqnn
		( \cM(X)_{X \in \ob \cA}, (\mu_{\cM}^d)_{d \geq 1} )
	\eqnd
where each $\cM(X)$ is a graded vector space, and $\mu_{\cM}^d$ is a map
	\eqnn
	\cM(X_{d-1}) \tensor \cA(X_{d-2},X_{d-1}) \tensor \ldots \tensor \cA(X_0,X_1) \to \cM(X_0)
	\eqnd
of degree $2-d$.
If $\mu^d_\cA$ denotes the $A_\infty$ operations of $\cA$, these $\mu_{\cM}^d$ must satisfy
	\begin{align}
	0
	=&
	\sum_{b+c=d} (-1)^{\spadesuit_c}
		\mu_{\cM}^{1+c}(\mu_{\cM}^{b} (x, a_{d-1} , \ldots, a_{c+1}),a_c, \ldots, a_1) \nonumber
	\\
	&+ \sum_{\substack{a+b+c=d,\\ a>0}} (-1)^{\spadesuit_c} \mu_{\cM}^{a+1+c}(x,a_{d-1},\ldots,a_{b+c+1}, \mu^b_\cA(a_{b+c},\ldots,a_{c+1}),\ldots,a_1). \nonumber
	\end{align}
Given two objects $\cM_0, \cM_1$, a degree $|t|$ element of the graded abelian group $\cA\Mod(\cM_0,\cM_1)$ is a collection $t = (t^d)_{d \geq 1}$, where each $t^d$ is a map
	\eqnn
		t^d :
		\cM_0(X_{d-1}) \tensor \cA(X_{d-2},X_{d-1}) \tensor \ldots \tensor \cA(X_0,X_1)
		\to
		\cM_1(X_0)
	\eqnd
of degree $|t|-d+1$.

{\bf The operation $\mu^1$.} We let $(\mu^1 t)^d$ denote the $d$th component of $\mu^1 t$. It is given by
	\begin{align}
		(\mu^1 t)^d &(x, a_{d-1},\ldots,a_1)
			 \nonumber \\
			&=\sum_{b+c=d} (-1)^\heartsuit \mu_{\cM_1}^{1+c}(t^{b}(x, a_{d-1},\ldots,a_{c+1}), a_c, \ldots, a_1) \nonumber \\
			&+ \sum_{b+c=d} (-1)^\heartsuit t^{1+c}(\mu^{b}_{\cM_0} (x,a_{d-1},\ldots, a_{c+1}), a_c,\ldots,a_1)  \nonumber \\
			&+ \sum_{\substack{a+b+c=d,\\ a>0}} (-1)^\heartsuit t^{a+1+c} (x,a_{d-1},\ldots,\mu^b_\cA(a_{b+c},\ldots,a_{c+1}),a_c,\ldots,a_1)  . \label{eqn.mu1}
	\end{align}

{\bf The operation $\mu^2$.} Likewise, $(\mu^2(t_2, t_1))^d$ denotes the $d$th component. It is given by
	\begin{align}
		(\mu^2(t_2, t_1))^d&(x, a_{d-1},\ldots,a_1) \nonumber \\
			=& \sum_{b+c=d} (-1)^\heartsuit t_2^{1+c} (t_1^{b}(x,a_{d-1},\ldots,a_{c+1}),a_c,\ldots,a_1).
	\end{align}

The following is a straightforward outcome of the above definitions. The notation is meant to match the notation we'll use in proving Theorem~\ref{theorem.cubes}.

\begin{lemma}\label{lemma.goal}
If we have a collection $\Xi(Y_i)$ of modules for $\cA$, and for each $K \subset [N]$ we have an element
	\eqnn
	\Xi_{Y_K} \in \hom_{\cA\Mod}^{|K|-1}(\Xi(Y_{K_{\min}}), \Xi(Y_{K_{\max}}))
	\eqnd
then the collection $( (\Xi(Y_i)), (\Xi_{Y_K}) )$ is an $N$-simplex of $N(\cA\Mod)$ if and only if these satisfy the equation
	\begin{align}
			\sum_{b+c=d} &
				(-1)^\heartsuit \mu_{\Xi(Y_N)}^{1+c}(\Xi_Y^{b}(x, a_{d-1},\ldots,a_{c+1}), a_c, \ldots, a_1) \nonumber \\
			&+ \sum_{b+c=d} (-1)^\heartsuit \Xi_Y^{1+c}(\mu^{b}_{\Xi(Y_0)} (x,a_{d-1},\ldots, a_{c+1}), a_c,\ldots,a_1)  \nonumber \\
			&+ \sum_{\substack{a+b+c=d,\\ a>0}} (-1)^\heartsuit \Xi_Y^{a+1+c} (x,a_{d-1},\ldots,\mu^b_\cA(a_{b+c},\ldots,a_{c+1}),a_c,\ldots,a_1) \nonumber \\
			= &(-1)^{N+1} 
			[				 (-1)^j \Xi_{Y_{[N] \setminus j}}^d(x,a_{d-1},\ldots,a_1)\nonumber \\
				&+\sum_{[N] = K' \wedge K} \sum_{b+c=d} (-1)^\heartsuit \Xi_{Y_{K'}}^{1+c} (\Xi_{Y_{K}}^{b}(x,a_{d-1},\ldots,a_{c+1}),a_c,\ldots,a_1)
				]
				.
	\end{align}
Or, without using elements, and changing indexing, we must have
	\begin{align}
			0=&\sum_{b+c=d}
				(-1)^\clubsuit \mu_{\Xi(Y_N)}^{1+c}(\Xi_Y^{b}\tensor\id^{\tensor c}) \nonumber \\
			&+ \sum_{b+c=d} (-1)^\clubsuit \Xi_Y^{1+c}(\mu^{b}_{\Xi(Y_0)} \tensor \id^{\tensor c})  \nonumber \\
			&+ \sum_{\substack{a>0 \\ a+b+c = d} } 
				(-1)^\clubsuit \Xi_Y^{a+1+c} (\id^{\tensor a} \tensor \mu^b_\cA\tensor \id^{\tensor c}) \nonumber \\
			&+  \sum_{0<j<N} (-1)^\clubsuit \Xi_{Y_{[N] \setminus j}}^d\nonumber \\
				&+\sum_{[N] = K' \wedge K} \sum_{b+c=d} (-1)^\clubsuit \Xi_{Y_{K'}}^{1+c} (\Xi_{Y_{K}}^{b}\tensor\id^{\tensor c})]
				.\label{eqn.goal}
	\end{align}
\end{lemma}

\begin{remark}
The signs $\clubsuit$ can be determined by the usual translation between ``functions evaluated on elements,'' and ``functions'' using the Koszul sign rule, but we do not write them out explicitly here. As we will see, all our proofs have signs induced by tautological Yoneda modules, so one can rest assured that the signs work out fine.
\end{remark}

% Activate the following line by filling in the right side. If for example the name of the root file is Main.tex, write
% "...root = Main.tex" if the chapter file is in the same directory, and "...root = ../Main.tex" if the chapter is in a subdirectory.
 
%!TEX root = _pairing.tex

\section{Geometric setup}\label{section.geometry}
Most of this section is standard, but we establish a trick we will use over and over again, called boundary-stripping, used first in the proof of Lemma~\ref{lemma.q>w}. This is the main tool we use to reduce our computations to combinatorics, and its validity relies heavily on the properties of analytic maps to $\CC$ (namely, the open mapping theorem).

\subsection{Pseudoholomorphic disks}

\begin{notation}
As usual, $S$ will denote the closed unit disk in $\CC$ missing $d+1$ boundary punctures for $d \geq 1$, with the induced conformal structure. At times, we will conflate $S$ with its conformal equivalence class. 

We will also let $\cR$ denote the moduli space of conformal structures on a disk with $(d+1)$ boundary points. (The dependence on $d$ is implicit in the notation, or will be made explicit as necessary.) When $d=1$, we will set $\cR = pt$ to be a single point (i.e., we take the coarse moduli space, rather than the stack $B\RR$). 
\end{notation}

We study smooth maps $u: S \to M \times T^*\RR$ (or $S \to M$ or $S \to T^*\RR$) satisfying Floer's equation
	\eqn\label{eqn.floer}
	(du - X_H \tensor \alpha)^{0,1} = 0
	\eqnd
whose dimension-zero counts give rise to the $A_\infty$ operations. 

Giving precise meaning to this equation amounts to specifying the kinds of $S$-dependent almost complex structures $J$ and (possibly $S$-dependent) Hamiltonian functions on $M \times T^*\RR$ or $M$ or $T^*\RR$ we utilize. As usual all this is coupled with a sub-closed one-form $\alpha$ on $S$, and chosen compatibly with gluing data specified by strip-like ends; on such ends, all the data becomes invariant under the translation action of strips. Details can be found in~\cite{abouzaid-seidel}. For now, recall that maps $u$ satisfying (\ref{eqn.floer}) extend to a continuous map on the obvious compactification $\ov S$, where one glues in a closed interval in place of the missing boundary punctures of $S$.

\subsection{Geometry of $M$}\label{section.M}
The interested reader can find more details in~\cite{abouzaid-seidel,abouzaid-geometric, ganatra, sylvan-thesis}.
	\begin{itemize}
		\item
			$M$ carries a 1-form $\theta_M$ whose deRham derivative $d\theta_M = \omega_M$ is symplectic. The Liouville vector field $X_\theta$, defined by 
				\eqn\label{eqn.liouville}
				\omega(X_\theta,-) = \theta
				\eqnd 
			must define a decomposition
				\eqn\label{eqn.decomposition}
				M = A \bigcup_{\del M} \del M \times [0,\infty) 
				\eqnd
			where $A \subset M$ is some compact subset, $\del M \subset M$ is some non-empty, codimension 1 submanifold, and $[0,\infty)$ parametrizes the flow of $\del M$ along $X_\theta$. This decomposition is required to make $(\del M,\theta_M|_{\del M})$ a contact manifold. Further, if $r$ parametrizes $[0,\infty)$, then $d\theta_M$ must agree with $\theta_M|_{\del M} \wedge e^rdr $ on $\del M \times [0,\infty)$; that is, this decomposition must realize $M \setminus A$ as a symplectization of $\del M$. 
			
			As an example, if $M = T^*Q$ with $\theta_M = pdq$, $A$ can be taken to be the unit cotangent disk bundle for some Riemannian metric on $M$. Then $\del M$ is the unit sphere bundle of the cotangent bundle, and the Liouville flow is proportional to dilating the cotangent bundle using the real vector space structure of the fibers.\footnote{Note if $Q$ is not compact, neither is $A$---so $T^*Q$ would not, strictly speaking, fit into the framework. This is why we take extra care in this paper with $T^*E^n \times T^*F^N$.}
			
		\item
			The flow of $-X_\theta$ collapses $M$ (after infinite time) to some limiting subset $\sk(M) \subset M$. We call this the skeleton of $M$. As an example, if $M = T^*Q$, $\sk(M)$ is the zero section.
		\item
			We have that $2c_1(TM) = 0$ for any choice of almost complex-structure as in Section~\ref{section.J}. This ``Calabi-Yau'' condition will allow us to define a $\ZZ$-grading on graded branes. We fix once and for all a trivialization
				\eqn\label{eqn.TM-trivialization}
				{\det}^2_\CC(TM) \cong M \times S^1.
				\eqnd
	\end{itemize}

\begin{remark}
Analytically, the role of (\ref{eqn.decomposition}) is to make $M$ convex in the following sense: Any holomorphic disk escaping $A$ must be contained entirely in some level set $\del M \times \{r\}$. See Section~\ref{section.M-analysis}.

On the other hand, the decomposition also serves a non-analytic role when defining Lagrangian cobordisms---since non-compact submanifolds can behave wildly, the decomposition allows us to define a broad and well-behaved class of submanifolds (eventually conical submanifolds) to use as our objects and morphisms as we will see in Section~\ref{section.branes}.
\end{remark}

\subsubsection{Branes}\label{section.branes}
By a {\em brane}, we will mean the tuple of data
	\eqn
	 (L, \alpha, f, P)
	\eqnd
as follows:
	\begin{itemize}
		\item
			$L$ is an embedded Lagrangian submanifold of $M$. Further, $L$ must eventually be conical, meaning that outside a compact set, $L$ must be invariant under the Liouville flow.  Finally, the image of $L$'s time-infinity negative Liouville flow (a subset of $\sk(M)$) must have compact closure.\footnote{This implies the condition ({eqn.compact-image}) we mentioned in Section~\ref{section.outline}.} A non-example is the zero section of $T^*\RR$, while an example is a cotangent fiber. 
		\item	
			$\alpha: L \to \RR$ is a lift of the usual phase map $L \to S^1$. Recall that this phase map depends on the trivialization~(\ref{eqn.TM-trivialization}) and is defined by the following diagram:
				\eqnn
					\xymatrix{
					 && GrLag(TM) \ar[rr]^-{\det^2 \, (\ref{eqn.TM-trivialization})} \ar[d] && S^1 \\
					L \ar[urr]^-{\text{tangent space}} \ar[rr] && M
					}
				\eqnd
			$GrLag(TM)$ is the Grassmanian of Lagrangians, and the map to $S^1$ is induced fiberwise by the map $\det^2: U(n)/O(n) \to S^1$. The choice of $\alpha$, analytically, allows us to define $\ZZ$-graded Floer cochain complexes. In terms of cobordisms, $\alpha$ guarantees that $\lag(M)$ is a stable $\infty$-category in which the shift functor is not periodic. (Equipping each $L$ with a $d$-fold cover of $S^1$ would result in a $d$-periodic stable $\infty$-category.)
		\item
			$f: L \to \RR$ is a smooth function such that $df = \theta_M|_L$, called a {\em primitive}. This implies that $f$ is locally constant outside a compact set. We will always assume that this locally constant value is 0.\footnote{One can always find an eventually linear Hamiltonian isotopy that moves $L$ to an $L'$ allowing for such a primitive.}
		\item
			$P$ is a relative Spin structure on $L$. This allows us to orient the moduli spaces of holomorphic disks, hence to define Floer cochains with $\ZZ$ coefficients. For cobordisms, this serves the role of changing the coefficient ring spectrum over which $\lag$ is linear.
	\end{itemize}

\subsubsection{Quadratic and linear Hamiltonians}
We will say that a Hamiltonian $H: M \to \RR$ is {\em eventually quadratic} if, for some $r_0 > 0$, the restriction of $H$ to $\del M \times [r_0,\infty)$ is proportional to $r^2$. Put in a less coordinate-dependent way, this means
	\eqn\label{eqn.quadratic}
	dH(X_\theta) = 2H.
	\eqnd
We fix an eventually quadratic Hamiltonian such that $H>0$ once and for all.

We will say $H$ is {\em eventually linear} if $H$ is proportional to $r$ on $\del M \times [r_0, \infty)$ for some $r_0 > 0$. This means that, on this region, $dH(X_\theta) = f$ for some smooth function $f: \del M \to \RR$. $H$ need not be positive.

Any flow by an eventually linear Hamiltonian gives rise to an equivalence in $\lag_{\sk(M)}$.  An example of an eventually linear Hamiltonian on $T^*Q$ is given by any vector field on a smooth manifold $Q$: The vector field defines a Hamiltonian on $T^*Q$ by pairing with covectors, and the projection map $T^*Q \to Q$ intertwines the Hamiltonian flow with the flow on $Q$.

\subsubsection{Almost complex structures}\label{section.J}
We will only consider almost complex structures which are compatible with $\omega$ and which are of contact type---that is,
	\eqn\label{eqn.contact-type}
	\theta \circ J = dr
	\eqnd
on $\del M \times [r_0',\infty)$ for some $r_0' >0$.

\subsubsection{Compactness and regularity for $M$.}\label{section.M-analysis}
The results below rely on the following:
	\begin{itemize}
		\item $M$ is exact and modeled on a symplectization outside a compact set,
		\item $H$ is eventually quadratic with respect to $r$, and $J$ is eventually of contact type,
		\item Each $L_i$ is exact and eventually conical, and 
		\item the primitives $f_i:L_i \to \RR$ can be chosen to equal zero outside a compact set.	
	\end{itemize}
	
\begin{proof}[Proof of Gromov compactness]
These assumptions combine to show that, by the maximum principle, if any pseudoholomorphic disk intersects $\del M \times \{r\}$, then it must be contained entirely in $\del M \times \{r\}$. Hence disks must remain below some height $r$ once boundary conditions are fixed. See for instance Section 7 of~\cite{abouzaid-seidel}. If in addition the interior of $M$ is compact, this bound is sufficient to prove Gromov compactness---i.e., to prove that the compactification of the moduli space of disks has the usual boundary strata necessary for the $A_\infty$ relations to hold. 
\end{proof}

As for regularity: Given the assumptions on $\gamma, H$, and $J$, one can prove various properties about the maps $u$ satisfying~(\ref{eqn.floer}) conforming to the ``similarity'' principle. (I.e., That maps $u$ satisfy properties similar to honest holomorphic maps.) For instance:

\enum
	\item
		Two distinct solutions can only agree on a discrete set of points. (Lemma 7.1 of~\cite{abouzaid-seidel}.)
	\item\label{item.discrete}
		If the Hamiltonian chords of $H$ are non-degenerate, there exists an open subset $U \subset S$ containing $\del S$ and its marked points at infinity on which the locus where
			$
			du = X \tensor \alpha
			$
		is a discrete subset of $U$. (Lemma 8.6 of~\cite{abouzaid-seidel}.).
\enumd

\begin{remark}
We say that $u$ is {\em trivial} at $z \in S$ if $du=X \tensor \alpha$ at $z$. We call $u$ itself trivial if $u$ is trivial on all of $S$. One can in fact show that $u$ is trivial only when (i) $d \gamma \neq 0$ and $u$ is constant, or (ii) $d \gamma = 0$ and $S \cong \RR \times [0,1]$ is the strip. These facts ultimately imply (\ref{item.discrete}) above.
\end{remark}

As usual, (\ref{item.discrete}) above implies regularity for generic perturbation data:

\begin{proof}[Proof of regularity for disks in $M$.]
Let $D$ be the linearized operator and $D^\ast$ its adjoint. Let $T$ be a purported non-trivial element of $\ker D^\ast$.  One finds a $z \in U$ for which $u$ is not trivial, and where $T$ is also non-zero---this is possible thanks to (\ref{item.discrete}). Then a well-chosen tangential perturbation $Z \in W^{1,p}$, supported near $z$, contradicts the non-triviality of $T$ as it must pair non-trivially with $T$, but $T$ was supposed to be in the kernel. In other words, $\ker D^\ast$ is zero; hence so is $\coker D$.
\end{proof}

\begin{remark}\label{remark.Z-small}
A key feature of this proof is that the set of possible choices of $Z$ can be restricted to be those $Z$ that are supported in some neighborhood of the boundary branes $L_i$ and the limiting Hamiltonian chords $x_i$---this is because the $z$ in the proof above can be chosen to be in an arbitrarily small neighborhood $U$ of $\del \ov S$, hence one can a priori fix small neighborhoods of the a priori prescribed boundary conditions of $u$, and utilize any $Z$ with support in these neighborhoods to have the proof go through without a hitch.
\end{remark}

\subsection{Geometry of $T^*\RR$}\label{section.T*R}
Here we introduce the key players (staircases) which will allow us to build maps and homotopies between bimodules out of cobordisms. 
The key lemma of the present section---Section~\ref{section.T*R}---is Lemma~\ref{lemma.fibers}, which will later allow us to relate disks in $M \times T^*\RR$ with those in $M$ itself.

We use $q$ to denote the coordinate of the zero section, and $p$ to denote the coordinate of the usual trivialization of $T^*\RR$ using the 1-form $dq$. We fix the 1-form $\theta_{T^*\RR} := pdq$ on $T^*\RR$. This is not a generic choice, as evidenced by the $\RR$-symmetry. The skeleton $\RR \subset T^*\RR$ is also not compact.

By and large, we will use the following structures on $T^*\RR$:

	\begin{itemize}
		\item
			We will always consider almost complex-structures on $T^*\RR$ which, along some specified region, agree with the standard structure
	\eqn \label{eqn.i}
		J_{T^*\RR}: \del_p \mapsto \del_q,
		\qquad
		\del_q \mapsto -\del_p.
	\eqnd
(We caution the reader that the coordinate transformation $x \mapsto q, y \mapsto p$ results in the opposite complex structure from that on $\CC$; this is the usual incompatibility between $T^*\RR$ and $\CC$.) This $J$ is also not of contact type; it does not satisfy (\ref{eqn.contact-type}). So this $J_{T^*\RR}$ will only be used within a bounded distance of the zero section of $\RR$.
	\item
		We choose $H$ such that $H(q,p) = H(p)$ only depends on the $p$ coordinate, is strictly positive, and equals $p^2$ outside some bounded interval in the $p$ coordinate. This translation invariance in the $q$ direction will allow us to apply the standard rescaling trick for Gromov compactness (zooming into produce non-constant holomorphic map to $\CC$ from $\CC P^1$). 
	\end{itemize}

\subsubsection{Tunnels}
We will make use of a particular class of curves $\gamma_i \subset T^*F$ which we call tunnel curves. 

To begin, we fix numbers $w>0$ and $\un{D}>\ov{D} >0$. $w$ is the width of a cobordism, and $\ov{D}$ is the depth of a cobordism. (See Definitions~\ref{defn.depth} and~\ref{defn.width}.)

We then choose a decreasing sequence of real numbers
	\eqnn
	(w_i)_{i \geq 0}^\infty, 
	\qquad
	w_i > w
	\eqnd
and an increasing sequence of real numbers
	\eqnn
	(D_i)_{i \geq 0}^\infty,
	\qquad
	\un{D} > D_i > \ov{D},
	\qquad
	(h_i)_{i \geq 0}^\infty, 
	\qquad
	h_i > -D_i,
	\qquad
	h:= \sup_i h_i < \infty.
	\eqnd
And we finally choose real numbers $\epsilon_i$ such that
	\eqnn
		0 < \epsilon_i < w_i - w_{i+1}.
	\eqnd

\begin{remark}
The $w_i$ stand for the width of a curve, the $D_i$ stand for the depth, and the $h_i$ stand for the height.
\end{remark}

\begin{defn}\label{defn.staircases}
A {\em tunnel curve} of depth $D_i$ and width $w_i$ is an embedded curve
	\eqnn
	\gamma_i \subset T^*F
	\eqnd
satisfying the following:
	\begin{itemize}
		\item Outside the box
			\eqnn
				[w_i - \epsilon_i, w_i] \times [-D_i, h_i]
			\eqnd
		$\gamma_i$ equals the set
			\eqnn
				(-\infty, w_i-\epsilon] \times \{-D_i\}
				\coprod
				[w_i,\infty) \times \{h_i\}.
			\eqnd
		Inside the box, we demand that $\gamma_i$ be equal to the image of some curve $\RR \to T^*F$ which has strictly positive derivative in the $F^\vee$ component. 
	\end{itemize}
One can give each $\gamma_i$ a standard brane structure, but we make no use of a primitive function $f$ on $\gamma_i$.
 
Given $(w_i)$, $(D_i)$, and $(h_i$ as above, we call the corresponding collection $(\gamma_i)_{i=0}^\infty$, or a finite collection $\gamma_0,\ldots,\gamma_{d-1}$,  a {\em staircase}.
\end{defn}

Note $\gamma_i$ and $\gamma_j$ have a unique intersection point for $i \neq j$. See Figure~\ref{figure.gamma-boxes}.

\begin{figure}
		\[
			\xy
			\xyimport(8,8)(0,0){\includegraphics[width=5in]{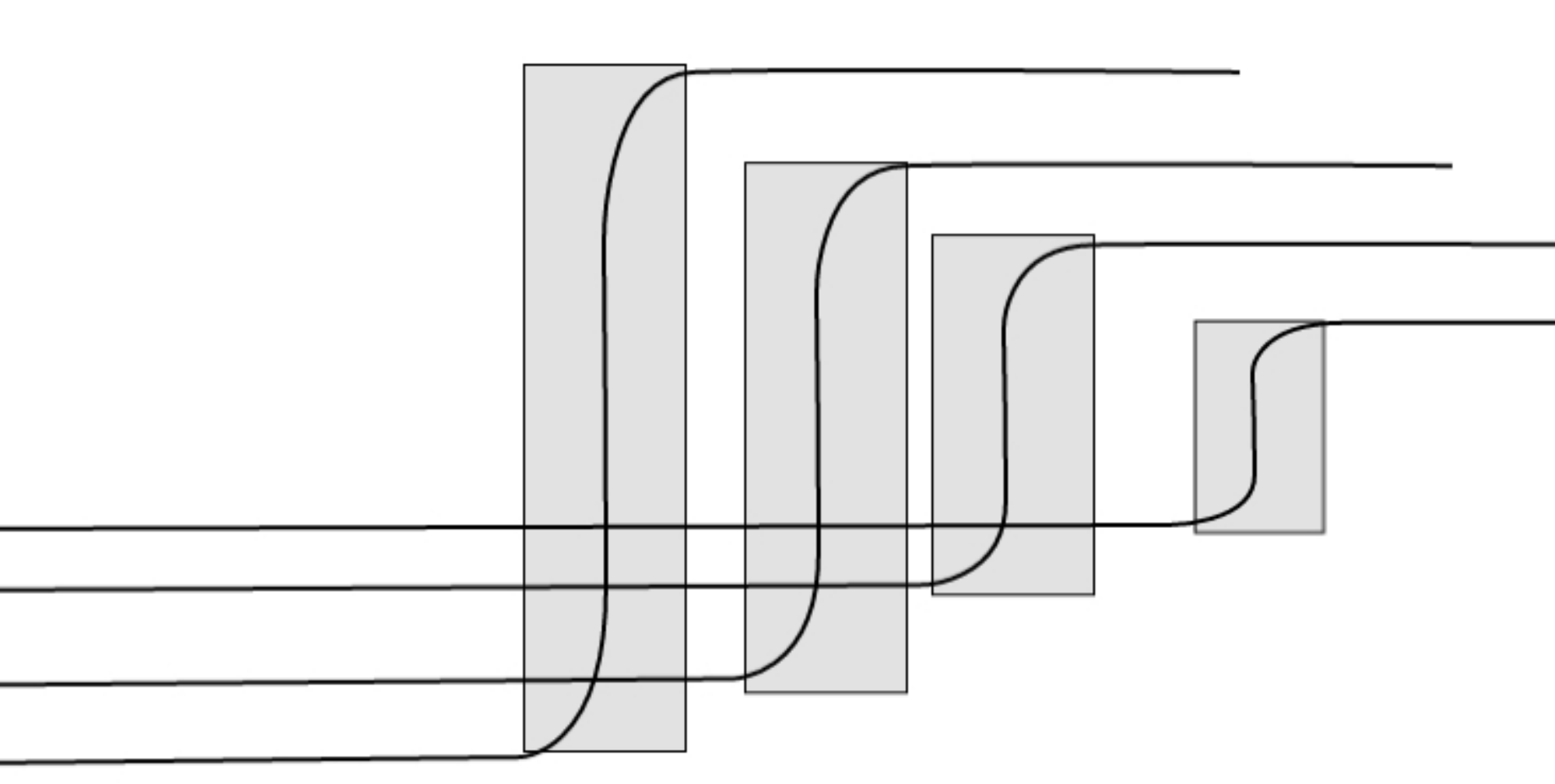}}
			%%Grid, for help.
%		 	,(2,1)*+{\bullet2},(3,1)*+{\bullet3},(4,1)*+{\bullet4},(5,1)*+{\bullet5},(6,1)*+{\bullet6},(7,1)*+{\bullet7},(1,1)*+{\bullet1},(1,2)*+{\bullet2},(1,3)*+{\bullet3},(1,4)*+{\bullet4},(1,5)*+{\bullet5},(1,6)*+{\bullet6},(1,7)*+{\bullet7}
				%%End Helping grid
			,(8,4.4)*+{\gamma_0}
			,(8,5.7)*+{\gamma_1}
			,(7.5,6.5)*+{\gamma_2}
			,(6.4,7.5)*+{\gamma_3}
			\endxy
			\]
\begin{image}\label{figure.gamma-boxes}
An example of a staircase $\gamma_i$. The shaded boxes are the boxes $[w_i - \epsilon_i, w_i] \times [-D_i, h_i]$. 
\end{image}
\end{figure}

\begin{remark}[$\beta_i$ and $\gamma_i$]
In this paper, we will use $\beta_i$ for a staircase in $T^*E$, and $\gamma_i$ for a staircase in $T^*F$. The main difference is as follows: When defining $\beta_i$, we will always have a specific brane in $M \times T^*E$ in mind. So the parameters $\un D, \ov D, D_i, w_i, \epsilon_i, h_i$ are chosen relative to the brane in $M \times T^*E$, rather than relative to a cobordism in $M \times T^*F$.
\end{remark}

\subsubsection{Cone tails}\label{section.cone-tails}
Fix $w>0$. Consider an embedded smooth curve $c \subset T^*F$ as follows:
	\enum
		\item There exists some $\epsilon>0$ such that outside of the boxes
			\eqnn
				\BB = [-w+\epsilon, -w] \times [-\epsilon,\epsilon] \coprod [w-\epsilon, w] \times [-\epsilon,\epsilon]
			\eqnd
		$c$ equals
			\eqnn
				\{-w\} \times (-\infty, -\epsilon)
				\coprod
				\{w\} \times (-\infty,-\epsilon)
				\coprod
				(-w+\epsilon, w-\epsilon) \times \{0\}.
			\eqnd
		\item There exists a primitive $f: C \to \RR$ such that $df = pdq|_{c}$ and $f=0$ outside of $\BB$. 
	\enumd 
We will call $c$ a {\em cone tail}. Much of our combinatorics comes down to understanding holomorphic disks with boundary on $c$ and on a staircase. Note that $c \cap \gamma_i$ consists of two points. (The $w$ used here is the same $w$ used to define the $\gamma_i$.)

\begin{remark}
The purpose of giving $c$ some freedom in $\BB$ is to allow ourselves to find such an $f$.

The convenience of having such an $f$ is so that, when we later make $B(Y)$ out of a cobordism $Y$, we know two things:
\enum
	\item $B(Y)$ is eventually conical in $M \times T^*F$ (to ensure ``vertical'' Gromov compactness), and
	\item $B(Y)$ indeed looks like a product of (some region of) a cone tail with the branes that collar $Y$.
\enumd
The product of two eventually conical branes, famously, need not be eventually conical. This can be fixed by flowing the product of conical branes by some amount dictated by a choice of primitives on each factor of the product. Moreover, we need not flow where a primitive $f$ is equal to zero. See Definition~\ref{defn.product}.
\end{remark}

\begin{figure}
		\[
			\xy
			\xyimport(8,8)(0,0){\includegraphics[width=5in]{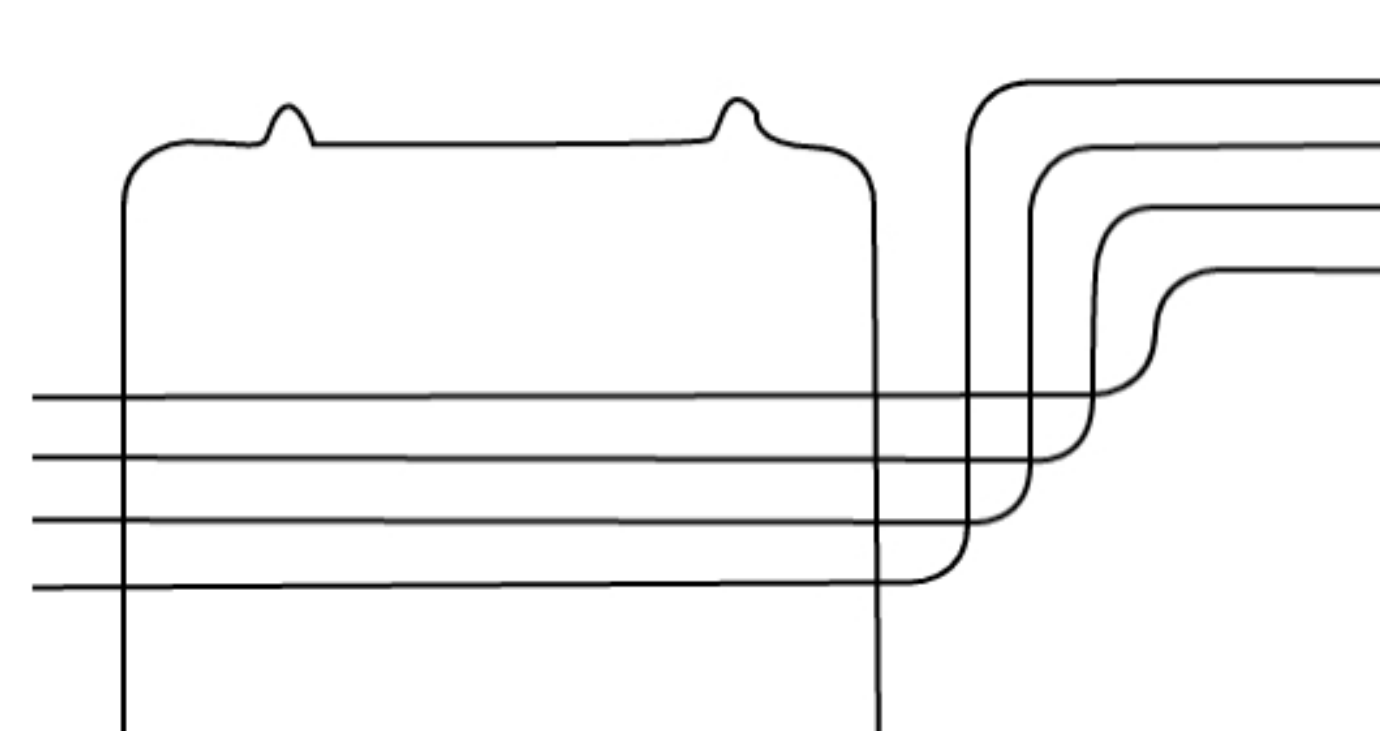}}
			%%Grid, for help.
%		 	,(2,1)*+{\bullet2},(3,1)*+{\bullet3},(4,1)*+{\bullet4},(5,1)*+{\bullet5},(6,1)*+{\bullet6},(7,1)*+{\bullet7},(1,1)*+{\bullet1},(1,2)*+{\bullet2},(1,3)*+{\bullet3},(1,4)*+{\bullet4},(1,5)*+{\bullet5},(1,6)*+{\bullet6},(1,7)*+{\bullet7}
				%%End Helping grid 
			,(2.8,6.7)*+{c}
			\endxy
			\]
\begin{image}\label{c-gamma.pdf}
Staircase curves and a conetail. 
\end{image}
\end{figure}

\begin{defn}[Regions defined by staircases]\label{defn.R} 
Fix a staircase and some $d \geq 1$. Fix also the vertical line $l_w = \{q=w\}$. 
Let $V$ the complement of $l_w \cup \bigcup_{0,\ldots,d-1} \gamma_i$. Let $V^o$ be the union of all its bounded components. We call $R := \ov{V^o}$, the closure of $V^o$, a {\em region bounded by the staircase}. The notation suppresses the dependence on the choices of $\gamma_i$.

Also, letting $V'$ be the complement of $l_{-w} \cup \bigcup_{0,\ldots,d-1} \gamma_i$, we let $R'$ equal $\ov{ (V')^o}$. We will also call this a region bounded by the staircase; the notation will be explicit so no ambiguity arises. 

Finally, letting $V''$ be the complement of $c \cup \bigcup_{0,\ldots,d-1} \gamma_i$, we define $R''$ similarly.

See Figure~\ref{figure.staircase-disks}.
\end{defn}

\subsubsection{The moduli of disks for staircases}
Many computations in life reduce to combinatorics; in this paper, the important computations reduce to the combinatorics of disks on staircases. Much of this is in the spirit of combinatorial Floer theory (i.e., Floer theory on real surfaces), see~\cite{abouzaid-higher-genus, seidel-book, de-silva-et-al}. When discussing disks $u: S \to T^*\RR$, we study honest holomorphic disks, so smooth maps $u$ with $J_{T^*\RR} \circ du = du \circ J_S$. The following will be important for us; it also proves Gromov compactness in the case of the $\gamma_i$ forming a staircase. See also Lemma~7.7 of~\cite{seidel-lefschetz-i} and Lemma~4.4.1 of~\cite{biran-cornea}.

\begin{proposition}\label{prop.bounded}
Let $C$ be a closed subset of $T^*\RR$. Suppose $u: S \to T^*\RR$ is a holomorphic map admitting a continuous extension $\ov{u} : \ov{S} \to T^*\RR$ such that $\ov{u}({\del \ov S}) \subset C$. Then $u$ has empty intersection with any unbounded connected component of $T^*\RR \setminus C$.
\end{proposition}

\begin{proof}[Proof of Proposition~\ref{prop.bounded}]
Let $U \subset T^*\RR \setminus C$ be an unbounded connected component. Since $\ov{S}$ is compact, any intersection that $u(S)$ has with $U$ has to be closed in $U$. The intersection must also be open by the open mapping theorem because $\ov S \setminus \del \ov S = int(S)$ is the open unit disk. Hence the intersection $U$ must be all of $U$, or empty---but $u(S)$ must have bounded image since it extends to $\ov u$, a map from a compact space. (This is essentially the same argument that appears in Remark~4.4.3 of~\cite{biran-cornea}.)

Another argument uses degree:
On any such $U$, $u$ must have constant degree, hence if $u(S) \cap U \neq \emptyset$, $u$ must have non-zero degree on all of $U$, implying it has unbounded image. (This is the argument used in~\cite{seidel-lefschetz-i}.)
\end{proof}

Fix a staircase
	\eqnn
	(\gamma_0, \gamma_1, \ldots, \gamma_{d-1}).
	\eqnd
We will let $\gamma_d = c$ be the cone tail curve for notational convenience. 
Fix $x_i \in \gamma_{i-1} \cap \gamma_i$ and $x_0 \in c \cap \gamma_0$ such that at most two of the $x_i$ have $q$-coordinate less than $w$. 
In this situation, there are three possible configurations of the $x_i$: 
	\begin{enumerate}
		\item\label{item.1} All $x_i$ have $q$-coordinate $q \geq w$.
		\item\label{item.2} $x_0$ and $x_{d}$ are the two points with $q<w$.
		\item\label{item.3} $x_0$ is the unique point with $q < w$.
	\end{enumerate}
See Figure~\ref{figure.staircase-disks}. We let $\cM(\gamma_i,x_i)$ denote the moduli space of holomorphic disks with the usual boundary conditions, where the ordering on the $\gamma_i$ and $x_i$ follow the usual conventions. By the Riemann mapping theorem, in each scenario the moduli space $\cM$ is non-empty. In fact, one can do better. 

\begin{lemma}\label{lemma.fibers}
Let $\cM(\gamma_i,x_i)$ denote the uncompactified moduli space of holomorphic disks. 
Then the map
	\eqnn
		\cM( (\gamma_i, x_i)) \to \cR
	\eqnd 
(to the moduli space of holomorphic structures on a disk with $d+1$ boundary marked points) is, using the same enumeration as above,
	\enum
		\item[\ref{item.1}, \ref{item.2}.] A diffeomorphism, or 
		\item[\ref{item.3}.] A surjective submersion with one-dimensional fibers.
	\enumd
\end{lemma}

\begin{figure}
		\[
			\xy
			\xyimport(8,8)(0,0){\includegraphics[width=1.5in]{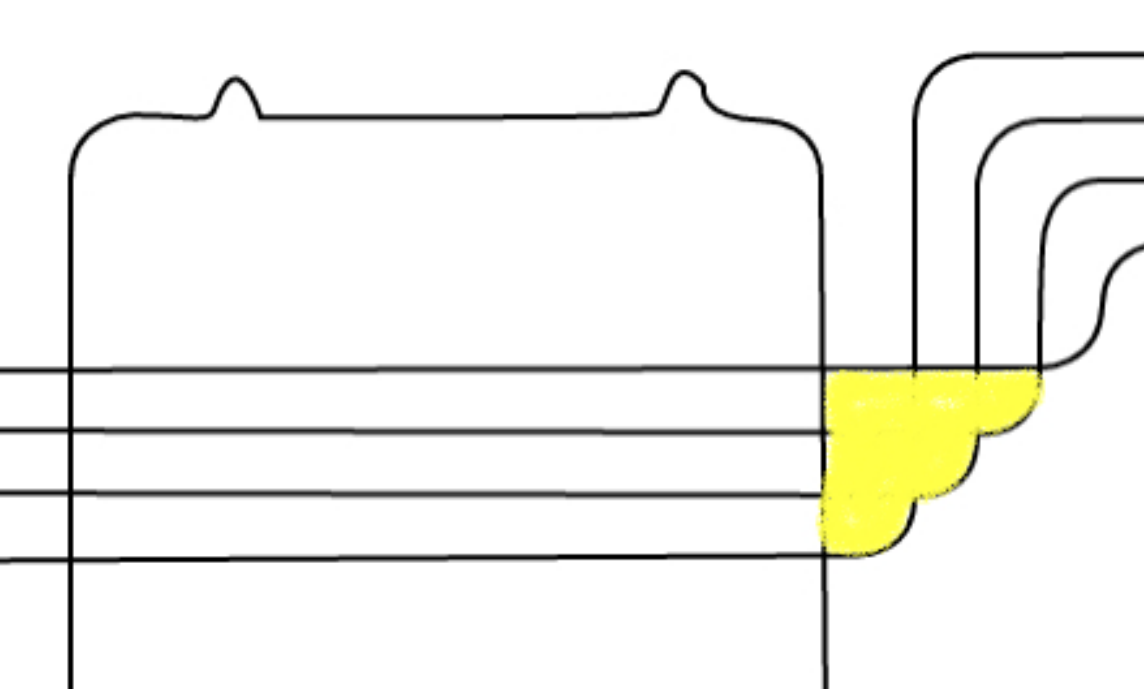}}
			%%Grid, for help.
%		 	,(2,1)*+{\bullet2},(3,1)*+{\bullet3},(4,1)*+{\bullet4},(5,1)*+{\bullet5},(6,1)*+{\bullet6},(7,1)*+{\bullet7},(1,1)*+{\bullet1},(1,2)*+{\bullet2},(1,3)*+{\bullet3},(1,4)*+{\bullet4},(1,5)*+{\bullet5},(1,6)*+{\bullet6},(1,7)*+{\bullet7}
				%%End Helping grid
			,(3,0)*+{\ref{item.1}.}
			\endxy
			\qquad
			\xy
			\xyimport(8,8)(0,0){\includegraphics[width=1.5in]{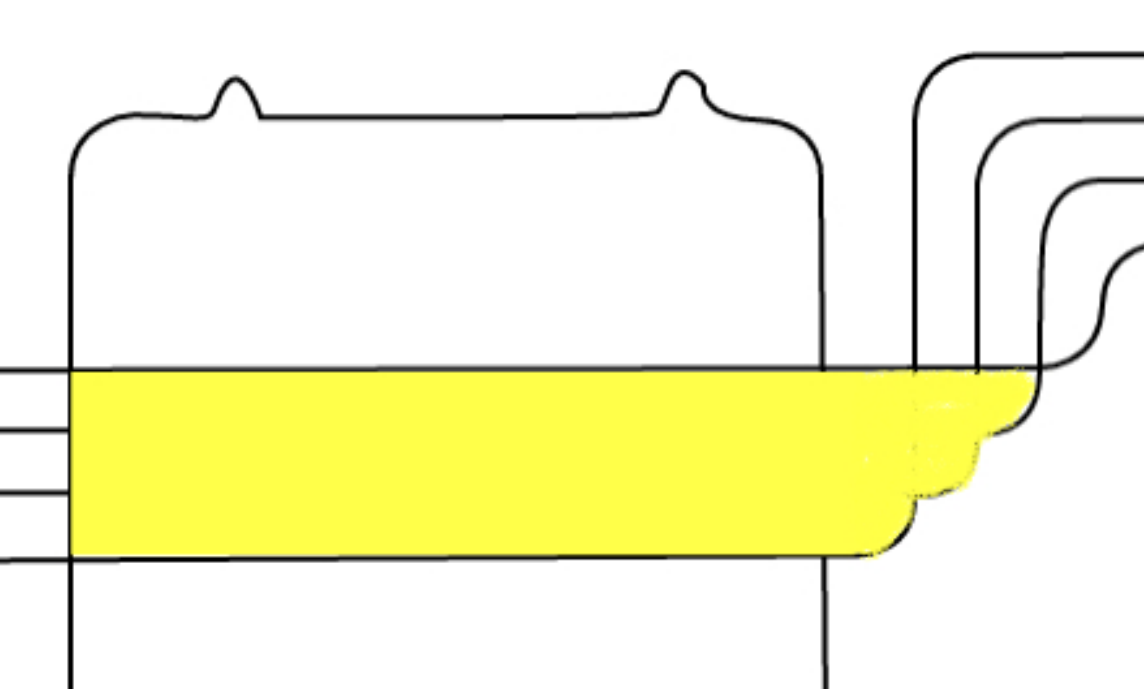}}
			%%Grid, for help.
%		 	,(2,1)*+{\bullet2},(3,1)*+{\bullet3},(4,1)*+{\bullet4},(5,1)*+{\bullet5},(6,1)*+{\bullet6},(7,1)*+{\bullet7},(1,1)*+{\bullet1},(1,2)*+{\bullet2},(1,3)*+{\bullet3},(1,4)*+{\bullet4},(1,5)*+{\bullet5},(1,6)*+{\bullet6},(1,7)*+{\bullet7}
				%%End Helping grid
			,(3,0)*+{\ref{item.2}.}
			\endxy
			\qquad
			\xy
			\xyimport(8,8)(0,0){\includegraphics[width=1.5in]{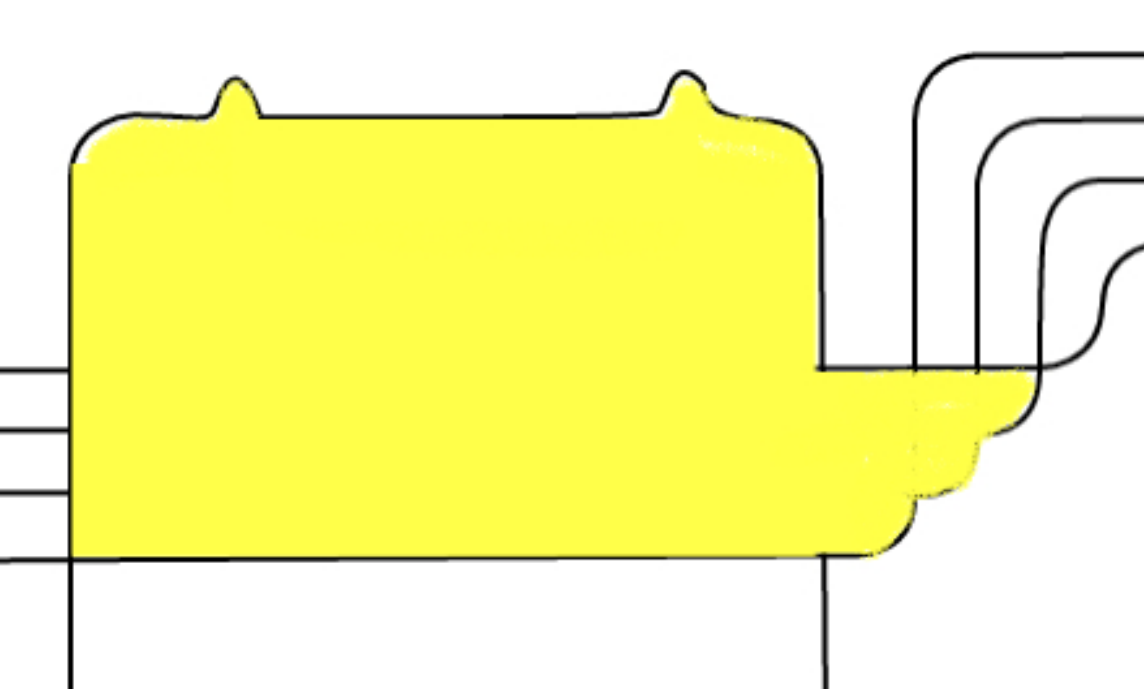}}
			%%Grid, for help.
%		 	,(2,1)*+{\bullet2},(3,1)*+{\bullet3},(4,1)*+{\bullet4},(5,1)*+{\bullet5},(6,1)*+{\bullet6},(7,1)*+{\bullet7},(1,1)*+{\bullet1},(1,2)*+{\bullet2},(1,3)*+{\bullet3},(1,4)*+{\bullet4},(1,5)*+{\bullet5},(1,6)*+{\bullet6},(1,7)*+{\bullet7}
				%%End Helping grid
			,(3,0)*+{\ref{item.3}.}
			\endxy
			\]
\begin{image}\label{figure.staircase-disks}
Holomorphic curves with boundary conditions as given in Lemma~\ref{lemma.fibers}, in yellow. Blue indicates the boundary marked points. Not pictured are the possible {\em slits}, where the endpoints of slits are places where the boundaries of disks have a critical point. The yellow regions in~\ref{item.1}. and~\ref{item.2}. and~\ref{item.3}. respectively depict the regions $R$ and $R'$ and $R''$ from Definition~\ref{defn.R}.
\end{image}
\end{figure}

We begin with a useful warm-up:

\begin{lemma}\label{lemma.q>w}
Fix a staircase with tunnel curves $\gamma_0,\ldots,\gamma_{d-1}$ and fix a kernel tail $c$. Let $u: S \to T^*\RR$ be a holomorphic disk. Then with boundary conditions using the same enumeration as above, we have that
	\enum
		\item[\ref{item.1}.] The image of $u$ is contained in the region $R$ bounded by the staircase. (See Definition~\ref{defn.R}.)
		\item[\ref{item.2}.]  The image of $u$ is contained in the region $R'$ bounded by the staircase.
		\item[\ref{item.3}.]  The image of $u$ is contained in the region $R''$ bounded by the staircase.
	\enumd
\end{lemma}

\begin{remark}[Boundary-stripping]\label{remark.boundary-stripping}
We will reuse the strategy in the following proof over and over again. We call it {\em boundary-stripping} for brevity. We use it to show that (even without using convexity---just by knowing the boundary conditions) certain solutions $\tilde u: S \to M \times T^*\RR$ to (\ref{eqn.floer}) have image fully contained in a region where $J$ splits as a direct sum $J_M \oplus J_{T^*\RR}$. 

Why the name? the proof begins by choosing some closed set $C_0$ on which to apply Proposition~\ref{prop.bounded}, then removing exterior curves from $C_{i-1}$ to obtain a new closed set $C_i$, and we repeat the process using Proposition~\ref{prop.bounded}. The removal of exterior curves is the ``boundary-stripping.'' 
\end{remark}

\begin{proof}[Proof of Lemma~\ref{lemma.q>w}.]
We begin with case \ref{item.1}.
We want to use Proposition~\ref{prop.bounded} where $C$ is the union of the $\gamma_i$ and $c$. To use the proposition, we just need to show that $u(\del S)$ does not intersect any point with $q<w$. So we proceed by assuming the opposite. By the boundary conditions of $U$, $u(\del S)$ has some point such that i. the point has image with $q<w$, and ii. $u|{\del S}$ has positive $\del_q$ derivative. Such a point cannot have image along the deepest curve $\gamma_{d-1}$, as then by the equation $J\circ du = du \circ J$ on $\del S$, $u$ would intersect the open region lying below $\gamma_{d-1}$. 

Setting $C_0 = C$, we can form $C_i$ by removing from $C_{i-1}$ the portion of $\gamma_{i-1}$ on which $u|{\del S}$ has no image. Then the same argument as above shows $u|_{\del S}$ can have no image on the part of $\gamma_i$ with $q<w$. 

By induction downward to $\gamma_{d-2}$ and also $c$, a point satisfying i. and ii. cannot have image along any $\gamma_i$, nor on $c$. In other words, no point on $\del S$ is taken by $u$ to a point with $q<w$, and we can finally use $C_{final} = \{q = w\} \bigcup \gamma_i$ to see that no point of $S$ is taken by $u$ to the region with $q<w$. 

The same arguments show that $u|_{\del S}$ cannot have image along any of the vertical portions of the $\gamma_i$ above $R$, and in turn shows that $u$ has image contained in $R$.

The proof for $R'$ and $R''$ are similar so we omit it.
\end{proof}

\begin{proof}[Proof of Lemma~\ref{lemma.fibers}, cases~\ref{item.1} and~\ref{item.2}.]
We begin with cases \ref{item.1} and \ref{item.2}. Since the latter and the former have the same combinatorics, we just prove \ref{item.1}. 
There are at least two proofs. One uses the degeneration argument Seidel uses in Lemma~7.5 of~\cite{seidel-lefschetz-i}. The other is as follows: 

First consider the case $d+1 = 3$. Since the image of any $u: S \to T^*\RR$ must be bounded, the image of $u$ must be contained in the obvious 3-gon given by the boundary conditions. Then that $\cM = pt = \cR$ is classical; it essentially follows from the Riemann mapping theorem, and the result can be found in other combinatorial Floer theory works such as~\cite{de-silva-et-al, abouzaid-higher-genus}. Now we prove the other cases by induction. Assuming true for all $d' +1 \leq d$, the Riemann mapping theorem tells us that $\cM$ is at least non-empty. And given any $u \in \cM$, it is easy to show that $\cM \to \cR$ is a submersion at $u$: the usual immersed pictures allow one to nudge the singular values of $u|_{\del S}$ along the boundary curves $\gamma_i$ (sometimes referred to as ``moving the slits''). Knowing that $\cM \to \cR$ admits a compactification $\ov{\cM} \to \ov{\cR}$, by induction one obtains a proper fibration which is degree one over the boundary strata. (Here we are using the explicit knowledge of the fact that the ``broken heart'' picture indeed corresponds to boundary degenerations of holomorphic disks; see~\cite{de-silva-et-al, abouzaid-higher-genus}.) Now we must show that the fibers have cardinality 1, and this one can do near the boundary of $\cM$: Given a fixed degenerate disk $u \in \ov \cM$ in a codimension one stratum, there is only one slit along which one can re-open $u$ to obtain a disk in the interior of $\cM$. Hence the cardinality of the fibration $\cM \to \cR$ does not change moving from $\del \cR$ into its interior. For an equivalent argument, see also Lemmas~4.22 and~4.23 in Pascaleff~\cite{pascaleff}.
\end{proof}

Before we go to case~\ref{item.3} of Lemma~\ref{lemma.fibers}, we begin with some generalities.

If $\cR = \cR_d$ is the moduli space of conformal structures on holomorphic disks with $d+1$ boundary marked points, it admits a closure $\ov \cR$ which is neither compact nor Hausdorff---it is obtained by attaching moduli spaces of semistable nodal disks. (One can also think of it as the coarse moduli space associated to the topological stack of broken holomorphic disks, defined on the usual site of topological spaces.) The details are not important in this paper, but we give a brief description in Remark~\ref{remark.coarse-moduli} below. In our argument to prove Lemma~\ref{lemma.fibers}, we use only the fact that (i) the interior $\cR \subset \ov \cR$ is open and is diffeomorphic to Euclidean space, and (ii) the quotient $\ov{\cR} / \del \ov{\cR}$ can be identified with the one-point compactification of $\cR$. Importantly, we thus see that the quotient admits a pseudometric 
	\eqnn
	\dist: \ov{\cR} / \del \ov{\cR} \times \ov{\cR} / \del \ov{\cR} \to \RR \cup \{\infty\}
	\eqnd
which is smooth when restricted to $\cR$, continuous, and satisfies
	\eqnn
		\dist(x,y) = 
			\begin{cases}
				\infty & \text{if either $x$ or $y$ is the point at infinity} \\
				< \infty & \text{$x$ and $y$ are in $\cR$.}
			\end{cases}
	\eqnd
Now, the compactification $\ov \cM$ of holomorphic disks $u: S \to M$ admits a continuous map $\ov f: \ov \cM \to \ov \cR$ satisfying the following property: 
	\eqnn
	f := \ov f|_{\cM} : \cM \to \cR
	\eqnd
is smooth, and $\ov f^{-1} (\del \ov \cR) = \del (\ov \cM).$  In terms of disks, this means that the compactification of $\cM$ only involves degenerating a sequence $u_i: S_i \to M$, $u_i \in \cM$, in such a way that the limit of the $u_i$ always defines a tuple of maps $(v_1,\ldots,v_k)$ for $k \geq 2$. Put a stronger way, if one stratifies the compactification $\ov \cM$ by $k \in \ZZ_{\geq 1}$, where $k$ counts the number of components of a semistable nodal disk, and likewise for $\ov \cR$, then the map $\ov \cM \to \ov \cR$ respects this stratification. (The $k=1$ stratum is the interior $\cM \subset \ov \cM$.)

Note that this strategy will avoid precisely defining a ``smooth manifold with corners'' structure on $\ov \cM$ (which is a delicate question), and will rely only on the smooth structure on the interior $\cM \subset \ov \cM$. 

\begin{proof}[Proof of Lemma~\ref{lemma.fibers}, case~\ref{item.3}.]
First, let us show that the map of boundary-less spaces $f: \cM \to \cR$ is a submersion. We do this by the same method as before: We explicitly use the slits of holomorphic disks (immersed along some of its boundary) to parametrize the moduli of marked points on a disk modulo automorphisms. Given some $u: S \to T^*\RR$ in $\cM$, one can move slits in $T^*\RR$ in every allowable direction to obtain another holomorphic map $u: S' \to T^*\RR$ from a different conformal structure $S'$.

Now we show it is a surjection. Assume not, and fix $y \in \cR$ which is not in the image of $f: \cM \to \cR$. Then the composite function
	\eqnn
	\dist(f(-),y): \ov \cM \to \ov \cR \to\ov{\cR} / \del \ov{\cR} \to \RR \cup \{\infty\}
	\eqnd
is continuous. Since $\ov \cM$ is compact, the function achieves a minimum somewhere. Knowing $\cM$ (hence $f(\cM)$) is non-empty, the minimum is achieved at some $x \in \cM$ such that $\dist(f(x),y) < \infty$. But we know that $f$ is a submersion, so there is some open neighborhood of $f(x)$ in $\cM$ which is contained in the image of $f$---this contradicts the minimality of $\dist(f(x),y)$. Hence $y$ could not exist, meaning $f$ was a surjection to begin with.

Since $f$ is a submersion, a simple index count shows that the fibers of $f$ are indeed 1-dimensional.
\end{proof}

\begin{remark}\label{remark.coarse-moduli}
Fix an integer $d \geq 1$. When one considers the Gromov compactification $\ov \cM = \ov \cM_d$ of the moduli of pseudoholomorphic maps with $d+1$ marked points, there is a natural (non-compact, non-Hausdorff) base space $\ov \cR$ to which $\ov \cM$ maps. For instance, suppose that $\cM$ is some (high-dimensional) moduli space of holomorphic disks. Then we know that $\ov \cM$, as a topological manifold with corners, is a compact space with codimension $k$ boundary strata given by $(k+1)$-tuples of holomorphic disks $u_i: S_i \to M$; this tuple is what one usually calls a ``broken disk.'' Since there are obvious automorphisms when some $S_i$ is a strip, the natural recipient of a map $\ov\cM\to \ov\cR$ would hence be a stack, also stratified by the combinatorics of broken strips and disks, but we won't define such a thing here. Instead, we define a non-compact, non-Hausdorff topological space as a colimit
	\eqnn
	\ldots \to (\ov \cR)^k \to (\ov \cR)^{k+1} \to \ldots
	\eqnd
given by attaching a point for the nodal disks at each stage. We don't elaborate too much on this construction because our proof of Lemma~\ref{lemma.fibers} part~\ref{item.3}. only needs the quotient of $\ov \cR$ by its boundary anyway, but here is a brief description:

As a set, $(\ov \cR)^k$ is the set of all degenerations of a disk with $(d+1)$ boundary marked points, with at most $k$ components. (Each component must have at least two boundary marked points.) Inductively, one endows $(\ov \cR)^k$ with the coarsest topology so that (1) $(\ov \cR)^{k-1} \into (\ov \cR)^k$ is an open inclusion, (2) the boundary strata have the direct product topology (well-defined by induction), and (3) each boundary stratum (where each stratum is indexed by the combinatorial data of a broken tree/disk, rather than by codimension) is a closed subset of $\ov \cR^{k}$. Importantly, we forget the $\RR$ action on strips---this topology treats the breaking of strips as simply passing to a formal boundary (hence the non-Hausdorff property).
\end{remark}

\begin{example}
As an example, let $d=2$. This means we consider the moduli space of disks with 3 marked boundary points.

As we know, $\cR = (\ov \cR)^1$ is just a single point---up to equivalence, there is only one conformal structure on a disk with three marked boundary points. When such a disk degenerates to have one more nodal component, we have three possible degenerations, given by attaching a strip at any of the three marked points. Let us call these degenerations $a_0, a_1, a_2$, and let $b$ be the unique point of $\cR$. Then 
	\eqnn
	(\ov \cR)^2
		=
		\{b,a_0,a_1,a_2\}
	\eqnd
is a four-element set, topologized so that the closed sets are generated by $\emptyset, (\ov \cR)^1$, $\{a_0\}$, $\{a_1\}$, and $\{a_2\}$. Note that $\cR = \{b\}$ is not a closed subset, but is open.
\end{example}

\begin{example}
Consider the case $d =1$, so we consider $\ov \cR$ to be the closed moduli space of broken strips. As a set, $\ov \cR$ is in bijection with $\ZZ_{\geq 1}$, and $\cR$ is homeomorphic to $\ZZ_{\geq 1}$ with the poset topology---a set $U$ is open if and only if $x \in U \implies y \in U$ whenever $x \geq y$. (I.e., every open set is upward closed, and vice versa.) Note $\ov \cR$ is in fact a non-unital monoid in the category of stratified spaces.
\end{example}

\subsection{Geometry of products}
Given two exact symplectic manifolds $M_1$ and $M_2$, their direct product is exact via the 1-form $\theta_{M_1 \times M_2} = \theta_{M_1} \oplus \theta_{M_2}$. Hence one can ask to endow $M_1 \times M_2$ with compatible Hamiltonians, almost-complex structures, and so on. In the next sections we make precise what kinds of $J$ and $H$ we use to do Floer theory in manifolds of the form $M \times T^*\RR^N$ with $\theta_{T^*\RR^N} = \sum p_i dq_i$. 

\begin{remark}
This $\theta$ is again not a generic choice. For instance, any Reeb orbit in $\del M_1$ induces a 1-parameter family of Reeb orbits in $M_1 \times T^*\RR$ living above the zero section of $T^*\RR$.
\end{remark}

Given branes $L_1 \subset M_1, L_2 \subset M_2$, even if each $L_i$ is eventually conical, their product need not be. This will not bother us, as Gromov compactness for $M \times T^*\RR$ will be proven using more standard complex analysis. However, note that there is a fix to non-conicalness in general:

First, choose a Hamiltonian $H_i: M_i \to \RR$ such that 
	\enum
	\item the associated Hamiltonian vector field $X_i$ eventually agrees with the Reeb vector field along $\del M_i$. This means, for instance, that if $r_i$ is the conical coordinate for $M_i$, $H_i$ is equal to $r_i$ for $\del M \times \RR_{\geq R}$, $R$ large enough.
	\item $H_i$ equals zero where $r_i \leq R'$ for some $R' \in \RR_{\geq 0}$, and along $A$. We want to choose $r_i$ so that: When $H_i$ is restricted to $L_i$, it equals zero on the interior of $L_i$, then increases to the function $r_i$ along some conical part of $L_i$, then finally equals $r_i$ far enough away from $A \subset M$.
	\enumd
We let $\Phi_{t_i}^{X_i}$ denote the flow along $X_i$ for time $t_i$. 

\begin{defn}\label{defn.product}
Given branes $L_i$ with primitives $f_i$, we let $L_1 \tensor L_2$ denote the embedding 
	\eqnn
	L_1 \times L_2 \to M_1 \times M_2,
	\qquad
	(x_1,x_2) \mapsto (\Phi_{-f_2(x_2)}^{X_1} (x_1), \Phi_{-f_1(x_1)}^{X_2}(x_2)).
	\eqnd
with the dependence on choice of $H_i$ implicit in the notation $\tensor$.
\end{defn}

We leave as an exercise to the reader that this indeed defines a brane inside $M_1 \times M_2$ with induced Liouville structure $\theta_1 \oplus \theta_2$. The assumption $f_i|_{\del L_i} = 0$ is essential---both to prove $L_1 \tensor L_2$ is indeed embedded, and to prove it is conical. This is in fact the beginnings of the tensor product we will use in~\cite{tanaka-symm} to show that $\lag(M)$ is linear over $\cL$ for any $M$. (See (~\ref{eqn.conjecture}).)

\subsubsection{Compactness and regularity for $M \times T^*\RR$}\label{section.MTR}
Throughout, we fix some number $T \in \RR_{>0}$. We demand that it be larger than $\un{D}$, so we have $D < \ov D < \un{D} < T$.

We utilize Floer and perturbation data which are, outside some compact region of $T^*\RR$, invariant under the translation symmetry of $T^*\RR$ (i.e., depends only on $p$ for $(q,p)$ large enough). Specifically:

{\em The Hamiltonians.} We demand that any Hamiltonian $H: M \times T^*\RR \to \RR$ 
	\enum
		\item
			Is eventually quadratic on $M \times T^*\RR$. This means that, outside of some finite radius tubular neighborhood of $\sk(M \times T^*\RR \cong \sk(M) \times \sk(T^*\RR)$, $H = r^2 + p^2$. Here, $r$ is the conical coordinate of $M$ in the region where $M$ can be identified with a symplectization, and $p$ is the cotangent coordinate of $T^*\RR$.  
		\item
			Outside $M \times T^*[-a, a]$ for some $a>0$, $H$ is independent of the $q$ variable, and
		\item
			Along the set $M \times \RR \times [-\un{D},-\ov{D}]$, $H$ has no dependence on the $T^*\RR$ variable.\footnote{This guarantees that if $\tilde u$ satisfies (\ref{eqn.floer}) but has projection in $\RR \times [-\un{D},-\ov{D}]$, then $u = \pi_{T^*\RR} \circ u$ satisfies the ordinary holomorphic equation.} 
		\enumd
We let $\cH = \cH(M \times T^*\RR)$ denote the space of such $H$. 

{\em The almost-complex structures.} On the one hand, having a direct sum almost complex structure $J_M \oplus J_{T^*\RR}$ is advantageous as it allows us to enumerate disks easily. On the other hand, this rarely achieves transversality, and does not guarantee compactness as $J_{T^*\RR}$ is certainly not of contact type with respect to $pdq$. So we specify a family $\cJ$ of almost complex structures that display the best of both worlds---the family depends on the choice of parameters $w, D, T$. On first glance, the conditions may seem too stringent to achieve transversality, but we can deduce a priori that any pseudoholomorphic disk $u$ solving (\ref{eqn.floer}) must stay within particular regions wherein we can actually deduce transversality. See Figure~\ref{figure.J-regions}.

\begin{figure}
		\[
			\xy
			\xyimport(8,8)(0,0){\includegraphics[width=2in]{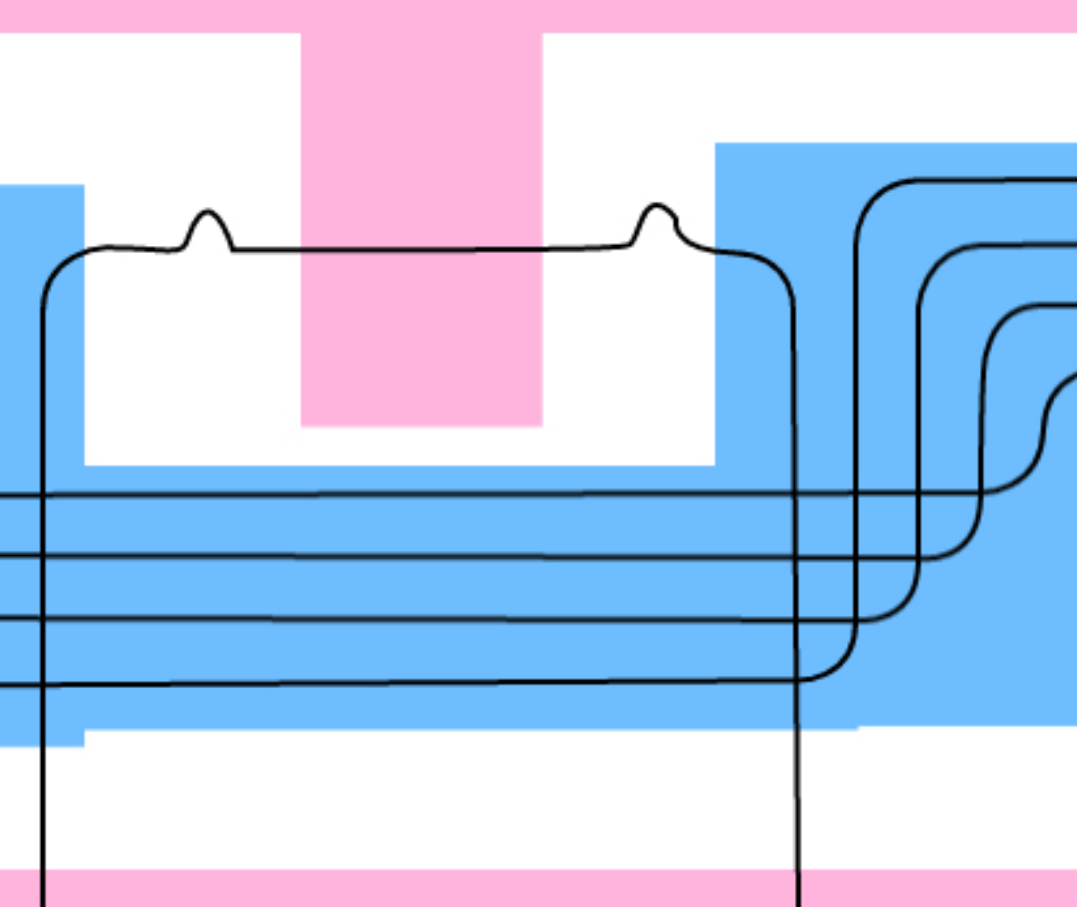}}
			%%Grid, for help.
%		 	,(2,1)*+{\bullet2},(3,1)*+{\bullet3},(4,1)*+{\bullet4},(5,1)*+{\bullet5},(6,1)*+{\bullet6},(7,1)*+{\bullet7},(1,1)*+{\bullet1},(1,2)*+{\bullet2},(1,3)*+{\bullet3},(1,4)*+{\bullet4},(1,5)*+{\bullet5},(1,6)*+{\bullet6},(1,7)*+{\bullet7}
				%%End Helping grid
%			,(4,7.5)*+{\text{eventually contact type}}
%			,(4,0.3)*+{\text{eventually contact type}}
%			,(7,2)*+{J_M \oplus J_{T^*\RR}}
			\endxy
			\qquad
			\xy
			\xyimport(8,8)(0,0){\includegraphics[width=2in]{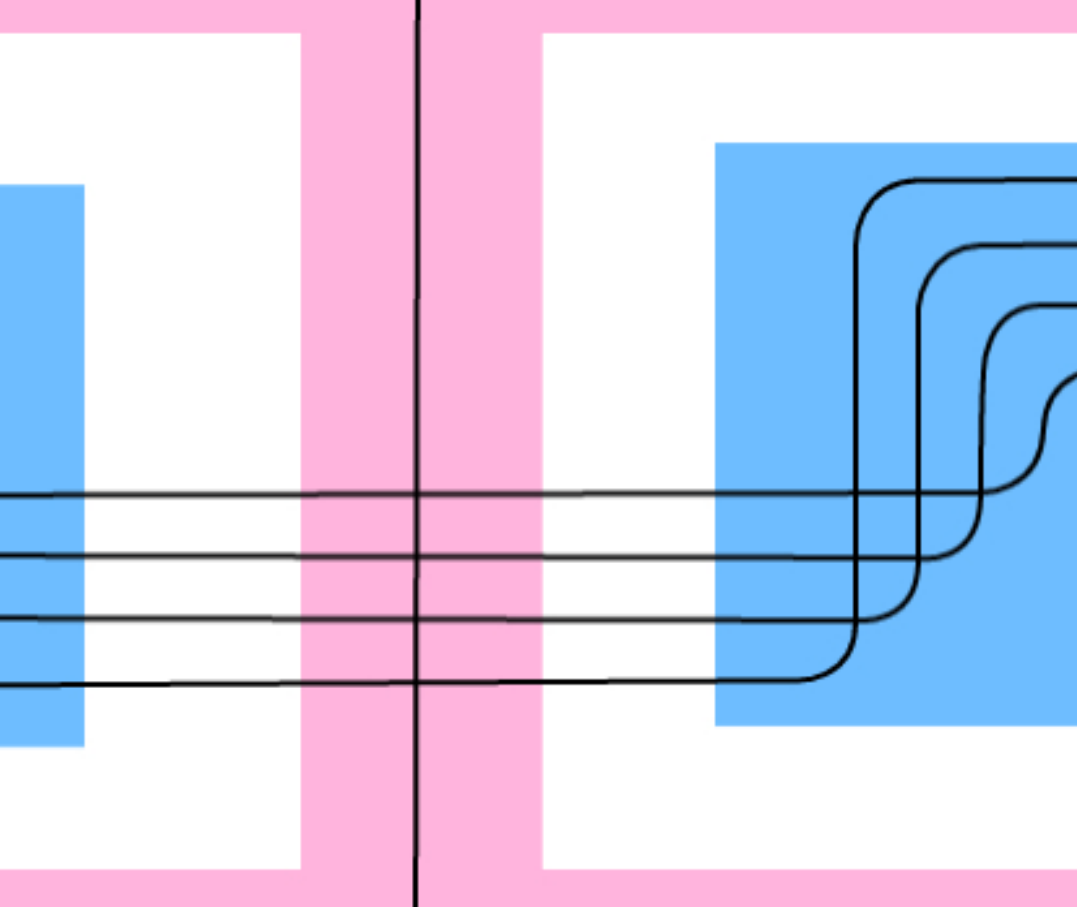}}
			%%Grid, for help.
%		 	,(2,1)*+{\bullet2},(3,1)*+{\bullet3},(4,1)*+{\bullet4},(5,1)*+{\bullet5},(6,1)*+{\bullet6},(7,1)*+{\bullet7},(1,1)*+{\bullet1},(1,2)*+{\bullet2},(1,3)*+{\bullet3},(1,4)*+{\bullet4},(1,5)*+{\bullet5},(1,6)*+{\bullet6},(1,7)*+{\bullet7}
				%%End Helping grid
%			,(4,7.5)*+{\text{eventually contact type}}
%			,(4,0.3)*+{\text{eventually contact type}}
%			,(7,2)*+{J_M \oplus J_{T^*\RR}}
			\endxy
			\]
\begin{image}\label{figure.J-regions}
The almost-complex structures when we pair against $B(Y) \subset M \times T^*F$ (on the left image), and against an arbitrary brane $Y \subset M \times T^*E$ (on the right image), with $\pi_{T^*\RR}(B)$ assumed inside the vertical pink column. In the pink region, one has an almost complex structure of eventually contact type. In the blue region, we demand it equal a direct sum $J_M \oplus T^*\RR$ where $J_M$ is some eventually contact type almost-complex structure on $M$, and $J_{T^*\RR}$ is the standard almost-complex structure (\ref{eqn.i}) on $T^*\RR$.
\end{image}
\end{figure}

In what follows, we again make the distinction between $E$ and $F$. This distinction is for readability, and we will write $T^*\RR$ (rather than $T^*E$ or $T^*F$) when the distinction does not matter.

We let $\cJ_F$ consist of all almost complex-structures on $M \times T^*\RR$ such that
	\enum
		\item 
			On the subset
  			\eqnn
  			M \times \left( [-w, w] \times [-\ov D, \infty) 
  			\bigcup
  			F\times [A, \infty)
				\bigcup
				F \times (-\infty,-A] \right)
				\subset M \times F \times F^\vee
  			\eqnd
			$J$ is an arbitrary almost complex structure which is eventually of contact type (with respect to $\theta_M \oplus pdq$) and compatible with $\omega$. Here, $A$ is some large positive number.
		\item
			On the subsets
				\eqnn
					M \times (-\infty,\infty) \times [-\un{D}, -\ov{D}]
					\subset M \times F \times F^\vee
				\eqnd
			and
				\eqnn
					M \times (w, \infty) \times [-\un D, h] \subset M \times F \times F^\vee
				\eqnd
			$J$ is a direct sum $J_M \oplus J_{T^*\RR}$ where $J_M$ is some eventually contact type almost-complex structure on $M$ compatible with $\omega$, and $J_{T^*\RR}$ is the standard almost-complex structure on $T^*\RR$ (\ref{eqn.i}). Recall also that $h = \sup h_i$, where the $h_i$ are the height of the staircase curves.
		\item
			For some $a>0$, we demand that $J$ is independent of the $q$ variable outside of $M \times T^*[-a, a]$. This need not be the same $a$ that appears in our constraints for $\cH$, but we will not lose any generality in demanding that they be equal.
	\enumd
This is the kind of $J$ we allow when we pair $X_i \times \gamma_i$ against $B(Y)$ for $Y$ a cobordism---hence the $F$ notation.

On the other hand, when one of the boundary conditions is an arbitrary conical brane $Y \subset M \times T^*E$, and assuming $Y$ is contained in\footnote{This is assumption~(\ref{eqn.compact-image}), and the interval $[-w,w]$ can be taken to be $K$ in the notation of 1. in Section~\ref{section.overview-transversality}. }  $M \times T^*[-w,w]$, we let $\cJ_E$ consist of the space of $J$ satisfying:
	\enum
		\item
			$J$ is eventually conical (with respect to $\theta_M \oplus pdq$) and
		\item
			On the region 
				\eqnn
				\left( (-\infty, -w+\epsilon] \coprod (w-\epsilon, \infty) \right)
				\times [-\un D , - \ov D] \subset E \times E^\vee
				\eqnd
			$J$ is a direct sum $J_M \oplus J_{T^*\RR}$ where $J_{T^*\RR}$ is the standard (\ref{eqn.i}) structure, and $J_M$ is some eventually conical structure on $M$. 
	\enumd

Up to conformal equivalence on the boundary of the moduli space of disks with conformal structures, one can as usual choose universal and consistent perturbation data respecting these conditions.

Throughout this section, we only consider moduli spaces of disks with boundary conditions on some fixed brane $B \subset M \times T^*F$, and branes of the form $X_i \times \gamma_i$ where $\gamma_i$ are a staircase with $i=0,\ldots,d-1$ and $X_i \subset M$ are branes within $e$ of $\Lambda$\footnote{Here, as usual, $\Lambda$ is either equal to $M$ itself, or is equal to the skeleton of $M$.}. We treat $B(Y)$ as the $d$th brane when ordering the boundary of $S$. We also assume that $B$ is contained in $M \times [-w,w] \times \infty$, and that it is eventually conical as a submanifold of $M \times T^*F$. The arguments for $\cJ_E$ almost-complex structures are included where they differ from the arguments for $\cJ_F$.

\begin{proof}[Proof of compactness]
As usual, $q$ denotes the $F$ coordinate, and $p$ denotes the $F^\vee$ (i.e., cotangent) coordinate. $\pi_\RR: T^*\RR \to \RR$ denotes the projection.

First note that, if the boundary of any disk $\tilde u: S \to M \times T^*\RR$ is to have very large $p$ coordinate, it must do so along the brane $B$, and not along $X_i \times \gamma_i$.

\begin{figure}
		\[
			\xy
			\xyimport(8,8)(0,0){\includegraphics[width=2in]{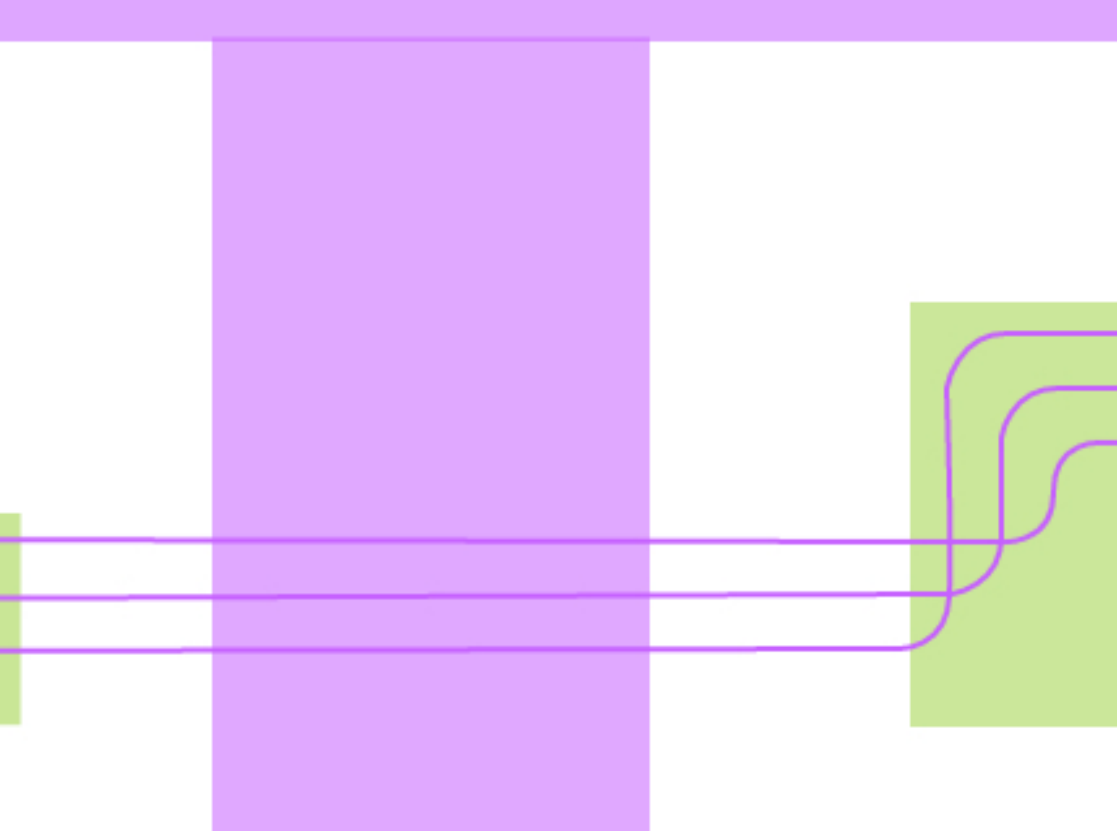}}
			%%Grid, for help.
%		 	,(2,1)*+{\bullet2},(3,1)*+{\bullet3},(4,1)*+{\bullet4},(5,1)*+{\bullet5},(6,1)*+{\bullet6},(7,1)*+{\bullet7},(1,1)*+{\bullet1},(1,2)*+{\bullet2},(1,3)*+{\bullet3},(1,4)*+{\bullet4},(1,5)*+{\bullet5},(1,6)*+{\bullet6},(1,7)*+{\bullet7}
				%%End Helping grid
%			,(4,7.5)*+{\text{eventually contact type}}
%			,(4,0.3)*+{\text{eventually contact type}}
%			,(7,2)*+{J_M \oplus J_{T^*\RR}}
			\endxy
			\qquad
			\xy
			\xyimport(8,8)(0,0){\includegraphics[width=2in]{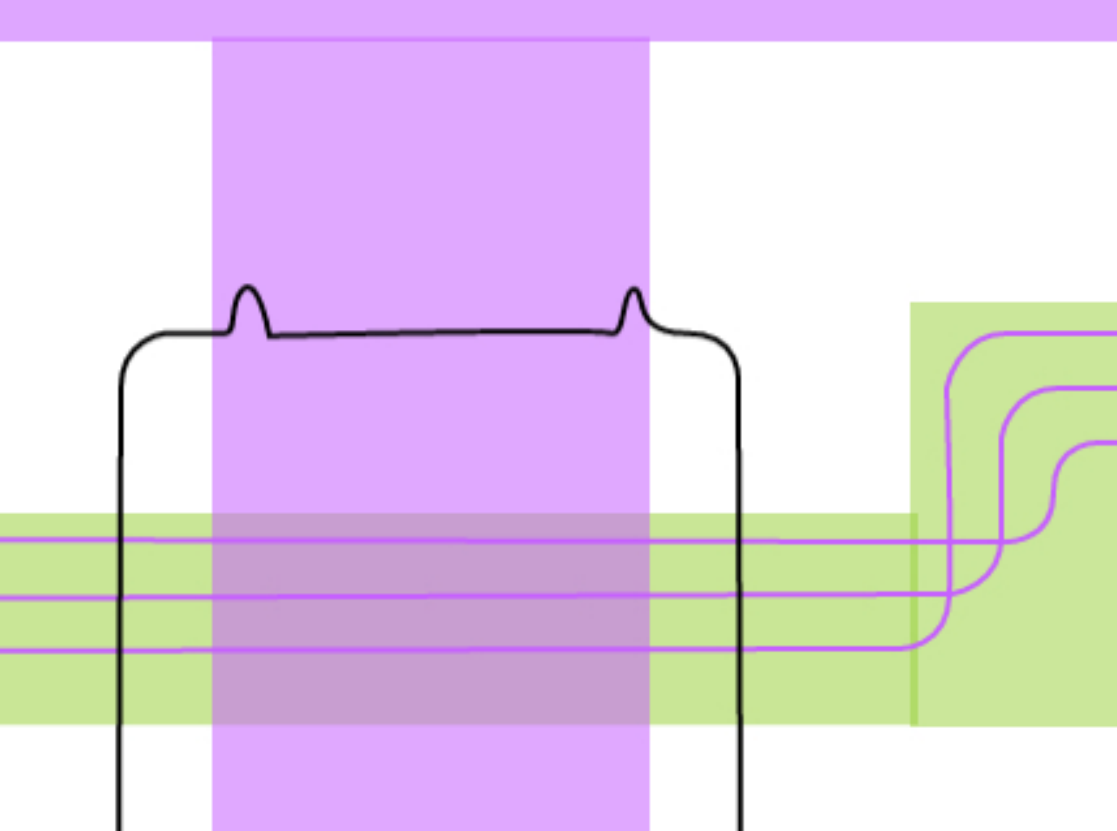}}
			%%Grid, for help.
%		 	,(2,1)*+{\bullet2},(3,1)*+{\bullet3},(4,1)*+{\bullet4},(5,1)*+{\bullet5},(6,1)*+{\bullet6},(7,1)*+{\bullet7},(1,1)*+{\bullet1},(1,2)*+{\bullet2},(1,3)*+{\bullet3},(1,4)*+{\bullet4},(1,5)*+{\bullet5},(1,6)*+{\bullet6},(1,7)*+{\bullet7}
				%%End Helping grid
%			,(4,7.5)*+{\text{eventually contact type}}
%			,(4,0.3)*+{\text{eventually contact type}}
%			,(7,2)*+{J_M \oplus J_{T^*\RR}}
			\endxy
		\]
\begin{image}\label{figure.C}
Indicated in purple is the closed subset $C \subset T^*\RR$ from (\ref{eqn.C}). (Note the staircase curves are purple, too.) In green are regions where the almost complex structure on $M \times T^*\RR$ is a direct sum $J_M \oplus J_{T^*\RR}$. The first figure applies when we are doing Floer theory in $M \times T^*E$, while the latter when we are doing Floer theory in $M \times T^*F$.
\end{image}
\end{figure}

But anything with large $p$ coordinate falls into the regime where $J$ is of contact type. So now we can apply Section~7.3 of~\cite{abouzaid-seidel}, where the details for the compactness argument in our Section~\ref{section.M-analysis} are carried out. The same convexity argument applies here to show that all disks satisfying (\ref{eqn.floer}) must live inside some finite radius of the skeleton $\Lambda_M \times \RR \subset M \times T^*\RR$, and in particular, inside $N_\delta(\Lambda_M) \times N_\delta(\RR)$, where $N_\delta$ is some bounded neighborhood of radius $\delta$. ($\delta$ may be large, but it is finite). 

So now let us show that every disk with prescribed boundary conditions must be bounded in the $q$ direction. Given a disk $\widetilde{u_i}: S \to M \times T^*\RR$ , let $u_i$ denote $\pi_{T^*\RR} \circ \widetilde{u_i}$. Using Proposition~\ref{prop.bounded} again with $C$ equal to 
	\eqn\label{eqn.C}
		C = \RR_q \times [T,\infty) \bigcup [-w+\delta,w-\delta] \times \RR_p  \bigcup (\bigcup_{0 \leq i \leq d-1} \gamma_i)
	\eqnd
for some choice of small $\delta>e> 0$.
See Figure~\ref{figure.C}. We see that $u_i(\del S)$ must have bounded $q$ coordinate. Hence if we have a sequence of disks  $\widetilde u_i$ escaping in the $q$ coordinate, the derivatives $||du_i||_\infty$ must tend to infinity.  Choose $z_i \in S$ to be points of maximal derivative for $u_i$. Fixing a given $\rho$, one can then look at neighborhoods $D_{i,\rho} \subset S$ of $z_i$ of radius $\rho / ||du_i||$. For large enough $i$, this must be in the interior of $S$, and we can define a map 
	\eqn
	\xymatrix{
	v_i: B_0(\rho) \ar[rrr]^-{( \bullet + z_i) / ||du_i||_\infty}_-{\cong} 
		&&&D_{i,\rho} \ar[r]^-{u_i}
		&  T^*\RR \ar[rrr]^-{+C-\pi_\RR(u_i(z_i))}
		&&& T^*\RR
	}
	\eqnd
so $\pi_\RR(v_i(0)) = c$ for some fixed $c$ independent of $i$. ($B_0(\rho)$ is the ball of radius $\rho$ centered at $0 \in \CC$.) The eventual $\RR$-invariance of the perturbation data guarantees that $v_i$ still satisfies equation~(\ref{eqn.floer}). In fact, since rescaling by $||du_i||_\infty$ scales $\alpha$, in the limit $(dv_i-X\tensor \alpha)^{0,1} \to (dv_i)^{0,1}$, and after choosing ever-expanding $\rho$, one obtains an honest holomorphic map from $\CC P^1$ to $T^*\RR\cong \overline{\CC}$ which is non-constant. This is a contradiction. This riff on the usual argument (going back at least in some form to Uhlenbeck) thus proves that any $u$ must remain inside $N_\delta(\Lambda_M) \times N_\delta([-a,a])$ for some finite $a \in \RR$. Since $\Lambda_M$ is assumed compact, we have Gromov compactness.
\end{proof}

\begin{proof}[Proof of regularity]
Given $\tilde u:  S \to M \times T^*\RR$, again set $u = \pi_{T^*\RR} \circ \tilde u$. Let us first assume we are considering disks only with boundary on $X_i \times \gamma_i$ (so $B$ is not one of the boundary conditions). Again (and this is a key point) using Proposition~\ref{prop.bounded}, and the same choice of $C$ as (\ref{eqn.C}), we see that 
	\eqn
	\text{$u$ has image where $J=J_M \oplus J_{T^*\RR}$}
	\eqnd
where $J_{T^*\RR}$ is the usual almost-complex structure (\ref{eqn.i}).

So we can apply the following, which we learned in~\cite{seidel-lefschetz-i,seidel-book} and~\cite{de-silva-et-al}. We refer to it as automatic regularity following~\cite{seidel-book}. Namely, one first considers the linearized operator for holomorphic (in the non-perturbed sense) disks $u: S \to T^*\RR$ with boundary on $\gamma_i$. Then one can show that the linearized Cauchy-Riemann operator $D$ is injective whenever $\text{index}(D) = \dim\ker D - \dim \coker D \leq 0$ and surjective when $\text{index}(D) \geq 0$. (See Lemma~12.4 of~\cite{de-silva-et-al}.)

By the direct sum assumption, we have a diagram of short exact sequences comprised of the usual domains and codomains of the linearized Floer operator:
	\eqn\label{eqn.regular-SES}
		\xymatrix{
		0  \ar[r] & W^{1,p}_{T^*\RR} \ar[r] \ar[d]^D & W^{1,p}_{M \times T^*\RR} \ar[d] \ar[r]& \ar[r] W^{1,p}_{M} \ar[d] \ar[r] & 0  \\
		0  \ar[r] & L^p_{T^*\RR} \ar[r]  & L^p_{M \times T^*\RR}  \ar[r]& \ar[r] L^p_{M}  \ar[r] & 0 .
		}
	\eqnd
By the assumption that our almost-complex structure on the $M$ component is generic, the right vertical arrow is a surjection, while by $\text{index}(D)\geq 0$, we deduce that the left (and hence middle) vertical arrow is also a surjection. 

So now let us consider the case when one of the boundary branes is $B(Y) \subset M \times T^*F$. As we saw in Remark~\ref{remark.Z-small}, we can guarantee regularity by simply allowing for perturbations in a small neighborhood near $u(\del S) \cap B(Y)$. On the other hand, by the boundary conditions, we know this neighborhood must non-trivially intersect the region $M \times [-w, w] \times [-D, \infty)$. This is precisely the region on which we allow our tangential perturbations $\{Z\}$ to have freedom when we do Floer theory with $B(Y)$.

Finally, in the case $B \subset M \times T^*E$ is an arbitrary conical brane, and not necessarily of the form $B(Y)$ for some cobordism $Y$, $J$ need not be a direct sum inside $M \times T^*(-w+\epsilon,w-\epsilon]$ so the same argument as in Remark~\ref{remark.Z-small} achieves transversality.
\end{proof}

\subsubsection{Compactness and regularity for $M \times T^*E^n \times T^*F^N$}\label{section.compact-cube}
This general case is similar to the previous section so we will be brief.

\begin{notation}
When we write $F_i$ above, we mean the $i$th component of some direct product $F^N$. Likewise, $E_i$ is the $i$th component of $E^n$. So for example, we have
	\eqnn
		F_1 \times F_2 \times \ldots F_N = F^N.
	\eqnd
Likewise, if an interval has a subscript
	\eqnn
		[-\un D, \ov D]_{F_i^\vee}
	\eqnd
we mean that this is an interval in the line $F_i^\vee$.
\end{notation}

We will again consider the following boundary branes: $X_i \times (\beta_i)^n \times (\gamma_i)^N$ for $i=0,\ldots,d-1$, and a conical $B \subset M \times T^*E^n$ or $B(Y) \subset M \times T^*E^n \times T^*F^N$ being the $d$th brane. We again choose almost complex structures such that:
	\enum
		\item When $B=B(Y)$, for every $1 \leq i \leq N$, we demand that
			\eqn\label{eqn.cube-J-split}
				J = J_{M \times T^*E^n \times T^*F^{N \setminus\{i\}}}
					\oplus J_{T^*F_i}
			\eqnd
		on the region
			\eqnn
			(M \times T^*E^n \times T^*F^{N\setminus\{i\}})\times (F_i \times [-\un{D},-\ov{D}]_{F_i^\vee}).
			\eqnd
		Here, 
			 	\begin{itemize}
					\item
						 $F^*F^{N\setminus\{i\}}$ denotes	the factor of $T^*F^N$ with the $i$th coordinate forgotten. 
				 	\item
						$J_{M \times T^*\RR^{m+N\setminus\{i\}}}$ is an eventually contact-type almost complex structure on $M \times T^*E^n \times T^*F^{N \setminus\{i\}}$. 
					\item
						$J_{T^*F_i}$ is the standard almost-complex structure (\ref{eqn.i}) on $T^*F_i$.
					\end{itemize}
			And of course, $J$ should be an eventually contact-type almost complex structure on $M \times T^*E^n \times T^*F^N$ itself.  Note that $M \times T^*E^n$ has an almost-complex structure as we now specify:
		\item When $B \subset M \times T^*E^n$, we assume $B$ is contained in the region $M \times T^*[-w,w]^n$.\footnote{The cube $[-w,w]^n$ is a proxy for the compact image in (\ref{eqn.compact-image}), or the $K$ mentioned in 1. of Section~\ref{section.overview-transversality}.}
		  Then we ask that $J$ be some eventually contact-type almost-complex structure on $M \times T^*E^n$. Finally, $J$ must satisfy the following: For any $1 \leq j \leq m$, on the region of $M \times T^*E^n$ whose projection to $T^*E_j$ has image
			\eqnn
			\left( (-\infty, -w-\epsilon] \coprod (w+\epsilon, \infty) \right) \times [-\un{D} , -\ov{D}]  
			\subset E_j \times E_j^\vee \cong T^*E_{j}
			\eqnd
		we ask that $J$ equal a direct sum $J_{M \times T^*E^{n-1}} \oplus (J_{T^*E_j})$ where $J_{T^*E_j}$ is the standard almost complex structure (\ref{eqn.i}) on $T^*E_j$.
	\enumd
All of these should be, outside a bounded region, translation invariant under the action of $E^n$ or of $E^n \times F^N$. 

\begin{remark}
A consequence of (\ref{eqn.cube-J-split}) is that on
	\eqnn
		M \times \left( F \times [-\un{D},\ov{D}]\right)^N,
	\eqnd
the almost complex structure splits
	\eqnn
		J = J_M \oplus (J_{T^*F})^N
	\eqnd
where $J_M$ is an arbitrary almost complex structure eventually of contact type, and $J_{T^*F}$ is the standard one (\ref{eqn.i}).
\end{remark}

Given this simple setup, compactness and regularity follow straightforwardly: We project onto the various factors $T^*E_i$ and $T^*F_i$, and re-apply the arguments for $M \times T^*\RR$.

\subsection{Some fine print (Dependence on $\cH, D, T, w$)}

\begin{remark}[Varying Hamiltonians]
The reader may worry that the class of Hamiltonians considered is not wide enough to achieve generic behavior---for instance, that $H$ is demanded to have no dependence on $T^*\RR$ in a large region may make it impossible to render chords between $B \subset M \times T^*\RR$ and $X_i \times \gamma_i$ non-degenerate. 

The first remark to make is that we never define a morphism between a brane $L \subset M \times T^*E^n \times T^*F^N$ and itself---the fact that we are considering modules over $\fuk(M)$ eliminates the need to define endormorphisms except when $n=N=0$, where $H_M$ can be chosen as generally as one wants within the usual class of eventually quadratic wrapping Hamiltonians. Indeed, because we still have freedom in how $H$ behaves in the $M$ component, we do not have to worry about defining $A_\infty$ operations taking place only in the manifold $M$; in particular, $\wrap$ is defined fine as usual. 

Second, we technically only define $\Xi$ on a subcategory of $\lag$---the subcategory consisting of those branes and cobordisms that are in general position with respect to $X_i$ and $X_i \times \beta_i^n \times \gamma_i^N$. Since an eventually linear Hamiltonian isotopy can easily be chosen between any brane (or cobordism) $Y$ not in general position to a $Y'$ that is in general position, the inclusion of this subcategory is in fact essentially surjective and fully faithful. So $\Xi$ is defined on all of $\lag$ via this equivalence.

Finally, note that if we were not to have freedom in our choice of $H_M$ in the $M$ component, one could still use the categorical trick in the previous paragraph, but this would mean we must restrict ourselves to defining $\Xi$ to have image in the category of modules over a countable collection of $X_i \in \fuk$. This is because finding $Y'$ which is generic with respect to a countable collection $(X_i \times \beta_i^n \times \gamma_i^N)_{i, n, N}$ is possible, but an uncountable collection is potentially impossible.
\end{remark}

\begin{remark}[Conical at $R$, or eventually conical]
All our branes are chosen to be eventually conical---in particular, there is no fixed $R$ and $A$ such that we require our branes to be collared along $\del M \times [R,\infty) \subset \del M \times [0,\infty) \cong M \setminus A$. Naively following our ``eventually collared'' definition actually leads to some problems: For instance, the space of eventually conical almost complex structures is not closed under sequences, as one can choose a sequence of $J_i$ that are conical further and further away from the skeleton. The obvious workaround/definition is to fix $A$, and to choose a Fukaya category $\wrap_R$ and a cobordism category $\lag_R$ all of whose objects and morphisms are demanded to be conical past collar distance $R$ from $A$. (This is equivalent to the usual definitions, which a priori chooses $M^{in} = A$ and $R=0$.) For $M$ with a compact skeleton, the inclusions
	\eqnn
	\wrap_R \to \wrap_{R'}, \qquad
	R \leq R'
	\eqnd
are essentially surjective after a finite $R$, and by changing the behavior of $H$ in the interior of $M$ (since $H_{R'}$ is collared further out than $H_R$), fully faithful as well. Thus the directed limit defines $\wrap \simeq \hocolim_{R \geq 0} \wrap_R$ unambiguously, and its objects and morphisms are precisely those given by the ``eventually conical'' definition.

The same remark applies for $\lag$: The sequence $(\lag_R)_{R>0}$ of objects and (higher) cobordisms collared past collar-distance $R$ from $A \times E^n \times F^N$ has honest union $\lag$ as a simplicial set. Of course, each map $\lag_R \to \lag_{R'}$ is a levelwise injection, hence a cofibration, and the hocolim is given by the honest union. In fact, one can even show that, if $\Lambda \subset M$ is compact, then $\lag_R \to \lag_{R'}$ is an equivalence for large enough $R$. This is because one can always find a cobordism to ``shrink'' a brane or cobordism collared past $R'$ to one collared along $R$.

So we lose no generality, in either the Fukaya or cobordism setting, to consider ``eventually conical'' objects and morphisms, rather than fixing an $R$ and $A$ once and for all. The same remark applies in defining the $a$ used to make $H$ and $J$ eventually translation-invariant (and to make $H$ have no $E^n \times F^N$-dependence outside $T^*(E^n \times F^N)|_K$). 
\end{remark}

\begin{remark}[Choices of $T, D$, and smoothing corners]
Usually one has to ``smooth corners'' to define a wrapped category for a product $M \times M'$, because $\del (M \times M') = (\del M \times M')\bigcup_{\del M \times \del M'} (M \times \del M')$ may not be smooth. Regardless of the smoothing, one has a Liouville completion equivalent to $M \times M'$. So the same strategy as the previous remark shows that the wrapped Fukaya category one obtains is unchanged.

One can think of $T$ and $\un{D}$ as data one needs to incorporate to choose a smoothing of $M \times T^*\RR$'s boundary. This is because choosing a smooth boundary also involves choosing which bit of $T^*\RR$ to think of as ``on the interior,'' and which bit as ``the conical end.'' 

Finally, the same remarks as before show that we can easily take the limit as $T \to \infty$ and $D, \ov{D}, \underline{D} \to \infty$ to incorporate all cobordisms and objects in $\lag$. This is how we end up with the functor $\Xi: \lag \to \fuk\Mod$. 

Note that the ability to take $D, \ov D, \un D$ larger and larger is especially important because there may be higher cobordisms that realize a higher morphism between cobordisms with different values of $D, \un D, \ov D$. 

Note also that the module maps associated to staircases $\gamma_i$ with depth at least $D$ are equivalent to the module maps associated to staircases $\gamma'_i$ with depth at least $D'$ by usual continuation map methods. We omit the details---essentially, by the same boundary-stripping arguments, one can see that the holomorphic disks behave as though the Hamiltonian isotopies taking $\gamma_i$ to $\gamma_i'$ are compactly supported in the relevant regions.
\end{remark}

\begin{remark}\label{remark.wrapping-non-compact}
The paper~\cite{abouzaid-seidel} (which introduced the wrapped category) and the subsequent paper~\cite{abouzaid-geometric} (which proposed an alternative construction) only define $WF$ when $M$ has a compact skeleton. In particular, $M \times T^*E^n$ does not fit into the framework present in these papers, as the standard Liouville form $pdq$ does not have a compact skeleton. The main obstacle is developing some ``horizontal Gromov compactness'' to ensure that disks cannot escape in the $E^n$ directions, as opposed to the cotangent directions of $E^n$. We have worked around this by defining $WF$ and its input perturbation data in a way where horizontal compactness is manifest by the maximum principle for $T^*E^n \cong \overline{\CC^m}$. Our definition for $WF$ when $m \geq 1$ is hence not {\em identical} to the definitions in loc. cit., but it is at least defined, and certainly satisfies the properties one would expect of $WF$. 
\end{remark}

% Activate the following line by filling in the right side. If for example the name of the root file is Main.tex, write
% "...root = Main.tex" if the chapter file is in the same directory, and "...root = ../Main.tex" if the chapter is in a subdirectory.
 
%!TEX root = _pairing.tex

\section{Definitions of $\lag_\Lambda(M)$ and $\wrap$}
Much of the data we discussed in the previous section will be irrelevant in defining $\lag_\Lambda(M)$. Namely, while we will care about the properties and structures on $M$ and $Y$, we will not need any mention of $H$ or $J$, which will only enter the picture when discussing Floer theory.

\subsection{Cubes}
Though the definition of an $\infty$-category is inherently simplicial, we will utilize (two!) cubical pictures in this paper to think about cobordisms. Let us describe how the reader can translate between the two cubical pictures and the usual simplicial structures.

The first picture is presented in Figure~\ref{figure.3-simplex}, which depicts how to think of an $n$-simplex for $n = 3$. For general $n$, one considers the set $[n] = \{0,1,\ldots, n\}$. Then we let $\cP_n$ denote the collection of all subsets of $[n]$ containing both $0$ and $n$. There is a natural bijection between the elements of $\cP_n$ and the vertices of a cube $I^{n-1}$. We let $C(n)$ denote the product $I^{n-1} \times [0,n]$, a rectangular prism of height $n$. On the edge $\{P\} \times [0,n]$ above a vertex corresponding to $P \in \cP_n$,  we put a marked point at height $i$ for each $i \in P$.

Note that the injective maps $\del_i: [n-1] \to [n]$, defined as those maps that do not hit $i \in [n]$, induce maps of cubes $C(n-1) \to C(n)$ in the obvious way; so if any object lives over $C(n)$, one can pull it back to an object living over $C(n-1)$. This defines the face maps for $\lag_\Lambda^0(M)$, where an $n$-simplex is a Lagrangian living over $C(n)$ allowing for smooth pullbacks along these faces of $C(n)$\footnote{A cobordism must also satisfy a non-characteristic condition, see Definition~\ref{defn.non-characteristic}.}.  This collaring is also inspired by the simplicial nerve construction; see Remark~\ref{remark.simplicial-nerve}.

\begin{figure}
		\[
			\xy
			\xyimport(8,8)(0,0){\includegraphics[width=2in]{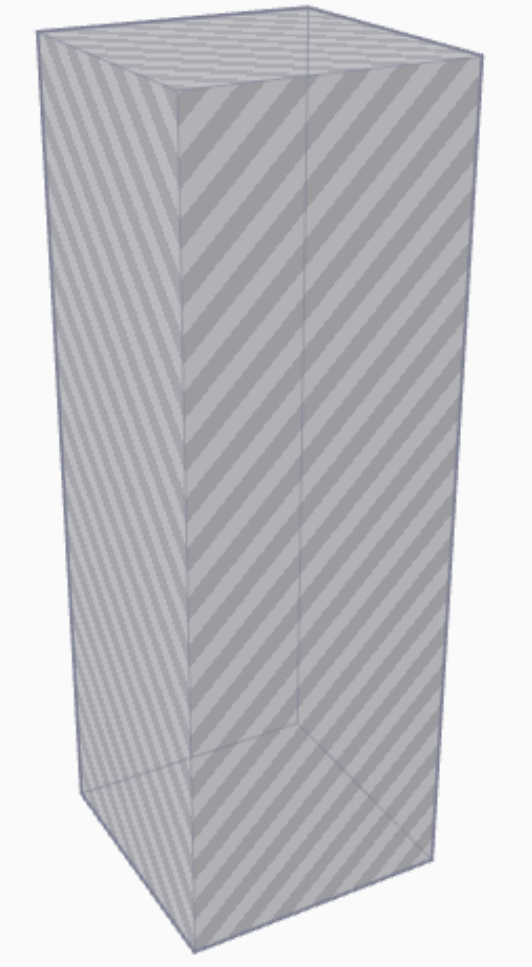}}
			%%Grid, for help.
%		 	,(2,1)*+{\bullet2},(3,1)*+{\bullet3},(4,1)*+{\bullet4},(5,1)*+{\bullet5},(6,1)*+{\bullet6},(7,1)*+{\bullet7},(1,1)*+{\bullet1},(1,2)*+{\bullet2},(1,3)*+{\bullet3},(1,4)*+{\bullet4},(1,5)*+{\bullet5},(1,6)*+{\bullet6},(1,7)*+{\bullet7}
				%%End Helping grid
			,(1,1.5)*+{0}
			,(3,0)*+{0}
			,(7,1)*+{0}
			,(4.8,2)*+{0}
			,(3,3)*+{\bullet 1}
			,(7,3.5)*+{\bullet 1}
			,(7.2,6)*+{\bullet 2}
			,(4.8,6.5)*+{\bullet 2}
			,(0.3,7.8)*+{3}
			,(2.8,7.3)*+{3}
			,(7,7.5)*+{3}
			,(4.4,8)*+{3}
			\endxy
			\xy
			\xyimport(8,8)(0,0){\includegraphics[width=2in]{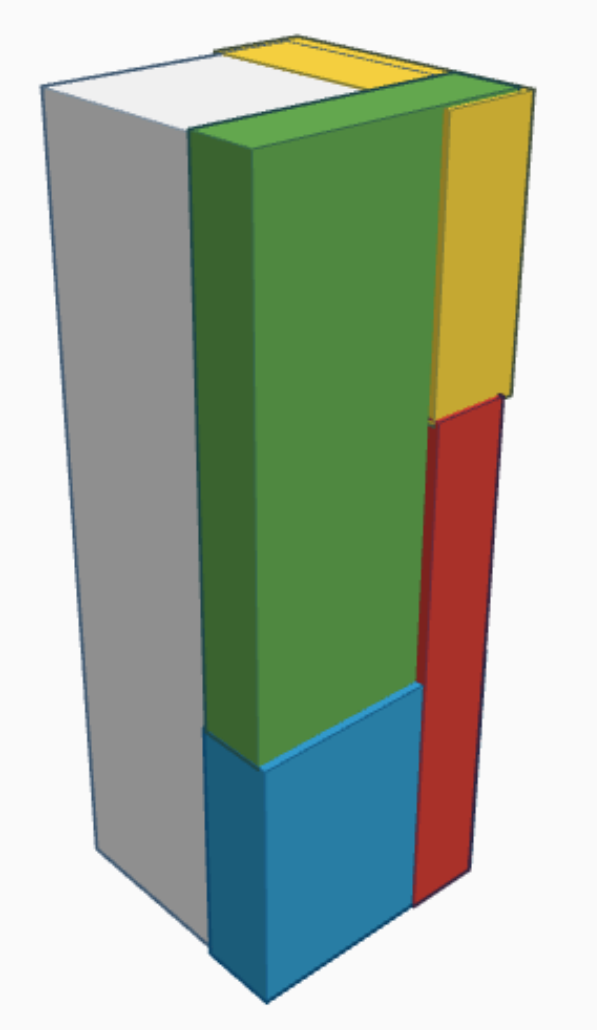}}
			%%Grid, for help.
%		 	,(2,1)*+{\bullet2},(3,1)*+{\bullet3},(4,1)*+{\bullet4},(5,1)*+{\bullet5},(6,1)*+{\bullet6},(7,1)*+{\bullet7},(1,1)*+{\bullet1},(1,2)*+{\bullet2},(1,3)*+{\bullet3},(1,4)*+{\bullet4},(1,5)*+{\bullet5},(1,6)*+{\bullet6},(1,7)*+{\bullet7}
				%%End Helping grid
			,(1,1.5)*+{0}
			,(3.5,0)*+{0}
			,(6.6,1)*+{0}
			,(3.5,2.2)*+{\bullet 1}
			,(6.3,2.7)*+{\bullet 1}
			,(6.7,5)*+{\bullet 2}
			,(0.3,7.5)*+{3}
			,(3.5,7)*+{3}
			,(7,7.4)*+{3}
			,(4.4,7.8)*+{3}
			\endxy
			\xy
			\xyimport(8,8)(0,0){\includegraphics[width=2in]{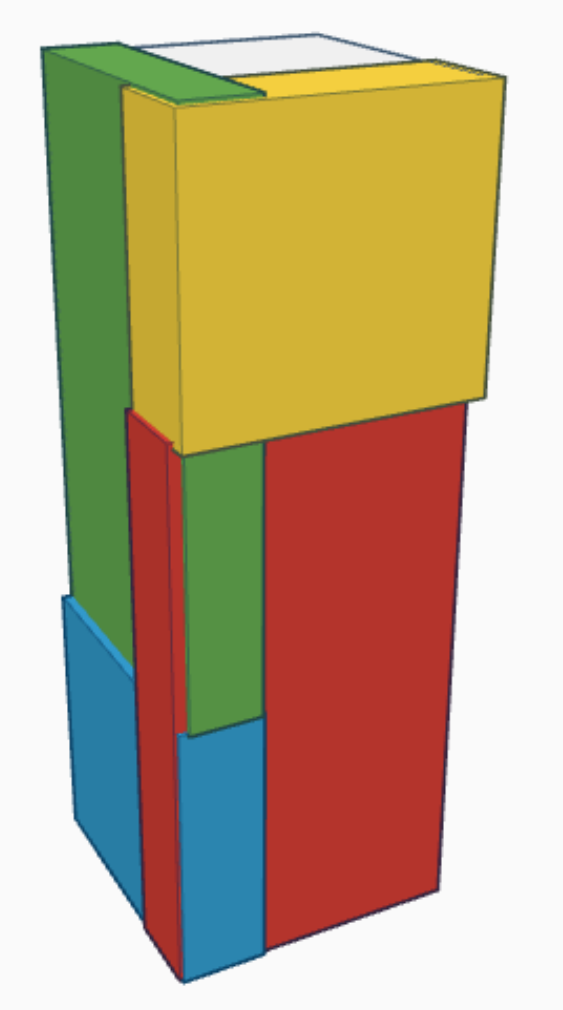}}
			%%Grid, for help.
%		 	,(2,1)*+{\bullet2},(3,1)*+{\bullet3},(4,1)*+{\bullet4},(5,1)*+{\bullet5},(6,1)*+{\bullet6},(7,1)*+{\bullet7},(1,1)*+{\bullet1},(1,2)*+{\bullet2},(1,3)*+{\bullet3},(1,4)*+{\bullet4},(1,5)*+{\bullet5},(1,6)*+{\bullet6},(1,7)*+{\bullet7}
				%%End Helping grid
			,(1,1.5)*+{0}
			,(2.7,0.1)*+{0}
			,(6.4,1)*+{0}
			,(1.2,3.2)*+{\bullet 1}
			,(3,2.3)*+{\bullet 1}
			,(2.9,4.5)*+{\bullet 2}
			,(6.6,4.9)*+{\bullet 2}
			,(0.5,7.7)*+{3}
			,(3,7.2)*+{3}
			,(7,7.5)*+{3}
			,(4.4,7.8)*+{3}
			\endxy
			\]
\begin{image}\label{figure.3-simplex}
The image of a 3-simplex $Y \subset M \times T^*F^3$, projected to $F^3$ and color-coded, viewed from three angles. The grey cube to the left shows where the cube's faces are collared by $Y_i$. In the blue region, $Y$ is collared by $Y_{01} \times [0,1]_{q_2}$. In the green region, $Y$ is collared by $Y_{123}$. In the red region, $Y$ is collared by $Y_{012}$. In the yellow region, $Y$ is collared by $Y_{23} \times [0,1]_{q_1}$. Note the collaring is compatible with the regions where colors overlap.
\end{image}
\end{figure}

The second picture is presented in Figure~\ref{figure.stout-cube}. This is the cubical picture that results by applying the $B$ construction on an $n$-simplex of $\lag_\Lambda^0(M)$, see Definition~\ref{defn.B}. Note that the cube $I^n$ has vertices in bijection with the collection of all $P \subset [n]$ containing $0$. Call this collection $\cP_0(n)$. There is a natural map $\cP_0(n) \to [n]$ sending $P$ to its maximal element, and this $\max(P)$ label the vertex of $I^n$ corresponding to $P$. Again, there are natural maps from $\cP_0[n-1] \to \cP_0[n]$ induced by the $\del_i$, which for $\del_0$ involves adjoining the element $0$. Applying $B$ to any simplex of $\lag_\Lambda(M)$ results in a Lagrangian living over $I^n$ collared by the $\max$ map, and with faces compatible with the $\del_i$ faces.

\begin{figure}
		\[
			\xy
			\xyimport(8,8)(0,0){\includegraphics[width=3in]{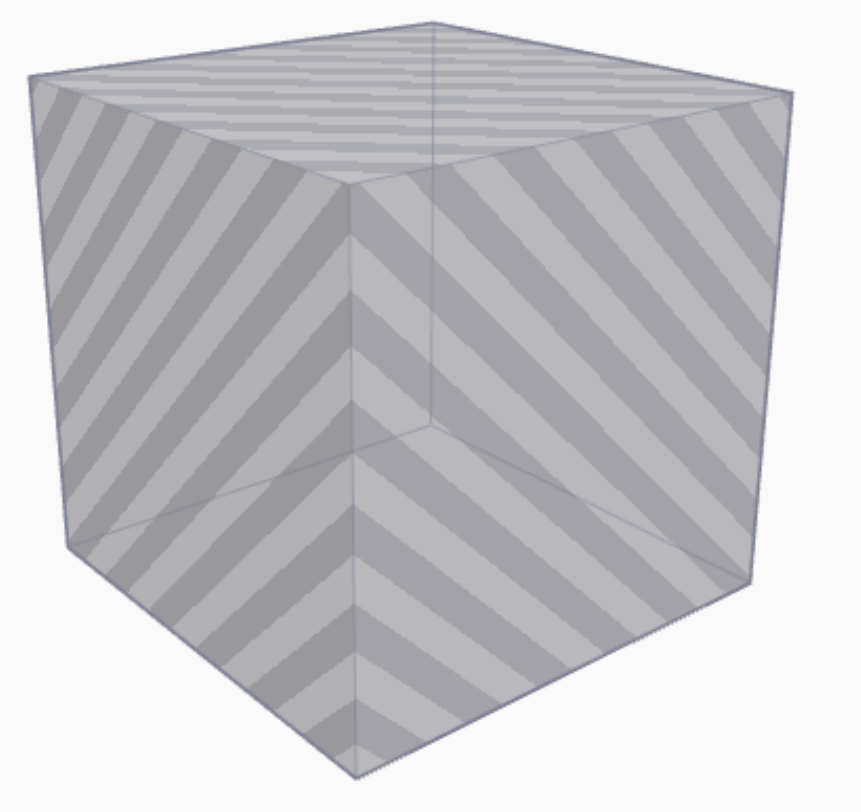}}
			%%Grid, for help.
%		 	,(2,1)*+{\bullet2},(3,1)*+{\bullet3},(4,1)*+{\bullet4},(5,1)*+{\bullet5},(6,1)*+{\bullet6},(7,1)*+{\bullet7},(1,1)*+{\bullet1},(1,2)*+{\bullet2},(1,3)*+{\bullet3},(1,4)*+{\bullet4},(1,5)*+{\bullet5},(1,6)*+{\bullet6},(1,7)*+{\bullet7}
				%%End Helping grid
			,(0.7,2.7)*+{0}
			,(3.2,0)*+{1}
			,(7,2)*+{2}
			,(4.2,3.8)*+{2}
			,(0.1,7.2)*+{3}
			,(4,8)*+{3}
			,(3.2,6.5)*+{3}
			,(7,7)*+{3}
			\endxy
			\]
\begin{image}\label{figure.stout-cube}
A cube whose vertices are labeled by the $\max$ map. This is how $B(Y)$ will be collared for $Y$ a 3-simplex. (See Figure~\ref{figure.B-3-simplex}.)
\end{image}
\end{figure}

\begin{remark}\label{remark.simplicial-nerve}
This first picture comes from a simplicial nerve construction. Briefly, the construction goes as follows:

For any $i\leq j \in [n]$, we let $P_{i,j}$ denote the collection of all subsets of $n$ whose minimal element is $i$ and whose maximal element is $j$. $P_{i,j}$ is a poset ordered by inclusion, hence a category, hence it makes sense to think of its nerve $N(P_{i,j})$, a simplicial set. We define $D^n$ to be a category enriched in simplicial sets, with objects $\ob D^n = [n]$, and $\hom_{D^n}(i,j) = N(P_{i,j})$. The composition $N(P_{i,j}) \times N(P_{j,k}) \to N(P_{i,k})$ is given by taking the union of subsets of $[n]$. Then, the cube $C(n)$ described above, with marked points as prescribed, is a picture of the space of morphisms from $0$ to $n$ in $D^n$, where each morphism is pictured as an edge, and each higher simplex is pictured as a higher-dimensional cube.

Thus, imagining for a moment that a cobordism category might be presented as a category enriched in simplicial sets, a functor from $D^n$ to this cobordism category is precisely given by a single cobordism living above $C(n)$, collared as indicated by the marked points.
\end{remark}

\subsection{Collaring and notation}\label{section.cube-notation}
The idea of collaring will be important for us. Collaring allows us to perform gluing constructions, and it will also allow us to analyze Floer moduli spaces more easily. (It is easy to characterize disks with boundary on product branes.) So we set some standing notation.

Throughout, we fix an integer $N \geq 1$, real numbers $w_i > 0$, and a prism
	\eqnn
	\cP = [0,w_1] \times [0,w_2] \times \ldots \times [0,w_{N-1}] \times [0,N].
	\eqnd

\begin{remark}
By scaling and translating in a rectilinear way, we will sometimes deal with prisms 
	\eqnn
	[-w_1,w_1] \times \ldots \times [-w_N, w_N] \subset F^N.
	\eqnd
This will make it easier to conform with our notations for $c$ and $\gamma_i$ from Section~\ref{section.geometry}.
\end{remark}

\begin{defn}[Collaring and $\del_i, \del_i^{\back}$]\label{defn.collared}
Given any Lagrangian
	\eqnn
	Y \subset M \times T^*\cP
	\eqnd
we say that it is {\em collared} if the following holds: 
	\begin{enumerate}
		\item
			For every $0 < i < N$, there is an $\epsilon_i>0$ for which
				\eqnn
					Y|_{[0,\epsilon_i]_{q_i}} 
					=
					\del_i^{\front} Y \times [0,\epsilon_i)_{q_i}
				\eqnd
			for some Lagrangian
				\eqnn
					\del_i^{\front} Y \subset M \times T^*F^{[N] \setminus\{0,i\}}.
				\eqnd
			Here, $[0,\epsilon_i)_{q_i}$ denote thes zero section of the cotangent bundles $T^*[0,\epsilon_i)_{q_i}$ in the $i$th component of $T^*F^N$. Also, as usual, the restriction notation means
				\eqnn
					Y|_{[0,\epsilon_i]_{q_i}}
					:=
					Y \cap \left(M \times T^*F^{[N] \setminus\{0,i\}} \times T^*[0,\epsilon_i]_{q_i}\right)
				\eqnd
		\item
			For every $0<i<N$, there is an $\epsilon_i>0$ for which
				\eqnn
					Y|_{[w_i-\epsilon_i, w_i]}
					=
					\del_i^{\back} Y \times [w_i-\epsilon_i,w_i]_{q_i}
				\eqnd
			for some  Lagrangian
				\eqnn
					\del_i^{\back} Y \subset M \times T^*F^{[N] \setminus\{0,i\}}.
				\eqnd	
	\end{enumerate}
\end{defn}

\begin{remark}
Of course, since the $\del_i^{\back} Y$ and $\del_i^{\front}Y$ are uniquely determined, the above also defines the notation $\del_i^{\front}$ and $\del_i^{\back}$. Sometimes, out of sloth, we will write
	\eqnn
	\del_i Y := \del_i^{\front} Y
	\eqnd
but we will never leave off the word ``back'' from our notation.
\end{remark}

Given an integer $0<i<N$, the $q$ coordinates of $\del_i Y$ are given by	
	\eqnn
		(q_1,\ldots,\widehat{q_i},\ldots,q_N)
	\eqnd
with $q_i$ omitted. But it will also be useful to think of $\del_i Y$ as living in another coordinate system.

\begin{defn}[$d_i$ of a brane] 
We write
	\eqnn
		d_i^{\front} Y \subset M \times T^*F^{N-1}
	\eqnd
to denote the pullback of $\del_i^{\front} Y$ along the map induced by the composition
	\eqnn
	F^{[N-1] \setminus \{0\}} \to F^{[N]\setminus\{0,i\}} \xra{q_i = 0} F^{[N] \setminus\{0\}}
	\eqnd
	\eqnn
	(q_1,\ldots,q_{N-1}) \mapsto (q_1, \ldots,q_{i-1}, 0,q_{i+1},\ldots,q_{N-1},q_i).
	\eqnd
(In other words, the $q_i$ coordinate is sent to the $q_N$ coordinate.) Likewise, we write
	\eqnn
		d_i^{\back}Y \subset M \times T^*F^{N-1}
	\eqnd
to denote the pullback of $\del_i^{\back} Y$ along the map induced by
	\eqnn
		F^{[N-1] \setminus\{0\}} \to F^{[N] \setminus\{0,i\}} \xra{w_i} F^{[N] \setminus \{0\}},
		\eqnd
		\eqnn
		(q_1,\ldots,q_{N-1}) \mapsto (q_1,\ldots,q_{i-1},w_i,q_{i+1},\ldots,q_{N-1}, q_i).
	\eqnd
\end{defn}

\subsection{Definition of $\lag_\Lambda(M)$}
\begin{defn}\label{defn.non-characteristic}
Let $M$ be an exact symplectic manifold with the data of $\theta$, satisfying the conditions of Section~\ref{section.M}. Then:
	\begin{itemize}
		\item
			A 0-simplex of $\lag_\Lambda^n(M)$ is a brane $Y \subset M \times T^*E^n$ satisfying the conditions of Section~\ref{section.branes}
		\item
			A 1-simplex of $\lag_\Lambda^n(M)$ is a Lagrangian $Y \subset M \times T^*E^n \times T^*F$ with all structures $\alpha_Y, f_Y, P_Y$ as in Section~\ref{section.branes}, satisfying the following constraints:
				\begin{itemize}
					\item 
						(Collaring.) There are real numbers $t_0\leq t_1 \in F$ such that 
							\eqn
							Y|_{t \leq t_0} = Y_0 \times (-\infty,t_0],
							\qquad
							Y|_{t \geq t_1} = Y_1 \times [t_1,\infty).
							\eqnd
						These equalities are as branes, so that the structures $\alpha_Y, f_Y, P_Y$ agree with those structures on $Y_i \times F$, where $F$ is given the trivial structures.
					\item
						($\Lambda$-non-characteristic.) There is some positive number $T \in \RR \cong F^\vee$ such that the intersection of $Y$ with
							\eqn
							\left( N_e(\Lambda) \times F \times (-\infty,-T] \right)  \subset M \times F \times F^\vee = M \times T^*F
							\eqnd
						is empty. In words, this means that if $Y$ ever goes off to $-\infty$ in the cotangent direction of $F$, it must do so in a way that is bounded away from $\Lambda$. Note that this is the same $e$ as we spoke of in Section~\ref{section.outline}, and in Remark~\ref{remark.e}. As before, $N_e(\Lambda)$ is an $e$-neighborhood of $\Lambda$. 
				\end{itemize}
		\item
			An $N$-simplex of $\lag_\Lambda^n(M)$ is a Lagrangian $Y \subset M \times T^*E^n \times T^*F^N$, with the structures $\alpha_Y, f_Y, P_Y$ as above, satisfying:
				\begin{itemize}
					\item (Collaring.) $Y$ is collared as in Definition~\ref{defn.collared}. Further, we demand that each $d_i^{front} Y$ is an $(N-1)$-simplex. Finally, the back faces $d_i^{back}$ must also satisfy the following condition:
						\eqn\label{eqn.back-leq-i}
							\del_i^{back}(Y)|_{(F_1)_{\leq i}} = d_{\{0,\ldots,i\}} Y
						\eqnd
					and
						\eqn\label{eqn.back-geq-i}
							\del_i^{back}(Y)_{(F_1)_{\geq i}} = d_{\{i,\ldots,N\}}(Y).
						\eqnd
					See Remark~\ref{remark.back-collaring} below.
					\item	($\Lambda$-non-characteristic.) There is some positive number $T \in \RR \cong F_1^\vee$ such that the intersection of $Y$ with
							\eqn
							\left( N_e(\Lambda) \times F \times (-\infty,-T] \right) \times F^{N-1}  \subset M \times F \times F^\vee \times T^*F^{N-1}
							\eqnd
						is empty. 
				\end{itemize}
	\end{itemize}
\end{defn}

\begin{remark}\label{remark.back-collaring}
We expound on the collaring for the back faces of $Y$. Given the brane collaring $Y$ along the back $i$th face, one can look at the portion of $\del_i^{\back} (Y)$ at time $\leq i$. On the other hand, one can take successive face maps of $Y$ $d_N, d_{N-1}, \ldots, d_{i+1}$ to obtain a brane collared by objects $Y_0$ through $Y_i$. (~\ref{eqn.back-leq-i}) requires these two branes to be the same. Likewise, one could look at the portion of $\del_i^{\back}(Y)$ at time $\geq i$, or take successive face maps $d_0, \ldots, d_0 Y$. (\ref{eqn.back-geq-i}) says these two are equal. This back-collaring ensures, among other things, that there is no extra ``face data'' that one needs to keep track of other than the obvious faces of $Y$.
\end{remark}

\begin{remark}[$\Lambda$-non-characteristic for $n>0$]
Being $\Lambda$-non-characteristic for $\lag^{\dd n}_\Lambda(M)$ is the same thing as being $(\Lambda \times F^n)$-non-characteristic for $\lag^{\dd 0}_{\Lambda times F^n}(M \times T^*F^n).$
\end{remark}

\begin{defn}[Depth of a cobordism]\label{defn.depth}
Let $Y \subset M \times T^*E^n \times T^*F^N$ be a (higher) cobordism. Its {\em depth} is the minimal $T$ in the non-characteristic condition. 
\end{defn}

\begin{defn}[Width]\label{defn.width}
Its {\em width} is the minimal value of $t_1 - t_0$ among all the $t_0, t_1$ satisfying the collaring condition.
\end{defn}

\begin{remark}
Later, we will define a new Lagrangian $B(Y)$ out of the data of $Y$. When $\gamma$ is a tunnel curve of depth greater than the depth of $Y$, one can enumerate the intersections of $B(Y) \cap X \times \gamma$ as in (\ref{eqn.intersections-N}).
\end{remark}

\begin{defn}\label{defn.Y_I}
Let $Y$ be an $N$-simplex. Then for any subset $J \subset [N]$, we let $Y_J$ denote the $(|J|-1)$-simplex given by the $J$th face of the cobordism given by $Y$. Specifically, recall that the vertices of the cube $I^{n-1}$ are in bijection with subsets of $[N]$ containing $0$ and $N$. Then consider the face $\del_J I^{n-1}$ of $I^{n-1}$ spanned by those subsets contained in $J$. Then $Y_J$ is the Lagrangian cobordism collaring the prism
	\eqnn
	\del_J I^{n-1} \times [J_{\min} , J_{\max}].
	\eqnd
When $i \in [N]$, we will also write $Y_i$ to mean the object (i.e., brane) collaring the $i$th vertex of $Y$.
\end{defn}

\subsection{$\lag$ does not depend on $e$}
For the following lemma and its proof only, let us make explicit the dependence on $e$ by writing
	\eqnn
	\lag^e(M)
	\eqnd
to denote the sub-$\infty$-category of $\lag(M)$ whose morphisms all avoid $N_e(\Lambda)$.

\begin{lemma}\label{lemma.epsilon-lag}
For any choice of $e$ small enough, the $\infty$-categories $\lag(M)$ are all equivalent. More precisely, for any $e < e'$ small enough, the inclusion
	\eqn
	\lag^e(M) \to 
	\lag^{e'}(M)
	\eqnd
is an equivalence.
\end{lemma}

\begin{example}
The cobordism $Y \times F \subset M \times T^*F$ is the identity cobordism. We will see in Proposition~\ref{prop.id} below that it does indeed induce the identity natural transformation from $CW^*(-,Y)$ to itself.
\end{example}

\begin{example}
Fix a curve $s \subset T^*F$ which, outside of a compact set, equals 
	\eqnn
		(-\infty, -w] \times \{0\} \coprod \{w\} \times [0,\infty) \subset F \times F^\vee \cong T^*F.
	\eqnd
Assume further that $s$ can be equipped with a primitive which identically equals zero outside the same compact set. Then for any $Y \in \ob \lag^n$, $Y \times s$ is the zero morphism from $Y$ to the empty manifold (which is the zero object of $\lag$). For obvious reasons, $\Xi$ sends this to the zero natural transformation.
\end{example}

\begin{example}
Let $H_t$ be a possibly time-dependent Hamiltonian on $M \times T^*E^n$. Assume it is compactly supported in time, so that $H_t = 0$ for $t<<0$ and $t>>0$. Then for any object $Y \in \lag^{\dd n}$, the {\em Hamiltonian suspension} is a Lagrangian cobordism given by
	\eqnn
		(\phi^H_t(y), t, H_t(\phi^H_t(y))) \in (M \times T^*E^n) \times F \times F^\vee. 
	\eqnd
If $H_t$ is eventually linear, one can verify that this cobordism is eventually conical as a submanifold of $M \times T^*E^n \times T^*F$. 

Finally, if $H_t$ itself is compactly supported, then the suspension is bounded in the $F^\vee$ coordinate. It was proven in~\cite{nadler-tanaka} that such a cobordism is always an equivalence in $\lag$.
\end{example}

\subsection{The wrapped Fukaya category}\label{section.wrapped}
We assume the reader is familiar with the foundational papers---~\cite{abouzaid-seidel} where a telescoping construction was used, and~\cite{abouzaid-geometric} where eventually quadratic wrapping was used. As far as we know, there is no proof in the literature that both definitions agree, though they do produce equivalent answers in known examples. Regardless, we do not specify the model we use in the present work, as both models fit into our framework. The only thing to be said is that, indeed, the monotone homotopy technique utilized in defining the continuation maps (in the telescoping construction) still work with our choices of $\cH$ and $\cJ$.

We will also make use of two results:

\subsubsection{$\wrap$ does not depend on choice of compact region}
We first discuss the case when $M^{in}$ (and hence $\Lambda$) can be chosen compact.

One might notice a slight difference between the description of $M$ presented in Section~\ref{section.M}, and that presented in most other works, including~\cite{abouzaid-geometric}. In most other works, one chooses a specific compact set $M^{in} \subset M$ whose boundary is a contact manifold denoted $\del M$; one then takes $M$ to be the completion of $M^{in}$ given by attaching a symplectization of $\del M$ to $M^{in}$. In contrast, we begin with $M$ itself, which could be the completion of many choices of $M^{in}$.

This is not a purely artificial difference---for instance, after fixing a wrapping Hamiltonian $H$, one needs to specify the class $\cJ$ of almost-complex structures compatible with $H$ and the exact structure  of $M$. We would like to make sure that the collection $\cJ$ is a complete metric space. But the sequence of eventually conical $J$ is not a complete metric space---one can construct a sequence of eventually conical almost-complex structures that does not converge to an eventually conical one (by letting $J_i$ be conical outside bigger and bigger sets). One solution is to fix a compact set $M^{in}$ to control the behavior of $J \in \cJ$ uniformly, demanding that they be conical outside $M^{in}$. Then one has the usual compactness and transversality results, thereby defining a wrapped Fukaya category $\wrap(M^{in})$ which a priori depends on the choice of $M^{in}$.

But this choice of $M^{in}$ is not so important. Let $M^{in}_i$ be an increasing and exhausting sequence of compact regions in $M$, equipped with a Liouville equivalence between $M$ and the completion of $M^{in}_i$. Then the induced functor
	\eqnn
	\wrap(M^{in}_i) \to \wrap(M^{in}_{j}),
	\qquad
	i\leq j
	\eqnd
is essentially surjective, as one can flow $M^{in}_{j}$ inward to collar its branes with respect to $M^{in}_i$. (Hence any object in $\wrap(M^{in}_j)$ is equivalent to an object from $\wrap(M^{in}_i)$.) The functor is also fully faithful, as the class of perturbation data allowed in $\wrap(M^{in}_i)$ are automatically allowed in $\wrap(M^{in}_{j})$. Thus one has a sequential diagram of equivalences
	\eqnn
	\ldots \to \wrap(M^{in}_i) \to \wrap(M^{in}_{i+1}) \to \ldots
	\eqnd
whose (homotopy) colimit is a single, well-defined $A_\infty$ category equivalent to any $A_\infty$-category in the sequence. Note also that when $\Lambda \subset M$ is compact, the poset of all possible $M^{in}$ (ordered by inclusion) has a cofinal poset given by any increasing, exhausting sequence. Thus ``the'' wrapped category of $M$ (as opposed to that of $M^{in}$) is defined unambiguously. This is the approach we take.

\subsubsection{$\wrap_\compact$ does not depend on $e$}
For the following lemma and its proof only, let us make explicit the dependence on $e$ by writing
	\eqnn
	\wrap_\compact^e(M)
	\eqnd
to denote the full subcategory of $\wrap(M)$ given by those compact branes that are within $e$ of $\Lambda$.

The following is obvious by flowing via the Liouville flow:

\begin{lemma}\label{lemma.epsilon-wrap}
For any choice of $e$, the $A_\infty$-categories $\wrap_\compact(M)$ are all equivalent. More precisely, for any $e < e'$, the inclusion
	\eqn
	\wrap_\compact^e(M) \to 
	\wrap_\compact^{e'}(M)
	\eqnd
is an equivalence.
\end{lemma}

% Activate the following line by filling in the right side. If for example the name of the root file is Main.tex, write
% "...root = Main.tex" if the chapter file is in the same directory, and "...root = ../Main.tex" if the chapter is in a subdirectory.
 
%!TEX root = _pairing.tex
\section{The $B$ construction}\label{section.B}

As usual, if $S$ is a set, we let $F^S$ denote the set of all functions $S \to F$. In particular, note there is an identification
	\eqnn
		F^N = F^{[N] \setminus \{0\}}.
	\eqnd
Also, if one replaces $M$ by $M \times T^*E^n$, nothing in this section changes. So for simplices in $\lag^{\dd n}$, the reader should just replace every instance of $M$ with ``$M \times T^*E^n$'' to define the $B$ construction, which acts by the identity on the $T^*E^n$ component.

\subsection{Definition}
Throughout the definition, we fix tiny, positive numbers $\epsilon$ and $\delta_{\collar}$. We also fix some integer $N>0$; throughout, $i$ refers to some integer between $0$ and $N$. We also fix numbers $w_i > 0$.

For fixed $0<i<N$, consider the open rectangle
	\eqnn
		U_i :=
			\{
			(\alpha,r) \, | \, \text{ $1 < \alpha < 2$ and $0 < r < i - 2\epsilon)$}
			\}
		\subset F^2
		.
	\eqnd
Fix also an open embedding
	\eqnn
		\phi: U_i \into F^2,
		\qquad
		\phi = (\phi_1,\phi_2)
	\eqnd
satisfying the following conditions:
	\begin{enumerate}
		\item\label{item.B1}
			We have
				\eqnn
  				\phi(\alpha,r)
  				=
  				\begin{cases}
						(\alpha,\epsilon+r) & \alpha \leq 1 + \delta_{\collar} \\
						(w_i+r+\epsilon, 2-\alpha) & \alpha \geq 2 - \delta_{\collar}
					\end{cases}.
				\eqnd
		\item
			$
				{ \frac {\del \phi_1} {\del \alpha}} > 0 
			$ 
			when $\alpha < 2 - \delta_{\collar}$,
		\item
			$ {\frac {\del \phi_2} {\del \alpha}} < 0$ when $\alpha > 1 + \delta_{\collar}$, and
		\item
			${\frac {\del \phi_1} {\del r}}, {\frac {\del \phi_2}{\del r}} >0$.
	\end{enumerate}
See Figure~\ref{figure.phi}.

\begin{figure}
		\[
			\xy
			\xyimport(8,8)(0,0){\includegraphics[width=4in]{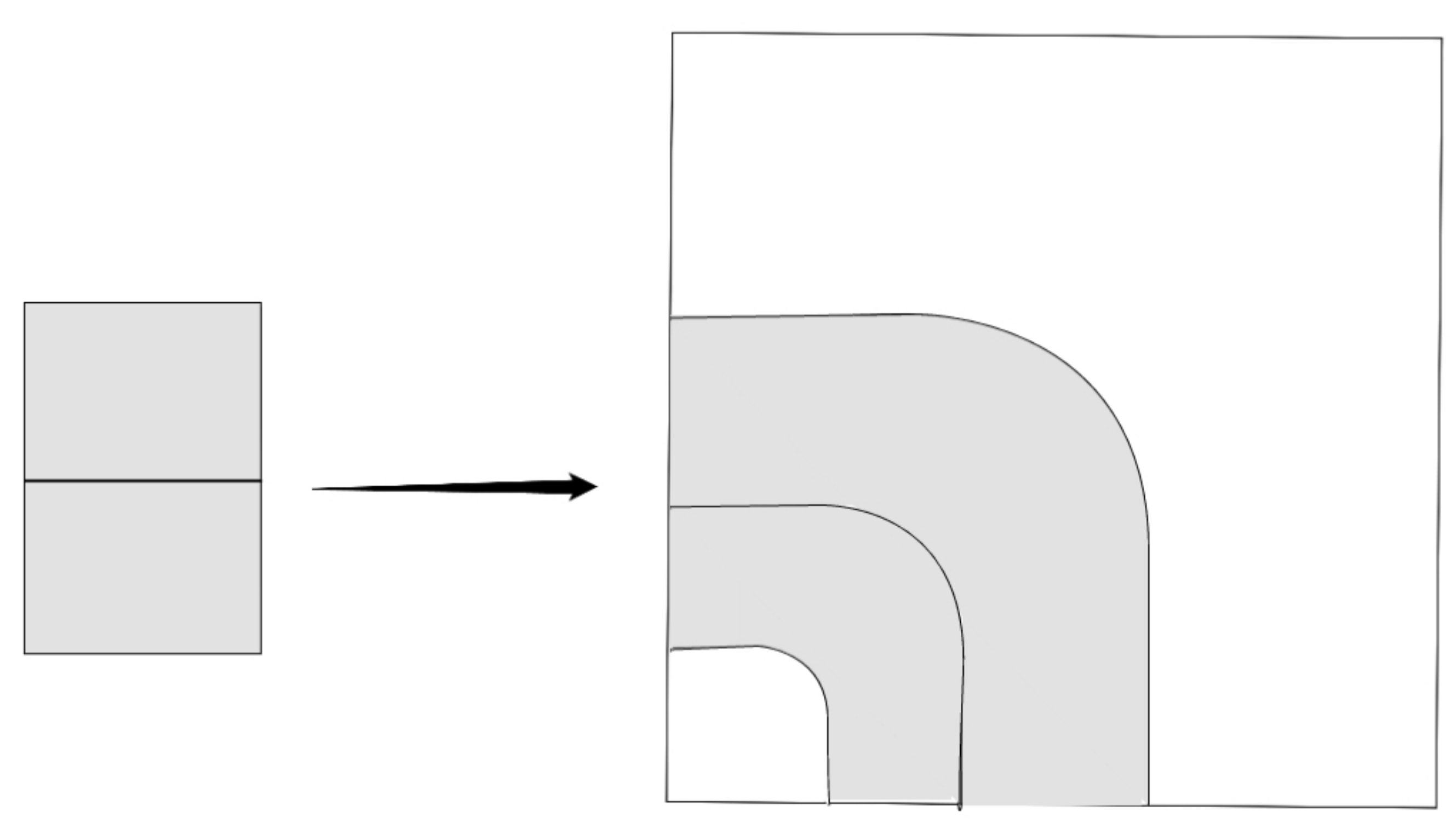}}
			%%Grid, for help.
%		 	,(2,1)*+{\bullet2},(3,1)*+{\bullet3},(4,1)*+{\bullet4},(5,1)*+{\bullet5},(6,1)*+{\bullet6},(7,1)*+{\bullet7},(1,1)*+{\bullet1},(1,2)*+{\bullet2},(1,3)*+{\bullet3},(1,4)*+{\bullet4},(1,5)*+{\bullet5},(1,6)*+{\bullet6},(1,7)*+{\bullet7}
				%%End Helping grid
			,(1.6,1.7)*+{\alpha}
			,(0,5.4)*+{r}
			,(8,0)*+{q_i}
			,(3.4,8)*+{q_N}
			,(2.4,3)*+{\phi}
			,(4.2,1)*+{V_i^{small}}
			,(7.3,6.8)*+{V_i^{big}}
			\endxy
			\]
\begin{image}\label{figure.phi}
The open embedding $\phi: U_i \to F^2$.
The horizontal lines on the  left rectangle indicate lines of constant $r$. They are mapped to the curves inside the rectangle on the right.
\end{image}
\end{figure}

\begin{remark}
There is a different map $\phi$ for each choice of $0<i<N$. We suppress the dependence from the notation.
\end{remark}

Note that the image $\phi(U_i)$ divides the open square $(w_i,w_i+i) \times (0,i) \subset F^2$ into two connected regions: $V_i^{big}$ and $V_i^{small}$, where $V_i^{small} \subset F^2$ is the region intersecting a small neighborhood of the point $(w_i,0) \in F^2$. 

Given a brane $Y \subset M \times T^*F^N$, let $\del_i Y$ denote the brane collaring the face at $q_i = 1$. Consider the Lagrangian
	\eqnn
		\del_i Y \times (1,2) \subset (M \times T^*F^{[N] \setminus\{0,i\}}) \times T^*(1,2).
	\eqnd
The open embedding $\phi$ induces an open embedding of symplectic manifolds
	\eqnn
		\Phi: (M \times T^*F^{[N] \setminus\{0,i,N\}}) \times T^*(0,i-2\epsilon) \times T^*(1,2) \into M \times T^*F^N
	\eqnd
as follows: $\Phi$ is the identity on the $M$ component. On the other factors, one applies the symplectic embedding induced by the open embedding
	\eqn\label{eqn.Phi}
		(q_1,\ldots,\widehat{q_i},\ldots,q_{N-1},r, \alpha)
		\mapsto
		\left(q_1,\ldots,q_{i-1},\phi_1(\alpha,r),q_{i+1},\ldots,q_{N-1},\phi_2(\alpha,r)\right).
	\eqnd
Now we construct a new manifold $B_i(Y) \subset M \times T^*F^N$ by the union
	\begin{align}
		B_i(Y) \label{eqn.B_i}
			:=
			&Y \\
			&\bigcup
			\left(
				V_i^{small} \times (\del_i Y)|_{q_N = 0}
			\right) \label{eqn.B1}
			\\
			&\bigcup
			\Phi(\del_i Y \times (1,2))\label{eqn.B2}
			\\
			&
			\bigcup
			\left(
				V_i^{big} \times (\del _i Y)|_{q_N = i}
			\right)\label{eqn.B3}
			\\
			&
			\bigcup
			\left(
				[w_i,w_i+i]_{q_i} \times (\del_i Y)|_{[i,N]_{q_N}}
				\right).\label{eqn.B4}
	\end{align}
Let us parse this formula. First, the portion of $B_i(Y)$ with $q_i \leq 1$ is identical to the brane $Y$ with which we began. The meat is in the later lines of the formula.

(\ref{eqn.B1}) says to take $V_i^{small}$, which is an open subset of the zero section of $F_i \times F_N$, and take its direct product with whatever is collaring $Y$ along where $q_N=0$ and $q_i=w_i$. For instance, if $Y$ represents an $N$-simplex in $\lag^{\dd 0}$, $(\del_i Y)|_{q_N=0}$ is simply a copy of $Y_0 \times F^{[N] \setminus\{0,N,i\}}$---that is, a direct product of the $0$th object $Y_0$ along with the zero section in the remaining directions.

(\ref{eqn.B2}) says to apply $\Phi$ to the brane obtained by taking $\del_i Y$ and extending it along the zero section in the $\alpha$ direction.

(\ref{eqn.B3}) says to consider the direct product of $V_i^{big} \subset F_i \times F_N$ with whatever collars $Y$ along $q_i=0$ and $q_N=i$. If $Y$ is an $N$-simplex, this simply means to take the product of $Y_i$ with the zero section---here, $Y_i$ is the $i$th object collaring $Y$.

(\ref{eqn.B4}) says to take whatever collars $Y$ along the face $q_i=w_i$ with $q_N$ in the interval $[i,N]$, and then extend it along the zero section of $T^*[w_i,w_i+i]$ in the $q_i$ direction. If $Y$ is an $N$-simplex, this $(\del_i Y)_{[i,N]_{q_N}}$ would be given by the brane
	\eqnn
		F^{ \{1,\ldots,i-1\}} \times \del_{0}\ldots\del_0 \del_0 Y.
	\eqnd
I.e., one takes the 0th face of $Y$, then the $0$th face of that, $i$ times, and then takes its direct product with the zero section in the first $(i-1)$ coordinates.

\begin{defn}\label{defn.B'}
Given an $N$-simplex $Y$, which we think of as living over the prism
	\eqnn
	[0,1]^{N-1} \times [0,N] \subset F^N,
	\eqnd
we let $B'(Y)$ equal the Lagrangian
	\eqnn
		B'(Y) := B_{N-1}( \ldots (B_2(B_1(Y))) \ldots ).
	\eqnd
This is a cube living over the prism
	\eqnn
		[0,1]_{q_1} \times [0,2]_{q_2} \times \ldots [0,N]_{q_N} \subset F^N.
	\eqnd
Note that when we apply $B_i$, $w_i = 1$. 
See Figures~\ref{figure.B-2-simplex} and~\ref{figure.B-3-simplex}. \end{defn}

Now we define $B$ itself. Note that if $Y$ is an $N$-simplex, then by condition~\ref{item.B1} of $\phi$, we know that $B'(Y)$ is collared along each of its faces. From hereon, to conform to the notation and coordinates set in Section~\ref{section.geometry}, we will scale and translate the prism
	\eqnn
		[0,1] \times [0,2] \times \ldots [0,N] \subset F^N
	\eqnd
so that $B'(Y)$ is now contained above
	\eqnn
		[-w_1,w_1] \times \ldots \times [-w_N,w_N] \subset F^N.
	\eqnd
Let $c_i \subset T^*F_i$ be the curve defined in Section~\ref{section.cone-tails}, whose tails we called cone tails, where we take $w_i$ to replace the $w$ in that section. Then we define $B(Y)$ by ``bending'' $B'(Y)$ in the way dictated by the brane
	\eqnn
	c_1 \times \ldots \times c_N \subset T^*\left(
		[-w_1,w_1] \times \ldots \times [-w_N,w_N]
	\right).
	\eqnd
To be more precise, we know that along the faces, $B'(Y)$ is collared in such a way that for any subset $0 \not \in K \subset [N]$, 
	\eqnn
		B'(Y) |_{q_i \in [-w_i,-w_i+\epsilon], i \in K} = \del_K B'(Y) \times \prod_{i \in K} [-w_i,-w_i+\epsilon].
	\eqnd
That is, $B'(Y)$ is a product with the zero section in the $F^K$ directions. Then we can replace this portion of $B'(Y)$ with the product brane
	\eqnn
	\del_K B'(Y) \times \prod_{i \in K} (c_i)|_{[-w_i,-w_i+\epsilon]}
	\eqnd
where for each $i \in K$, we take the portion of the cone tail $c_i$ living above the interval $[-w_i, -w_i + \epsilon]$ in $F_i$. Here, $\epsilon$ is chosen so that $c_i$ equals the zero section near $-w_i + \epsilon$, so this gluing is smooth. Likewise, we know that
	\eqnn
		B'(Y) |_{q_i \in [w_i-\epsilon,w_i], i \in K} = \del_K^{\back} B'(Y) \times \prod_{i \in K} [w_i-\epsilon,w_i].
	\eqnd
So we can again replace this portion of $B'(Y)$ by gluing in a copy of 
	\eqnn
	\del_K^{\back} B'(Y) \times \prod_{i \in K} (c_i)|_{[w_i-\epsilon,w_i]}.
	\eqnd
Summing up:

\begin{defn}[B]\label{defn.B}
Given $Y$ an $N$-simplex, we let
	\eqnn
	B(Y) \subset M \times T^*\left(
		\prod_{i=1,\ldots,N} [-w_i,w_i]
	\right)
	\eqnd
denote the brane whose underlying Lagrangian submanifold is given by the set
	\begin{align}
	B(Y) :=&
		B'(Y)|_{\prod_{i=1,\ldots,N}[-w_i+\epsilon,w_i-\epsilon]}
		\\ &
		\bigcup
		\left(
			\bigcup_{0 \not \in K \subset [N]}
			\del_K^{\back} B'(Y) \times \prod_{i \in K} (c_i)|_{[w_i-\epsilon,w_i]}
		\right)
		\\ &\bigcup
		\left(
			\bigcup_{0 \not \in K \subset [N]}
			\del_K B'(Y) \times \prod_{i \in K} (c_i)|_{[-w_i, -w_i+\epsilon]}
		\right).
	\end{align}
\end{defn}

\begin{figure}
		\[
			\xy
			\xyimport(8,8)(0,0){\includegraphics[width=3in]{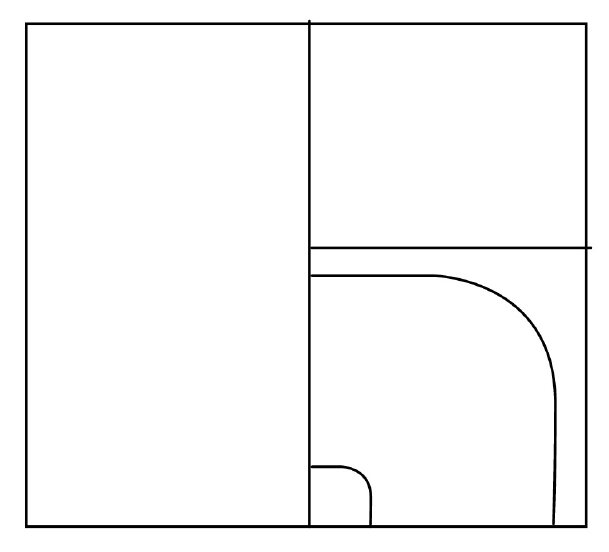}}
			%%Grid, for help.
%		 	,(2,1)*+{\bullet2},(3,1)*+{\bullet3},(4,1)*+{\bullet4},(5,1)*+{\bullet5},(6,1)*+{\bullet6},(7,1)*+{\bullet7},(1,1)*+{\bullet1},(1,2)*+{\bullet2},(1,3)*+{\bullet3},(1,4)*+{\bullet4},(1,5)*+{\bullet5},(1,6)*+{\bullet6},(1,7)*+{\bullet7}
				%%End Helping grid
			,(4.4,0.8)*+{Y_0}
			,(7.3,4)*+{Y_1}
			,(5.8,6)*+{Y_{12} \times [0,1]_{q_1}}
			,(0,0)*+{Y_0}
			,(8,0)*+{Y_1}
			,(8,8)*+{Y_2}
			,(0,8)*+{Y_2}
			,(5.4,3)*+{Y_{01} \text{ rotated} }
			,(2.4,4)*+{Y}
			\endxy
			\]
\begin{image}\label{figure.B-2-simplex}
An image of $B'(Y)$ for $Y$ a 2-simplex. This is a drawing in $F^2$, with the $F^\vee$ components omitted. 
\end{image}
\end{figure}

\begin{figure}
		\[
			\xy
			\xyimport(8,8)(0,0){\includegraphics[width=2in]{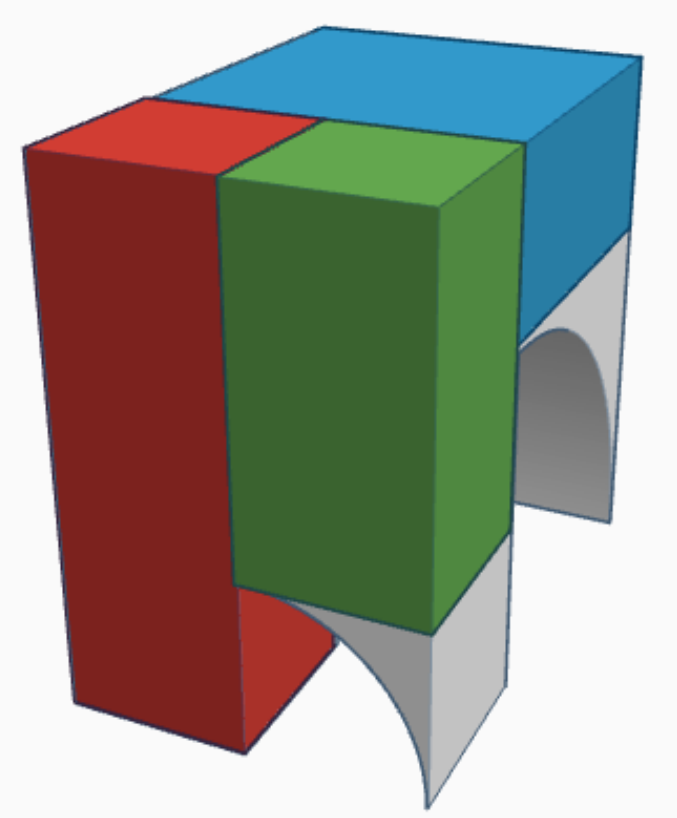}}
			%%Grid, for help.
%		 	,(2,1)*+{\bullet2},(3,1)*+{\bullet3},(4,1)*+{\bullet4},(5,1)*+{\bullet5},(6,1)*+{\bullet6},(7,1)*+{\bullet7},(1,1)*+{\bullet1},(1,2)*+{\bullet2},(1,3)*+{\bullet3},(1,4)*+{\bullet4},(1,5)*+{\bullet5},(1,6)*+{\bullet6},(1,7)*+{\bullet7}
				%%End Helping grid
			,(4,1)*+{\Phi(Y_{01})}
			,(7,3)*+{\Phi(Y_{12})}
			\endxy
			\xy
			\xyimport(8,8)(0,0){\includegraphics[width=2in]{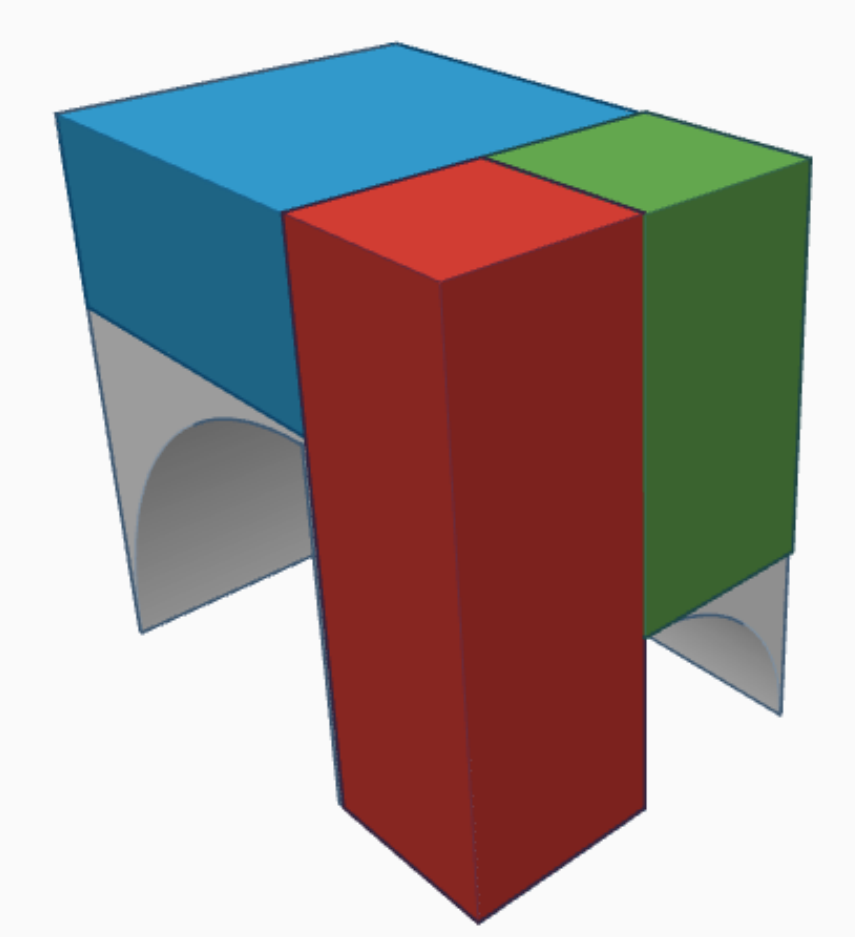}}
			\endxy
			\xy
			\xyimport(8,8)(0,0){\includegraphics[width=2in]{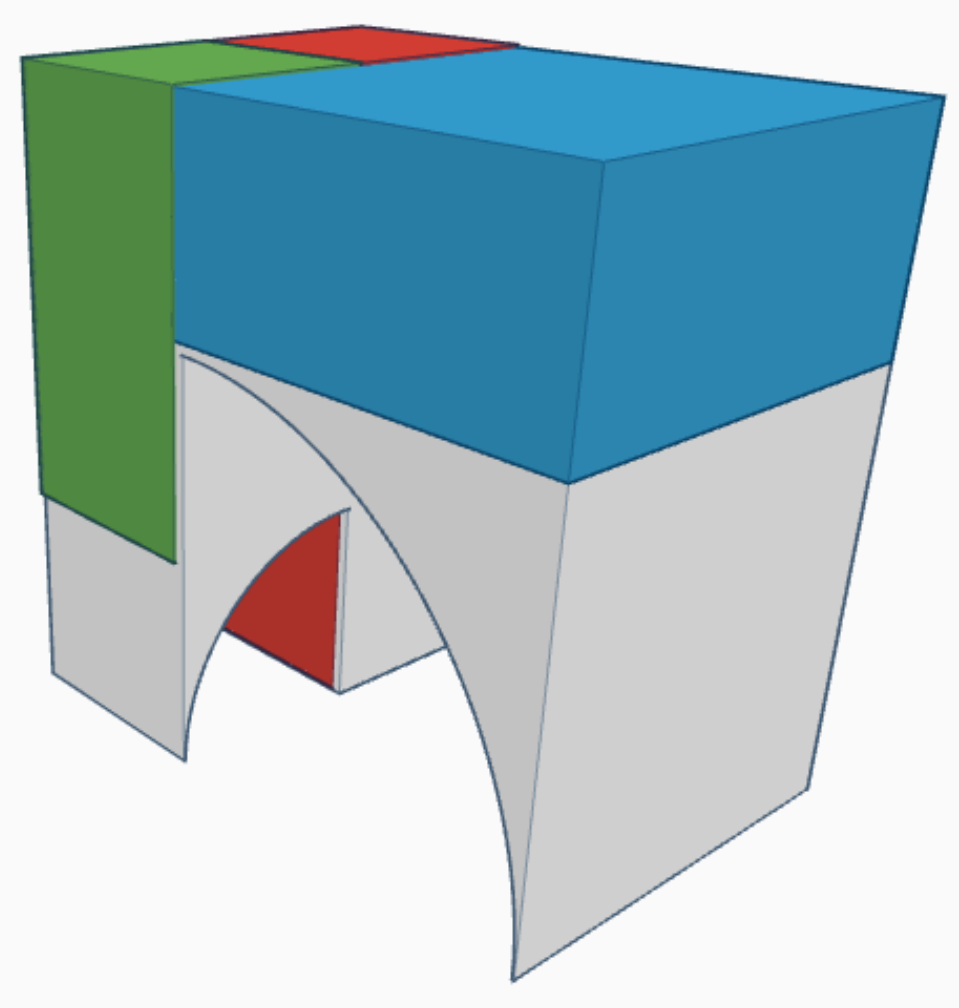}}
			\endxy
			\]
\begin{image}\label{figure.B-3-simplex}
An image of $B'(Y)$ for $Y$ a 3-simplex, from three different angles. The images of the $\Phi$ are carved out to improve visibility. The red prism is the original $Y$, unaltered. The top of $B'(Y)$ (the face shaded by three different colors) is where $q_N = q_3 = 3$, and is collared by $Y_3 \times F^2_{q_1,q_2}$. The green cube is a copy of $Y_{123} \times [1,2]_{q_1}$. The blue cube is a copy of $Y_{23} \times [0,2]_{q_1}\times [1,2]_{q_2}$. 
\end{image}
\end{figure}

\subsection{Faces determined by $0 \in K \subset [N]$}
Fix a subset $ K \subset N$. 
We let $\ov K$ be the set of all $j \in [N]$ such that $\min(K) \leq j \leq \max(K)$. 

\begin{defn}
We say $K$ is {\em consecutive} if and only if $\ov K = K$.
\end{defn}

\begin{defn}\label{defn.K'}
Given $K \subset [N]$, we let
	\eqn
	K' = \{
		 \text{$j \in [N]$ such that $j \geq \max(K)$.}
	\}	
	\eqnd
\end{defn}

\begin{defn}[Faces of $B(Y)$]
Now let $K \subset [N]$ be a subset.
We let the {\em front $K$th face of $B(Y)$}
	\eqnn
	\del^{\front}_K B(Y)
	\eqnd
to be the portion of $B(Y)$ collared by the face where $q_j = -w$ for all $0 \neq j \not \in K$.

We finally let the {\em back $K$th face of $B(Y)$}
		\eqnn
		\del^{\back}_K B(Y)
		\eqnd
defined to be the portion of $Y$ collared by the face where $q_j = w$ for all $0 \neq j \in K$.
\end{defn}

In terms of the simplicial set $\lag(M)$: Fix an $N$-simplex $Y$. Given a subset $0 \in K \subset [N]$, we let $Y_K$ denote the $N'$-simplex determined by the natural order-preserving embedding
	\eqnn
	\rho: [N'] \to [N]
	\eqnd
with image $K$. Here, $N' = \#K - 1$. 

\begin{lemma}\label{lemma.B-faces}
For all $0 \in K \subset [N]$, we have that
	\eqn\label{eqn.BY_K}
	\del^{\front}_K B(Y) = B(Y_K).
	\eqnd
	
For $|K| < N$, we have that
	\eqn\label{eqn.BY_K'}
		\del^{\back}_K B(Y)
		=
		c^{\ov K \setminus K} \times B(Y_{K'})
	\eqnd
where $K'$ is as in Definition~\ref{defn.K'}.
\end{lemma}

We relegate the proof to the appendix, as it is purely combinatorial.
% Activate the following line by filling in the right side. If for example the name of the root file is Main.tex, write
% "...root = Main.tex" if the chapter file is in the same directory, and "...root = ../Main.tex" if the chapter is in a subdirectory.
 
%!TEX root = _pairing.tex

\section{The functor $\Xi$ on $\lag^{\dd 0}$ and stabilization}
Whenever $Y_i$ is an object of $\lag^{\dd 0}$, we let $\Xi(Y_i)$ denote the module $WF^*(-,Y_i)$. So the structure maps for the module $\mu^d_{\Xi(Y_i)}$ are precisely the maps $\mu^d_M$ used in the wrapped Fukaya category of $M$.

\subsection{On edges}
Let $Y = Y_{01}$ be an edge of $\lag^{\dd 0}$. We now define an element 
	\eqnn
	\Xi_Y \in \hom_{\fuk\Mod}(\Xi(Y_0),\Xi(Y_1)).
	\eqnd

Let $\gamma_0,\ldots,\gamma_{d-1}$ be branes forming a staircase, wider and deeper than $Y$. Let $X_0,\ldots, X_{d-1}$ a collection of objets in $\wrap^\compact$. Setting $X^\dd_i = X_i \times \gamma_i$, one can define operations $\mu^d$
	\eqnn
	CF^*(X^{\dd}_{d-1}, Y) \tensor CF^*(X^{\dd}_{d-2},X^{\dd}_{d-1})
					\tensor \ldots \tensor CF^*(X^{\dd}_0, X^{\dd}_1)
					\to CF^*(X_0^{\dd}, Y)
	\eqnd
of degree $2-d$ 
by counting (with signs) points in the zero-dimensional moduli space of $u:S \to M \times T^*F$ satisfying (\ref{eqn.floer}). By the compactness and regularity results of Section~\ref{section.MTR}, these $\mu^d$ satisfy the $A_\infty$ relations as usual. 

\begin{notation}
Here, $CF^*$ denote the Floer complex determined by the Floer and perturbation data we specified in Section~\ref{section.geometry}. As we mentioned in Section~\ref{section.wrapped}, we leave it to the reader to apply the telescoping definition, or quadratic-near-infinity definition, of the Hamiltonians we use in the $M$ component.
\end{notation}

We set some notation. Note that as graded vector spaces, we have isomorphisms
	\eqnn
		\iota_{[1]}: WF^*(X_{i},X_j) \cong CF^*(X_{i}^{\dd}, X_{j}^{\dd}) : \pi_{[1]}	\eqnd
(the subscript $[1]$ will conform to later notation) and
	\eqn\label{eqn.edge-decomposition}
		CF^*(X_i^{\dd}, B(Y)) \cong CF^*(X_i, Y_0) \oplus CF^*(X_i, Y_1)[-1].
	\eqnd
(The grading shift is computed by keeping track of the slope of the cone tail $c$.) 

\begin{remark}
When $\fuk = \wrap^\compact$, (\ref{eqn.edge-decomposition}) uses the fact that $e$ was chosen to be the same for both $\lag$ and for $\wrap^\compact$---this ensures that $X_i \times \gamma_i$ has no intersection with $B(Y)$ other than where $B(Y)$ is collared by the cone tail.\footnote{The potential for further intersection is where $\gamma_i$ passes below $Y$, but both the depth of the $\gamma_i$ and the $e$ are chosen so that there is no intersection in $(-w_Y,w_Y)$.} When $\fuk = \wrap$, (\ref{eqn.edge-decomposition}) holds because our cobordisms $Y$ have no intersection with $F^\vee_{\leq -D}$ for some $D$, while the $\gamma_i$ are to have depth greater than $D$. 
\end{remark}

We let $\iota_0$ denote the inclusion
	\eqnn
	CF^*(X_i, Y_0) \into CF^*(X_i^{\dd}, Y)
	\eqnd
into the first summand, and let $\iota_{[1]}$ (by abuse of notation) denote the inclusion of the second summand:
	\eqnn
	\iota_{[1]} :
	CF^*(X_i, Y_1) \into CF^*(X_i^{\dd}, Y).
	\eqnd
This $\iota_{[1]}$ is a degree -1 map.
Likewise, we let $\pi_0$ and $\pi_{[1]}$ denote the projections to these summands; $\pi_{[1]}$ is a degree 1 map.

Note that we have
	\eqn\label{eqn.id-edge}
	\id_{CF^*(X_i^{\dd}, Y)} = \iota_0 \pi_0 + \iota_{[1]}\pi_{[1]}
	\qquad
	\text{and}
	\qquad
	\id_{CF^*(X_i^{\dd}, X_j^{\dd})} = \iota_{[1]} \pi_{[1]}.
	\eqnd

\begin{defn}
For all $d \geq 1$, we let
	\eqnn
	\Xi^{d}_{Y} := \pi_{[1]} \circ \mu^{d} \circ ( \iota_0 \tensor \iota_{[1]}^{\tensor d-1} ).
	\eqnd
It is a map
	\eqnn
	\Xi(Y_0)(X_{d}) \tensor CF^*(X_{d-1}, X_{d}) \tensor \ldots \tensor CF^*(X_0,X_1) \to \Xi(Y_1)(X_0)
	\eqnd
of degree $-d+1$. 
Geometrically, $\Xi^d_Y$ counts holomorphic disks in $M \times T^*F$ whose boundary points $x_0,\ldots,x_{d-1}$ have $q$ coordinate $\geq w$, while $x_d$ has $q=-w$. 
\end{defn}

\begin{theorem}\label{theorem.edges}
Fix an edge $Y$ of $\lag^{\dd 0}$. The maps $(\Xi^d_Y)_{d \geq 1}$ define a closed element of $\hom_{\fuk\Mod}(\Xi(Y_0), \Xi(Y_1))$. In other words, using the notation of Definition~\ref{defn.dg-nerve}, the collection
	\eqnn
	\Xi(Y) := ( (\Xi(Y_i)), (\Xi_{Y})).
	\eqnd
is an edge in $N(\wrap^\compact\Mod)$.
\end{theorem}

Let us first deal with a special case:

\begin{proposition}\label{prop.id}
Suppose $Y = \id_{Y_0}$ is the identity cobordism. Then
	\eqnn
	\Xi^1_Y = \id,
	\qquad
	\Xi^{\geq 2}_Y = 0.
	\eqnd
\end{proposition}

\begin{proof}[Proof of Proposition~\ref{prop.id}.]
In this case, we are computing solutions $\tilde u: S \to M \times T^*F$ solving the perturbed Floer equation~(\ref{eqn.floer}) with boundary on $Y_0 \times c$ and $X_i \times \gamma_i$. Since $Y_0 \times F$ has no depth, we can choose our almost complex structure on $F_q \times [-\un{D}, T]$ to split as a direct sum $J_M \oplus J_{T^*F}$ and regularity holds by considering~(\ref{eqn.regular-SES}) again. 

Then, the boundary-stripping argument shows any $\tilde u$ must project to the region $R''$ (see Definition~\ref{defn.R} and Figure~\ref{figure.staircase-disks}, 3.). 

When $d=1$, the result is standard---it is a Kunneth theorem for $X_0 \times \gamma_0$ and $Y_0 \times c$. Since any non-trivial strip comes in a 1-dimensional moduli, the only strips we count in $\Xi^1$ must be trivial in one component, and non-trivial in the other. In particular, $\Xi^1$ counts disks that are trivial in the $M$ direction and non-constant in the $T^*F$ component---there is one such disk for each Hamiltonian chord from $X_0$ to $Y_0$, hence we recover the identity map from $CF^*(X_0,Y_0)$ to itself.

For $d \geq 2$: Let $u = \pi_{T^*F} \circ \tilde u$. Here, 
	\eqnn
		\pi_{T^*F} : M \times T^*F \to T^*F
	\eqnd
is the projection. By Lemma~\ref{lemma.fibers}, any such $u$ comes in a one-dimensional family even after fixing the conformal structure $S$ of the domain. Since we have a direct sum perturbation datum, we know that
	\eqnn
	\tilde u = (u_M, u)
	\eqnd
has components $u_M = \pi_M \circ \tilde u$ and $u$ which are individually solutions to the perturbed Floer equation~(\ref{eqn.floer}). (Of course, $u$ is an honest holomorphic disk.) This shows that $\tilde u$ cannot come in a 0-dimensional family, as any choice has 1-parameter freedom in the $u$ factor. This shows $\Xi^d = 0$ for $d \geq 2$.
\end{proof}

The proof of Theorem~\ref{theorem.edges} relies on the following lemma, which says that disks with certain boundary conditions on $M \times T^*F$ recover disks in $M$:

\begin{lemma}[Disks reduce]\label{lemma.collapse}
Let $\mu^m_M$ denote the usual $A_\infty$ operations of $\wrap(M)$, while $\mu^m$ comes from counting disks in $M \times T^*F$. 
We have that
	\eqn\label{eqn.edge-recover}
	\pi_{[1]} \circ \mu^b \circ (\iota_{[1]})^{\tensor b} = \mu^b_M
	\eqnd
and
	\eqn\label{eqn.edge-recover-plus}
	\pi_{0} \circ \mu^{b} \circ (\iota_{0} \tensor \iota_{[1]}^{\tensor b-1}) = \mu^{b}_M.
	\eqnd
\end{lemma}

\begin{proof}[Proof of Theorem~\ref{theorem.edges} assuming Lemma~\ref{lemma.collapse}.]
We know the operations $\mu^m$ satisfy the $A_\infty$ relations. This means that for some coherent system of signs $\clubsuit$, we have
	\eqn\label{eqn.Aoo-functions}
		\sum_{a + b + c = d} (-1)^{\clubsuit} \mu^{a+1+c}(\id^{\tensor a} \tensor \mu^b \tensor \id^{\tensor c}) = 0.
	\eqnd
In particular, we can pre-compose the above with $\iota_0 \tensor (\iota_{[1]})^{\tensor d-1}$ and post-compose with $\pi_{[1]}$ to still equal zero. Let us analyze the summation one summand at a time. When $a>0$, we have that
	\begin{align}
	&\pi_{[1]}\circ \mu^{a+1+c}(\id^{\tensor a} \tensor \mu^b \tensor \id^{\tensor c}) \circ \iota_0 \tensor \iota_{[1]}^{\tensor d-1} \nonumber \\
		&= \pi_{[1]}\circ \mu^{a+1+c}(\iota_0 \tensor \iota_{[1]}^{\tensor a-1} \tensor (\iota_{[1]} \circ \mu^b_M) \tensor \iota_{[1]}^{\tensor c}) \nonumber \\
		&= \Xi_Y^{a+1+c}(\id^{\tensor a} \tensor\mu^b_M \tensor \id^{\tensor c}) \label{eqn.edge-2}.
	\end{align}
We used (\ref{eqn.id-edge}) and (\ref{eqn.edge-recover}), and then the definition of $\Xi_Y$.

While when $a=0$, we have:
	\begin{align}
	&\pi_{[1]}\circ \mu^{1+c}(\mu^b \tensor \id^{\tensor c}) \circ (\iota_0 \tensor \iota_{[1]}^{\tensor b+c-1}) \nonumber \\
		&= \pi_{[1]}\circ \mu^{1+c}( (\iota_{[1]}\pi_{[1]} +  \iota_0\pi_0) \mu^b \tensor \id^{\tensor c}) \circ (\iota_0 \tensor \iota_{[1]}^{\tensor b+c-1}) \nonumber \\ %1
		&= \pi_{[1]}\circ \mu^{1+c}((\iota_{[1]}\circ \Xi^b_Y) \tensor \iota_{[1]}^{\tensor c})
				+\pi_{[1]}\circ \mu^{1+c}((\iota_0\pi_0) \mu^b \tensor \id^{\tensor c}) \circ (\iota_0 \tensor \iota_{[1]}^{\tensor b+c-1})\nonumber \\ %2
		&= \mu^{1+c}_M(\Xi^b_Y \tensor \id^{\tensor c})
				+\pi_{[1]}\circ \mu^{1+c}((\iota_0\pi_0) \mu^b \tensor \id^{\tensor c}) \circ (\iota_0 \tensor \iota_{[1]}^{\tensor b+c-1})\nonumber \\ %3
		&= \mu^{1+c}_M(\Xi^b_Y \tensor \id^{\tensor c})
				+\pi_{[1]}\circ \mu^{1+c}((\iota_0 \circ \mu^b_M) \tensor \iota_{[1]}^{\tensor c})\nonumber  \\ %4
		&= \mu^{1+c}_M(\Xi^b_Y \tensor \id^{\tensor c})
				+\Xi_Y^{1+c}(\mu^b_M \tensor \id^{\tensor c})\label{eqn.edge-1} 
	\end{align}
Here we used, in order,  
	(\ref{eqn.id-edge}),  %1
	the definition of $\Xi_Y$, %2
	(\ref{eqn.edge-recover}),  %4
	(\ref{eqn.edge-recover-plus}),  %3
	and then the definition of $\Xi_Y$.

Plugging (\ref{eqn.edge-1}) and (\ref{eqn.edge-2}) into (\ref{eqn.Aoo-functions}), we obtain
	\begin{align}
			&\sum_{b+c=d}
				(-1)^\clubsuit \mu_{\Xi(Y_1)}^{1+c}(\Xi_Y^{b}\tensor\id^{\tensor c}) \nonumber \\
			&+ \sum_{b+c=d} (-1)^\clubsuit \Xi_Y^{1+c}(\mu^{b}_{\Xi(Y_0)} \tensor \id^{\tensor c})  \nonumber \\
			&+ \sum_{a>0, a+b+c = d} (-1)^\clubsuit \Xi_Y^{a+1+c} (\id^{\tensor a} \tensor \mu^b_\cA\tensor \id^{\tensor c}) \nonumber \\
			&= 0. \nonumber
	\end{align}
This is precisely the equation from  Lemma~\ref{lemma.goal} with $N=1$. Note that we have used that $\mu^m_{\Xi(Y_i)} = \mu^m_M$ when the output point of $\mu^m_M$ is in $CF^*(-,Y_i)$. 
\end{proof}

\begin{proof}[Proof of Lemma~\ref{lemma.collapse}.]
The first assertion of the Lemma counts disks $\tilde u: S \to M \times T^*F$ with boundary punctures all sent to points $x_i \in M \times T^*F$ with $q \geq w$. The second assertion counts disks with boundary punctures sent to $x_i$ such that $x_0$ and $x_b$ have $q=-w$, with all other $x_i$ having $q\geq w$. As before, we let $u = \pi_{T^*F} \circ \tilde u$. We have to show that there is a bijection between disks $\tilde u$, and disks $u_M: S \to M$ with boundary conditions given by $\pi_M(x_i)$. 

We first claim that any $u$ must have image contained in the strip $F_q \times [-\un{D}, -\ov{D}] \subset T^*F$. One does this by boundary-stripping, first showing that $u$ has image contained entirely outside the region $[-w, w] \times [-D_Y,\infty)$, hence contained entirely inside the regions depicted in 1. and 2. of Figure~\ref{figure.staircase-disks}.  Thus, while a priori the Floer equation and boundary conditions allowed for more arbitrary disks, all the $\tilde u$ are actually in bijection with disks with boundary conditions where $Y$ is replaced by $Y_0 \times \{q=w\}$ (for (\ref{eqn.edge-recover}) of the Lemma) and by $Y_1 \times \{q=-w\}$ (for (\ref{eqn.edge-recover-plus}).

Then we are in the cases \ref{item.1} and \ref{item.2} of Lemma~\ref{lemma.fibers}. The same argument as in Seidel~\cite{seidel-lefschetz-i}, Lemma~7.3, shows the bijection of moduli spaces. Here is another proof: 

Let $\cM_M$ denote the moduli space of $u_M: S \to M$ with the relevant boundary conditions $X_i$ (and the chords between them),  $\cM_{T^*F}$ the moduli space of $u: S \to T^*F$ with the given boundary conditions $(\gamma_i, \pi_{T^*F}(x_i))$, and let $\cM$ denote the moduli space of $\tilde u$. Then one can write
	\eqnn
	\cM \cong \cM_M \times_{\cR} \cM_{T^*F}
	\eqnd
where the fiber product is over the forgetful map $(u,S) \mapsto S \in \cR$. (Recall we are using the direct sum almost-complex structure---hence a holomorphic map to $\cM$ is a pair of holomorphic maps to $M$, and to $T^*F$.) In general, this fiber product need not be smooth, but by Lemma~\ref{lemma.fibers}, the map $\cM_{T^*F} \to \cR$ is a diffeomorphism, hence the fiber product is a transverse fiber product. This completes the proof.
\end{proof}

\subsection{For higher simplices}
Let $Y \subset M \times T^*F^N$ be an $N$-simplex of $\lag^{\dd 0}$. As usual we center $Y$ so we can speak of its depth $D$ and its width $w$. We fix a staircase configuration $\gamma_0,\ldots,\gamma_{d-1}$ compatible with this $w$ and $D$. Given $X_i \in \ob \fuk$, let $X_i^{\dd} = X_i \times (\gamma_i)^N \subset M \times T^*F^N$. 

Recall from \ref{section.B} that one can think of $B(Y)$ as living over a cube. The vertices of the cube are in bijection with subsets $0 \in K \subset [N]$. Hence the intersection
	\eqnn
	B(Y) \cap X^{\dd}_i
	\eqnd
can be indexed to give an isomorphism of graded abelian groups (but not of chain complexes):
	\eqnn
	CF^*(B(Y), X_i^{\dd})
	\cong
	\bigoplus_{0 \in K \subset [N]}
		CF^*(X_i, Y_{ K_{\max}})[1 - \# K].
	\eqnd
\begin{remark}
Again, the $\Lambda$-non-characteristic condition for $Y$ ensures that there are no further intersections away from the cone tails of $B(Y)$. Also, a simple slope computation shows the grading shifts above.
\end{remark}
We define inclusion maps $\iota_K$ and projection maps $\pi_K$ of degrees $1-\#K$ and $\# K-1$, respectively. Likewise, since all the intersections of $X^\dd_i$ and $X^{\dd}_j$ occur where the cubical region is indexed by $K=[N]$, we let
	\eqnn
		\iota_{[N]} : CF^*(X_i,X_j) \cong CF^*(X_i^\dd, X_j^\dd) : \pi_{[N[}
	\eqnd
denote the isomorphisms.
Obviously,
	\eqn\label{eqn.id-cube}
	\id_{CF^*(B(Y),X_i^{\dd})} = \sum_{0 \in K \subset [N]} \iota_K \circ \pi_K.
	\eqnd

As in the previous section, we have $A_\infty$ maps $\mu^d$ 
	\eqnn
	CF^*(X_{d-1}^{\dd}, B(Y)) \tensor CF^*(X_{d-2}^{\dd}, X_{d-1}^{\dd}) \tensor \ldots \tensor CF^*(X_{0}^{\dd}, X_1^{\dd})
	\to
	CF^*(X_0^{\dd}, B(Y))
	\eqnd
of degree 2-d, 
satisfying the $A_\infty$ relations. 

\begin{defn}
Let $Y$ be an $N$-simplex of $\lag^{\dd 0}$. Then we define $\Xi^d_Y$ to be the linear map
	\eqnn
	\Xi^d_Y := \pi_{[N]} \circ \mu^d \circ (\iota_0 \tensor \iota_{[N]}^{\tensor d-1}).
	\eqnd
\end{defn}

The goal of the remainder of this section is to prove:

\begin{theorem}\label{theorem.cubes}
Let $Y$ be an $N$-simplex of $\lag^{\dd 0}$. Then in the notation of Definition~\ref{defn.dg-nerve}, consider the collection
	\eqnn
	\Xi(Y) := ( (\Xi(Y_i))_{i \in [N]}, (\Xi_{B(Y_K)})_{K \subset [N]}).
	\eqnd
This is an $N$-simplex of $N(\fuk\Mod)$.
\end{theorem}

\begin{remark}
As in the previous section, note the difference in subscript and input. $\Xi_Y$ denotes a linear map, while $\Xi(Y)$ denotes the simplex that $\Xi$ assigns.
\end{remark}

\begin{remark}
As usual, given $0 \in K \subset [N]$, we have let $Y_K$ denote the morphism collaring the face of $Y$ spanned by the subsets of $K$. $Y_i$ denotes the objects by which $Y$ is collared.
\end{remark}

To prove the theorem, we will show that the $A_\infty$ relation (\ref{eqn.Aoo-functions}) recovers the requisite relation (\ref{eqn.goal}) found in Lemma~\ref{lemma.goal}. So we state a few lemmas, and then prove the theorem through straight algebra. We then present the proofs of the lemmas, where the geometry is used.

Recall that $K \subset [N]$ is called {\em consecutive} if it does not skip any elements---that if, for all $i, j \in K$ with $i<j$, we must have that $i < k < j \implies k \in K$.

We begin with the $a>0$ term of the $A_\infty$ relations.

\begin{lemma}\label{lemma.a>0}
If $a>0$, we have
	\eqnn
		\pi_{[N]} \circ \mu^{a+1+c} \circ (\iota_{[0]} \tensor \iota_{[N]}^{\tensor a-1} \tensor \mu^b \tensor \iota_{[N]}^{\tensor c})
		=
		\Xi^{a+1+c}_Y (\id^{\tensor a} \tensor \mu^b_M \tensor \iota_{[N]}^{\tensor c}).
	\eqnd
\end{lemma}

\begin{lemma}\label{lemma.Xi_K}
For all $0 \in K \subset [N]$, we have
	\eqn
		\pi_K \mu^b \circ (\iota_0 \tensor \iota_{[N]}^{\tensor b-1}) = \Xi^b_{Y_{K}}.
	\eqnd
\end{lemma}

\begin{lemma}\label{lemma.Xi^1+c}
Fix $0 \in K \subset [N]$. Then
	\eqnn
		\pi_{[N]} \circ \mu^{1+c} (\iota_K \tensor \iota_{[N]}^{\tensor c})
		=
		\begin{cases}
			\mu^{1+c}_M & K = [N] \\
			\Xi^{1+c}_{K'} & \text{$K$ consecutive}, |K| \leq N \\
			\id					& \text{$K$ not consecutive}, |K| = N, c=0\\
			0 					& \text{$K$ not consecutive}, |K| = N, c > 0 \\
			0						& \text{$K$ not consecutive}, |K| < N.
		\end{cases}
	\eqnd
Here, when $K$ is consecutive, $K' \subset [N]$ is the set of all $k'$ such that $k' \geq \max(K)$. Note also that the first line ($K= [N]$) is actually a special case of the second line, though we separate the two cases for clarity in the proof.
\end{lemma}

\begin{corollary}\label{cor.1}
By combining Lemmas~\ref{lemma.Xi_K} and~\ref{lemma.Xi^1+c}, we have:
	\begin{align}
	\sum_{0 \in K \subset [N]} (-1)^{\clubsuit} \pi_{[N]} \mu^1 \circ \iota_K \circ (\pi_K \circ \mu^d \circ (\iota_0 \tensor \iota_{[N]}^{\tensor d-1}))
	=& (-1)^\clubsuit \mu^1_{\Xi(Y_N)} \Xi^d_Y\nonumber\\
	&+\sum_{0<j<N} (-1)^\clubsuit \Xi^d_{Y_{[N] \setminus \{j\}}} \nonumber\\
	&+ \sum_{\text{$K$ consecutive}} (-1)^\clubsuit \Xi^1_{K'}\circ \Xi^d_K\nonumber
	\end{align}
where as before, $K' \subset [N]$ is the set consisting of all $k'$ such that $k' \geq \max(K)$. Note we have used the fact that $\mu^1_M = \mu^1_{\Xi(Y_N)}$ in the first term.
\end{corollary}

\begin{corollary}\label{cor.2}
By Lemmas~\ref{lemma.Xi_K} and~\ref{lemma.Xi^1+c}, we have
	\begin{align}
		&\sum_{
				\substack{
					a=0 \\
					c > 0 \\
					b+c = d\\
					\text{$K$ consecutive}
          }
    }
		(-1)^{\clubsuit} \pi_{[N]} \mu^{1+c}  \circ (\iota_K \tensor \iota_{[N]}^{\tensor c}) \circ ( (\pi_K \circ  \mu^b \circ (\iota_{0} \tensor \iota_{[N]}^{\tensor b-1})) \tensor \id^{\tensor c}) \nonumber \\
		&= (-1)^{\clubsuit} \mu_{\Xi(Y_N)}^{1+c}(\Xi^b_Y \tensor \id^{\tensor c})
		+ \Xi^{1+c}_K (\mu^b_M \tensor \id^{\tensor c})
		+ \sum_{\substack{K \neq \{0\} \\ \text{$K$ consecutive}}}
			\Xi^{1+c}_{K'} (\Xi^b_K \tensor \id^{\tensor c})\nonumber.
	\end{align}
\end{corollary}

\begin{corollary}\label{cor.3}
Lemma~\ref{lemma.Xi^1+c} tells us
	\eqnn
	0 = \sum_{\substack{a=0 \\ c>0 \\ \text{$K$ not consecutive}}}
		(-1)^{\clubsuit} \pi_{[N]} \mu^{1+c}  \circ (\iota_K \tensor \iota_{[N]}^{\tensor c}) \circ \left( (\pi_K \circ  \mu^b \circ (\iota_0 \tensor \iota_{[N]}^{\tensor b-1})) \tensor \id^{\tensor c}\right).
	\eqnd
\end{corollary}

\begin{corollary}\label{cor.4}
By Lemma~\ref{lemma.a>0} and the definition of $\Xi_Y$, 
		\begin{align}
			\sum_{\substack{a > 0 \\ a+b+c=d}} &(-1)^{\clubsuit} \pi_{[N]} \mu^{a+1+c} (\iota_0 \tensor \iota_{[N]}^{\tensor a-1} \tensor (\mu^b \circ \iota_{[N]}^{\tensor b}) \tensor \iota_{[N]}^{\tensor c}) \nonumber\\
			&= \sum_{\substack{a>0, \\ a+b+c = d}} \Xi^{a+1+c}_Y (\id^{\tensor a} \tensor \mu^b_M \tensor \id^{\tensor c}).\nonumber
		\end{align}
\end{corollary}

\begin{proof}[Proof of Theorem~\ref{theorem.cubes} assuming lemmas.]
Let us write out the expression
	\eqnn
	\pi_N \circ \mu^{a+b+c} (\id^{\tensor a} \tensor \mu^b \tensor \id^{\tensor c}) \circ (\iota_0 \tensor \iota_{[N]}^{\tensor d}).
	\eqnd
Using (\ref{eqn.id-cube}) for the $a=0$ case, we can group these terms as follows:
	\begin{align}
	0 = 
		& \sum_{0 \in K \subset [N]} (-1)^{\clubsuit} \pi_{[N]} \mu^1 \circ \iota_K \circ (\pi_K \circ \mu^d \circ (\iota_0 \tensor \iota_{[N]}^{\tensor d-1})) \label{proof.1}\\ %1
		&\sum_{
				\substack{
					a=0 \\
					c > 0 \\
					b+c = d\\
					\text{$K$ consecutive}
          }
    }
		(-1)^{\clubsuit} \pi_{[N]} \mu^{1+c}  \circ (\iota_K \tensor \iota_{[N]}^{\tensor c}) \circ ( (\pi_K \circ  \mu^b \circ (\iota_{0} \tensor \iota_{[N]}^{\tensor b-1})) \tensor \id^{\tensor c})  \label{proof.2}\\ %2
		&+ \sum_{\substack{a=0 \\ c>0 \\ \text{$K$ not consecutive}}}
		(-1)^{\clubsuit} \pi_{[N]} \mu^{1+c}  \circ (\iota_K \tensor \iota_{[N]}^{\tensor c}) \circ \left( (\pi_K \circ  \mu^b \circ (\iota_0 \tensor \iota_{[N]}^{\tensor b-1})) \tensor \id^{\tensor c}\right) \label{proof.3}\\ %3
			&+\sum_{\substack{a > 0 \\ a+b+c=d}} (-1)^{\clubsuit} \pi_{[N]} \mu^{a+1+c} (\iota_0 \tensor \iota_{[N]}^{\tensor a-1} \tensor (\mu^b \circ \iota_{[N]}^{\tensor b}) \tensor \iota_{[N]}^{\tensor c})  \label{proof.4}%4
	\end{align}	
The expression in line number (\ref{proof.1}) is simplified by Corollary~\ref{cor.1}, in line (\ref{proof.2}) is simplified by Corollary~\ref{cor.2}, and so forth in order. (We have arranged the order of the corollaries to match with order of the line numbers above.) Thus using all the corollaries, we arrive at
	\begin{align}
			0=&\sum_{b+c=d}
				(-1)^\clubsuit \mu_{\Xi(Y_N)}^{1+c}(\Xi_Y^{b}\tensor\id^{\tensor c}) \nonumber \\
			&+ \sum_{b+c=d} (-1)^\clubsuit \Xi_Y^{1+c}(\mu^{b}_{\Xi(Y_0)} \tensor \id^{\tensor c})  \nonumber \\
			&+ \sum_{a>0, a+b+c = d} (-1)^\clubsuit \Xi_Y^{a+1+c} (\id^{\tensor a} \tensor \mu^b_\cA\tensor \id^{\tensor c}) \nonumber \\
			&+  \sum_{0<j<N} (-1)^\clubsuit \Xi_{Y_{[N] \setminus j}}^d\nonumber \\
				&+\sum_{J = J'' \wedge J'} \sum_{b+c=d} (-1)^\clubsuit \Xi_{Y_{J''}}^{1+c} (\Xi_{Y_{J'}}^{b}\tensor\id^{\tensor c})]
				.\nonumber
	\end{align}
which is precisely Equation~(\ref{eqn.goal}). The theorem is proven.
\end{proof}

\subsection{Proof of Lemmas}
\begin{proof}[Proof of Lemma~\ref{lemma.a>0}.]
Note that for $a>0$, $\mu^b$ counts disks
	\eqnn
	\tilde u: S \to M \times T^*F^N
	\eqnd
such that
	\eqnn
		\tilde u(\del S) \subset \bigcup_{c \leq i \leq b+c} X_i^{\dd}
		=
		\bigcup_{c \leq i \leq b+c} X_i \times \gamma_i^N.
	\eqnd
In particular, all Hamiltonian chords involved have $F^N_q$ coordinate with $q_j \geq w$ for all $1 \leq j \leq N$.

For any $1 \leq j \leq N$, consider the component 
	\eqn\label{eqn.u_j}
		u_j := \pi_{T^*F_j} \circ \tilde u : S \to T^*F_j
	\eqnd
to the $j$th factor. We claim $u_j$ has image where $J$ splits as a direct sum $J_M \oplus (J_{T^*F})^N$ with $J_{T^*F}$ the standard almost-complex structure (\ref{eqn.i}). To see this, one applies boundary-stripping (see Remark~\ref{remark.boundary-stripping}) beginning with
	\eqn\label{eqn.C_0}
		C_0 = F_q \times \left( (-\infty, -\un{D}] \coprod [-\ov{D},\infty) \right) \bigcup \left(\bigcup_{0 \leq i < b} \gamma_i \right) \bigcup c.
	\eqnd
See Figure~\ref{figure.C-a>0}. Importantly, the open regions in the complement of $C_0$ are all regions where $J$ splits as a direct sum $J_{M \times T^*F^{N \setminus \{j\}}} \oplus J_{T^*F_j}$. By stripping away the portion of $c$ with $q_j < w$, we find that $u_j(S)$ is fully contained in $[w,\infty) \times [-\un{D} , -\ov{D}]$, and hence---by boundary-stripping again--- in the region $R$ bounded by the staircase (see Definition~\ref{defn.R}).

\begin{figure}
		\[
			\xy
			\xyimport(8,8)(0,0){\includegraphics[width=4in]{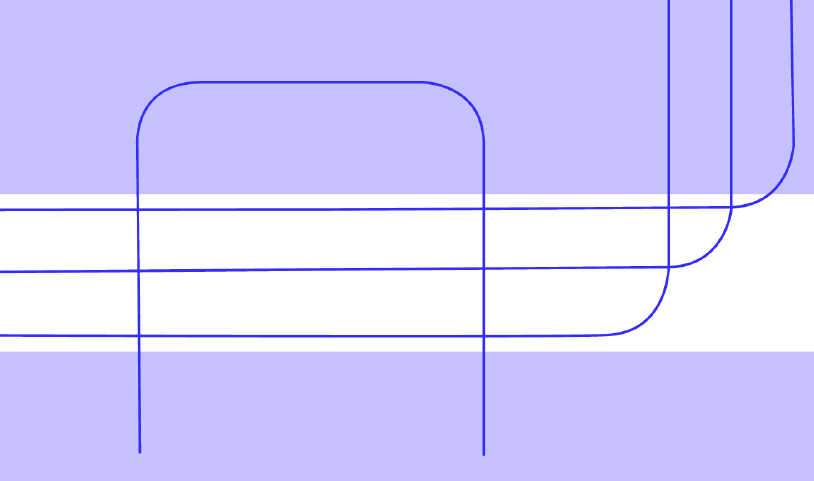}}
			%%Grid, for help.
%		 	,(2,1)*+{\bullet2},(3,1)*+{\bullet3},(4,1)*+{\bullet4},(5,1)*+{\bullet5},(6,1)*+{\bullet6},(7,1)*+{\bullet7},(1,1)*+{\bullet1},(1,2)*+{\bullet2},(1,3)*+{\bullet3},(1,4)*+{\bullet4},(1,5)*+{\bullet5},(1,6)*+{\bullet6},(1,7)*+{\bullet7}
				%%End Helping grid
			\endxy
			\]
\begin{image}\label{figure.C-a>0}
The $C_0$ used in the proof of Lemma~\ref{lemma.a>0}, indicated in blue. (The curves themselves are also blue---they are part of $C_0$.)
\end{image}
\end{figure}

Since the choice of $j$ was arbitrary, it follows that $\tilde u$ has image in $M \times \left( [w,\infty) \times [-\un{D},-\ov{D}]\right)^N$. Here, the almost complex structure is indeed required to split as $J_M \oplus (J_{T^*F})^N$, so each of the components of
	\eqnn
	\tilde u = (u_M, u_1,\ldots,u_N)
	\eqnd
satisfies the perturbed Floer equation (\ref{eqn.floer}) (which of course is the honest holomorphic curve equation in $T^*F_j$). Since each $u_j$ is uniquely determined given $S$ (Lemma~\ref{lemma.fibers}), the assignment $\tilde u \mapsto u_M$ defines a bijection
	\eqnn
	\{ \tilde u\} \cong \{u_M\}.
	\eqnd
\end{proof}

\begin{proof}[Proof of Lemma~\ref{lemma.Xi_K}.]
Note that the expression $\pi_K \mu^b \circ (\iota_0 \tensor \iota_{[N]}^{\tensor b-1})$ counts disks whose boundary Hamiltonian chords $x_0,\ldots,x_b$ in $M \times T^*F^N$ project to points in $T^*F^N$ whose $j$th $q$-coordiantes are given by
	\begin{align}
		q_j(x_i) & \geq w \qquad ( i \in 1, \ldots, b-1) \nonumber \\
		q_j(x_b) & \leq w \nonumber \\
		q_j(x_0) & \begin{cases} \geq w & j \in K \\ \leq -w & j \not \in K \end{cases} \nonumber.
	\end{align}
Now fix $j \not \in K$ and consider the component $u_j = \pi_{T^*F_j} \circ \tilde u$. We claim that $u_j$ has image as depicted in \ref{item.2} of Figure~\ref{figure.staircase-disks}---in other words, it has image contained in the region $R'$ bounded by the staircase. (See Definition~\ref{defn.R}.) To see this, we apply boundary-stripping starting with the same $C_0$ as in (\ref{eqn.C_0}). (Note we are using assumption (\ref{eqn.cube-J-split}) here.) We omit the details. We conclude $u_j$ is then uniquely determined by the conformal structure $S$ by Lemma~\ref{lemma.fibers}. This means the map
	\eqn\label{eqn.u-components}
		\tilde u = (u_M, u_1,\ldots, u_N)
	\eqnd
is completely determined by its ``non-$K$'' components. Moreover, since each $u_j$ for $j \not \in K$ is contained in $R''$, its boundary along $Y$ actually reduces to the portion of $Y$ collared along $q_j = - w$. So the ``non-$K$'' component has boundary on $\del^{\front}_K B(Y)$.  Now we are in the home stretch.

Let $N'= \# K - 1$ and consider the natural order-preserving embedding
	\eqnn
	\rho: [N'] \to [N]
	\eqnd
with image $K$. Define a map
	\eqnn
		\tilde u_K : S \to M \times T^*F^{N'}
	\eqnd
by setting the components of $\tilde u_K$ to be
	\eqnn
		\tilde u_K :=		(u_M, u_{\rho(1)}, \ldots, u_{\rho(N')}).
	\eqnd
(We are using the notation from (\ref{eqn.u-components}). Then $\tilde u_K$ is a disk satisfying the perturbed Floer equation (\ref{eqn.floer}) with boundary conditions as follows:
	\begin{itemize}
		\item
			The $i$th brane is $X_i \times \gamma_i^{N'}$ for $0 \leq i < b$.
		\item
			The $b$th brane is 
				\eqnn
				\del^{\front}_K B(Y)=B(Y_K)
				\eqnd
			by~(\ref{eqn.BY_K}).
		\item
			The $i$th marked point is sent to a Hamiltonian chord $y_i $ in $M \times T^*F^{N'}$ such that for all $j \in 1, \ldots, N'$, we have
				\eqnn
					q_j(y_0),\ldots,q_j(y_{N'}) \geq w,
					\qquad
					q_j(y_b) \leq -w.
				\eqnd
	\end{itemize}
This means $\tilde u_K$ is precisely a disk appearing in the operation $\Xi^b_K$. Since the map $\tilde u \mapsto \tilde u_K$ establishes a bijection of moduli spaces (given a $\tilde u_K$, one has a unique lift to $\tilde u$ by the given $S$), we are finished. 
\end{proof}

\begin{proof}[Proof of Lemma~\ref{lemma.Xi^1+c}]
We proceed case by case:

{\bf The case $K = [N]$.} This case counts disks $\tilde u: S \to M \times T^*F^N$ for which all boundary Hamiltonian chords lie above points in $T^*F^N$ with $q_j\geq w$ for all $1 \leq j \leq N$. By considering $u_j = \pi_{T^*F_j} \circ \tilde u$, one can apply the boundary-stripping trick with the same $C$ as in Lemma~\ref{lemma.q>w} to show that each $u_j$ has image contained in the region $R$ bounded by the staircase $(\gamma_i)_{0\leq i \leq c}$. This proves $\tilde u$ has image in the region $M \times \left( [w,\infty) \times [-\un{D},-\ov{D}]\right)^N$ where $J = J_M \oplus (J_{T^*F})^N$, and by Lemma~\ref{lemma.fibers}, the map
	\eqn
	\tilde u = (u_M, u_1,\ldots,u_N),
	\eqnd
is determined completely by $u_M= \pi_M \circ \tilde u: S \to M$ once one fixes (the conformal structure of) $S$. Hence the assignment $\tilde u \mapsto u_M$ defines a bijection of moduli spaces proving the claim.

{\bf The case $K$ consecutive, $|K| \leq N$.} For $k \in K$, the same argument as the case $K = [N]$ shows that the component $u_{k}:S \to T^*F_k$ is uniquely determined given a conformal structure $S$. Since the data of $(u_k)_{k \in K}$ is redundant, we only consider the ``non-$K$'' component
	\eqnn
	\tilde u_{K^C} = (u_M, u_{\max(K)+1}, \ldots, u_N): S \to M \times T^*F^{[N] \setminus K}.
	\eqnd
Since each $u_k$ for $k \in K$ lands in region $R$, the almost complex structure used near the image of $\tilde u$ splits as a direct sum
	\eqnn
	J = J_{M \times T^*F^{[N] \setminus K}} \oplus (J_{T^*F})^{K \setminus \{0\}}.
	\eqnd
Hence $\tilde u_{K^C}$ is a solution to the perturbed Floer equation~(\ref{eqn.floer}). Moreover, we can determine its boundary conditions: 
Note that the original $\tilde u$ had boundary on $Y$ where $q_{k} \geq w$ for $k \in K$. This means the boundary of $\tilde u$ along $Y$ reduces to the back $K$th face. By~(\ref{eqn.BY_K'}), the back $K$th face is
	\eqnn
	\del_K^{\back} Y = B(Y_{K'})
	\eqnd
where $K' \subset [N]$ is the set of all $k'$ such that $k' \geq \max(K)$. (By the assumption that $K$ is consecutive, $\ov K \setminus K$ is empty.) In summary, the boundary conditions are:
	\begin{itemize}
		\item
			Branes $X_i \times \gamma_i^{[N] \setminus K}$ for $0\leq i \leq c$, and the $(c+1)$st boundary brane is given by
				\eqnn
				\del^{\back}_K Y = B(Y_{K'}).
				\eqnd
		\item
			The $0$th marked point is sent to a Hamiltonian chord $y_0$ in $M \times T^*F^{[N] \setminus K}$ lying above a point in $T^*F^{[N] \setminus K}$ with $q_j \leq - w$ for all $j$. The 1st marked point is sent to a Hamiltonian chord $y_1$ lying above a point in $T^*F_j$ with $q_j \geq w$. 
	\end{itemize}
In other words, $\tilde u_{K^C}$ is exactly a point in the moduli space of disks determining the operation $\Xi^{1+c}_{Y_{K'}}$. Further, any $\tilde u_{K^C}$ lifts to a $\tilde u: S \to M \times T^*F^N$ by declaring the $K$ components to be given by the unique $(u_k)_{k \in K}$ given by Lemma~\ref{lemma.fibers}. Thus the assignment $\tilde u \mapsto \tilde u_{K^C}$ is a bijection on the zero-dimensional components of moduli spaces, and the proof is complete.

{\bf The case $K$ not consecutive, $|K| = N, c = 0$.} Proceeding as in the previous two cases, we again consider the ``non-$K$'' component
	\eqnn
	(u_M, u_{j}) =: \tilde u_j.
	\eqnd
Here $j$ is the unique element of $[N] \setminus K$.\footnote{Unlike the previous case, the components $u_M$ and $u_j$ need not be solutions to the perturbed Floer's equation~(\ref{eqn.floer}) in each component since $J_{M \times T^*F_j}$ need not split---we have no control on the image of $u_M$ or $u_{j}$.} Since $\tilde u$ has image where we have a splitting
	\eqnn
	J = (J_{M \times T^*F_j}) \oplus (J_{T^*F})^{N \setminus \{j\}}
	\eqnd
we see that $\tilde u_j$ does describe a map $S \to M \times T^*F_j$ satisfying the perturbed Floer's equation~(\ref{eqn.floer}).

Performing the same reasoning as in the previous case about where $\tilde u$ has boundary, we deduce that $\tilde u_j$ satisfies boundary conditions given by:
	\begin{itemize}
		\item
			Two branes, the 0th being $X_0 \times \gamma_0$, and the 1st being
				\eqnn
				\del^{\back}_K Y = Y_N \times c \subset M \times T^*F_j
				\eqnd
			by~(\ref{eqn.BY_K'}), because $K' = \{N\}$.
		\item
			The $0$th marked point is sent to a Hamiltonian chord $y_0$ in $M \times T^*F_j$ lying above a point in $T^*F_j$ with $q_j \leq - w$. The 1st marked point is sent to a Hamiltonian chord $y_1$ lying above a point in $T^*F_j$ with $q_j \geq w$. 
	\end{itemize}
In other words, we have reduced ourselves to the case $\Xi^1_{\id_{Y_N}}$, which we covered in Proposition~\ref{prop.id}: Counting $\tilde u_j$ is precisely computing the identity map from $CF^*(X_0, Y_N)$ to itself. We are finished because the assignment $\tilde u \mapsto \tilde u_j$ is a bijection of moduli spaces. 

{\bf The case $K$ not consecutive, $|K| = N, c > 0$.} We repeat the exact same argument as the $c=0$ case and reduce to computing $\Xi^{1+c}_{\id_{Y_N}}$. Proposition~\ref{prop.id} shows that this vanishes.

{\bf The case $K$ not consecutive, $|K| < N$.} 
For any $0 \neq k \in K$, the boundary conditions for $\tilde u$ and the boundary-stripping trick show that $u_k = \pi_{T^*F_k} \circ \tilde u$ is contained in the region $R$ bounded by the staircase. By Lemma~\ref{lemma.fibers} case~\ref{item.1}., $u_k$ for any $k \in K$ is completely determined by choice of conformal structure $S$ on the domain. So the meat is in understanding the map
	\eqnn
		\tilde u_{K^C} := (u_M, (u_j)_{j \not \in K}) : S \to M \times T^*F^{[N] \setminus K}.
	\eqnd
What are the boundary conditions for $\tilde u_{K^C}$? All the curves $u_k$ with $0 \neq k \in K$ have image contained in the region $R$. This means their $(c+1)$st boundary brane can be reduced to a copy of $(l_w)_k = \{q_k = w\} \subset T^*F_k$ times the face of $Y$ collared at $q_k = w$. By definition of $\del^{\back}_K Y$, this means $\tilde u$ itself has boundary on
	\eqnn
			\del^{\back}_K Y \times	\prod_{0 \neq k \in K} (l_w)_k \subset (M \times T^*F^{K \setminus \{0\}}) \times T^*F^{[N] \setminus K} .
	\eqnd
Taking out the $[N]\setminus K$ factor, we see that $\tilde u_K$ has boundary on
	\eqn\label{eqn.u_K-boundary}
	\del^{\back}_{K} Y = c^{\ov K \setminus K} \times B(Y_{K'}).
	\eqnd
Here, $K'$ is given in Definition~\ref{defn.K'}, and the equality follows from~(\ref{eqn.BY_K'}).

So now let us fix $j \in \ov K \setminus K$ (which is guaranteed to exist because $K$ is not consecutive). Then we are considering the $\mu^{1+c}$ operation with boundary branes of the form
	\eqnn
		(X_i \times \gamma_i^{[N] - K - \{j\})}) \times \gamma_i
		\qquad
		\text{and}
		\qquad
		Y' \times c 
		\qquad
		\subset
		(M \times T^*F^{[N] - K - \{j\}}) \times T^*F_j
	\eqnd
Moreover, one can show that $u_j = \pi_{T^*F_j} \circ \tilde u_K$ has image contained in $R''$ by another boundary-stripping argument. Since the boundary conditions are a direct product of branes, and one of the factors is a brane in $T^*F$, we can use the short exact sequence~(\ref{eqn.regular-SES}) to conclude that the direct sum almost complex structure 
	\eqnn
		J = J_{M \times T^*F^{[N] - K - \{j\}}} \oplus J_{T^*F_j}
	\eqnd
is regular. So we specify this almost-complex structure. In other words, $u_j$ always comes in one-dimensional families for any fixed conformal structure $S$ on the domain (by Lemma~\ref{lemma.fibers}). Hence so does $\tilde u_K$. This proves the claim of this case for $1+c \geq 2$.

When $1+c = 1$, we note that if $u_j$ is non-constant, then the remaining component 
	\eqnn
	\pi_{M \times T^*F^{[N] - K - \{j\}}}  \circ \tilde u_K
	\eqnd
of $\tilde u_K$ must be  trivial for $\tilde u_K$ to be in a 0-dimensional component of moduli space. But this means $\tilde u_K$ cannot satisfy the requisite boundary conditions! So $\tilde u_K$ must come in a higher-dimensional family (after modding out by the $F$ action), hence the dimension-0 disk-count is zero. This proves the claim for $1+c = 1$.
\end{proof}

\subsection{Degeneracies}
By definition $\Xi$ respects face maps; but as it turns out, it may not respect all degeneracy maps $s_i$. All we know so far is that the diagram
	\eqn\label{eqn.deg-commute}
		\xymatrix{
			\lag(M)_0 \ar[r]^-{\Xi} \ar[d]_{s_0}
			&	\wrap_\compact(M)\Mod_0 \ar[d]^{s_0} \\
			\lag(M)_1 \ar[r]^-{\Xi}
			& \wrap_\compact(M)\Mod_1^{s_0}
		}
	\eqnd
commutes, by Proposition~\ref{prop.id}. Put another way, all identity morphisms are respected by $\Xi$. At this point, we ought to be happy: Whether a functor between unital ($\infty$-)categories respects units is a property, and not additional structure, so one knows that we can formally find an honest functor from $\lag$ to $\wrap_\compact\Mod$. 

We make this concrete in~\cite{tanaka-non-strict}. There we prove that any assignment of simplices $\Xi$ satisfying (\ref{eqn.deg-commute}) is homotopic to an actual functor $\Xi'$. Moreover, $\Xi'$ has the following property: For any simplex $Y$ of the domain, $\Xi(C)$ and $\Xi'(C)$ are homotopic in the target.

This completes the proof that $\Xi$ defines a functor on $\lag^{\dd 0}$.

\subsection{Stabilization}
So far we have defined a functor
	\eqnn
		\lag^{\dd 0} \to \fuk\Mod.
	\eqnd
Now, for every $i$, we must construct a commutative diagram
	\eqnn
		\xymatrix{
			\lag^{\dd i+1} \ar[r]^{\Xi} & \fuk\Mod\\
			\lag^{\dd i} \ar[u]^{\times E^\vee} \ar[ur]_{\Xi}
		}
	\eqnd
(which we already saw in (\ref{eqn.stabilization-criterion}).)
Then by the definition of $\lag := \bigcup_{i \geq 0} \lag^{\dd i}$ as a colimit, we have defined the map $\Xi$. Recall from the definition of $\lag$ that the vertical arrow in the above diagram is given by taking the direct product of a brane/cobordism in $\lag^{\dd i}$ with the vertical fiber at the origin of $T^*E$, which we write as $E^\vee = T_0^* E$ for brevity.

\begin{remark}
We define $\Xi$ on the full subcategory of those $Y \subset M \times T^*E^n$ which are transverse to $M \times E^n \subset M \times T^*E^n$.\footnote{Recall from (\ref{eqn.transverse}), and the surrounding discussion, that the inclusion of this full subcategory is essentially surjective, as a compactly supported Hamiltonian isotopy of an object can be chosen arrange for this transversality condition.}
\end{remark}

\begin{defn}[$\Xi$ on $\lag^{\dd n}$.]
For every object $Y \in \ob \lag^{\dd n}$, choose a collection of staircase curves $\beta_i$ whose heights and depths are very close to 0. (This ensures that $Y$ is transverse to $M \times \beta_i$ for each $i$.)

\begin{itemize}
	\item
		On objects, we define $\Xi(Y)$ to be the module
			\eqnn
				CF^*(- \times \beta^n, Y).
			\eqnd
		Let us elaborate. On any tuple of objects $X_0,\ldots,X_k$, the operation
			\eqnn
				CF^*(X_k \times \beta_k^n, Y) \tensor \ldots \tensor WF^*(X_0,X_1) \to CF^*(X_0 \times \beta_0^n, Y)
			\eqnd
		is obtained by (i) counting holomorphic disks with boundary on
			\eqnn
				Y, X_k \times \beta_k^n, \ldots, X_0 \times \beta_0^n
			\eqnd
		and (ii) utilizing the natural isomorphism of cochain complexes
			\eqnn
				CF^*(X_i \times \gamma_i, X_j \times \gamma_j) \cong WF^*(X_i, X_j),
				\qquad
				i\neq j.
			\eqnd
		Note that absorbed into this are the usual continuation map equivalences that show $CF^*(X \times \beta_i^n, Y)$ is equivalent to $CF^*(X \times \beta_j^n, Y)$.\footnote{Note that the $X$ is the same here, but the $i$ and $j$ are not.}
	\item
		On (higher) morphisms, we employ the exact same methods as for $\lag^{\dd 0}$. Namely, given a cobordism $Y \subset M \times T^*E^n \times T^*F^N$, we pair it against branes of the form
			\eqnn
				(X_i \times \beta_i^n) \times \gamma_i^N.
			\eqnd
		Then all the algebra and proofs from the previous section go through mutatis mutandis, and these operations define a functor of $\infty$-categories.
\end{itemize}
\end{defn}

\begin{lemma}\label{lemma.stabilization}
The diagram~(\ref{eqn.stabilization-criterion}) commutes.
\end{lemma}

\begin{proof}
We must show a natural equivalence of modules
	\eqnn
		CF^*(-, Y \times E^\vee) \simeq CF^*(-,Y).
	\eqnd
This follows from the exact same method as in Lemma~\ref{lemma.collapse}: The curves $E^\vee, \beta_0,\ldots,\beta_k$  form the same boundary conditions as found in region $R$ (see 1. of Figure~\ref{figure.staircase-disks}). As we showed in the proof of Lemma~\ref{lemma.collapse}, holomorphic disks in $(M' \times T^*E) \times T^*F^N)$ with boundary on product branes, and with image in region $R \subset T^*E$, are in bijection with  holomorphic disks in $M' \times T^*F^N$. Setting $M' = M \times T^*E^n$, this completes the proof that (\ref{eqn.stabilization-criterion}) commutes.
\end{proof}
% Activate the following line by filling in the right side. If for example the name of the root file is Main.tex, write
% "...root = Main.tex" if the chapter file is in the same directory, and "...root = ../Main.tex" if the chapter is in a subdirectory.
 
%!TEX root = _pairing.tex

\section{Appendix: The $B$ construction's faces}
\subsection{Proof of Lemma~\ref{lemma.B-faces}}
Throughout, we follow the conventions and notation of Section~\ref{section.cube-notation}. For utility, consider the map
	\eqnn
	T^*\RR_N
	\to
	T^*\RR_i
	\qquad
	(q_N,p_N) \mapsto (q_N + w_i, p_N).
	\eqnd
This induces a symplectic embedding
	\eqn\label{eqn.alpha}
	\alpha_{N,i} : M \times T^*\RR^{[N] \setminus\{0,i,N\}}
	\times T^*[0,i]_{q_N}
	\to
	M \times T^*\RR^{[N] \setminus\{0,i,N\}} \times T^*[w_i,w_i+i]_{q_i}.
	\eqnd
Here's the main use of this map: When studying the $N$th front face of $B_i(Y)$, note that $\Phi$ precisely has the effect of sending 
	\eqnn
	(\del_i^{\back} Y)|_{[0,i]_{q_N}}
	\eqnd
to the set
	\eqnn
	\alpha_{N,i} (\del_i^{\back} Y)|_{[0,i]_{q_N}}). 
	\eqnd

\begin{lemma}\label{lemma.B1}
Assume $Y$ is collared. Then
	\enum
		\item For all $k < N$, 
			\eqnn
				\del_k^{\front} B'_i(Y)
				\cong
				\begin{cases}
					B'_i(\del_k^{\front} Y) & i \neq k\\
					\del_i^{\front} Y & i = k
				\end{cases}
			\eqnd
		\item For all $k < N$, 
			\eqnn
				\del_k^{\back} B'_i(Y)
				\cong
				\begin{cases}
					B'_i(\del_k^{\back} Y) & i \neq k\\
					(\del_i^{\back} Y)|_{[i,N]_{q_N}} & i = k
				\end{cases}
			\eqnd
		\item 
			\eqnn
				\del_N B'_i(Y) \cong
				\del_N Y \bigcup \alpha_{N,i} \left( (\del_i^{\back} Y)|_{[0,i]_{q_N}}\right)
			\eqnd
		\item
			\eqnn
				\del_N^{\back} B'_i (Y)
				\cong
				\del_N^{\back} Y.
			\eqnd
	\enumd
Here, the isomorphism symbol means that there is a collared re-parametrization of the $q$ coordinates taking one manifold to the other.
\end{lemma}

\begin{proof}[Proof of Lemma~\ref{lemma.B1}.]
1. Recall 
	\eqnn
	\del_k B'_i(Y) = (B'_i Y)|_{q_k = 0}.
	\eqnd
Assume $i \neq k< N$.
Then by the formulas for the $B'_i$---which are oblivious to the value of $q_k$---we see that this does indeed equal $B'_i(\del_k^{\front} Y)$. Specifically, restricting the definition of $\Phi$ (\ref{eqn.Phi}) to $q_k = 0$ is the same as performing $\Phi$, then restricting to the locus $q_k = 0$ since $\Phi$ acts by the identity on the $q_k$ coordinate. Likewise, in (\ref{eqn.B_i}), one could replace every instance of $Y$ on the righthand side with $\del_k Y$; this is exactly what collars $B_i(Y)$ along the face $q_k = 0$.

Now assume $i=k < N$. The construction of $B_k(Y)$ leaves unaffected the portion of $Y$ with $q_k$-coordinate close to 0. (Indeed, every term on the righthand side of (\ref{eqn.B_i}) except $Y$ lives in a region where $q_i \geq w_i$.) So $\del_i B_i(Y) = \del_i Y$.

2. By the same reasoning as above, the claim for $i \neq k <N$ follows. When $i=k$, we look to the definition of $B_i$ (\ref{eqn.B_i}). By construction, the portion of $B_i(Y)$ that lives above the face $q_k = w_k + k$ is given by the terms in (\ref{eqn.B3}) and (\ref{eqn.B4}). The term $V_i^{big} \times (\del_i Y)|_{q_N = i}$ is by definition, along $q_k = w_k + k$, just a copy of $[0,k]_{q_N} \times (\del_i Y)|_{q_N = i}.$ While the second term (\ref{eqn.B4}) is, by definition, collared at $q_k = w_k +k$ by $(\del_i Y)|_{[i,N]_{q_N}}$. Gluing these together, one can reparametrize the interval where $q_N \leq i$---coming from (\ref{eqn.B3})---to obtain the desired result. Put another way, the union of (\ref{eqn.B3}) and (\ref{eqn.B4}) along the $i$th face is just given by extending $(\ref{eqn.B4})$ to the interval $[0,i]_{q_N}$ by the zero section, so the result follows.

3. Let us break up the face $\del_N B_i(Y)$ into the portion with $q_i \leq w_i$ and $q_i \geq w_i$. The former is, as in the previous case, given by the first line in the definition of $B_N$ (\ref{eqn.B_i}). Hence it is equal to $\del_N Y$. Meanwhile, the portion with $q_i \geq w_i$ is just an extension by the zero section of the portion (\ref{eqn.B2}). By construction, $\Phi$ sends the $N$th coordinate of $\del_i Y$ to the $i$th coordinate, translated by $w_i$. Hence one concludes that the $q_i \geq w_i$ portion of $\del_N B_i(Y)$ is given by
	\eqnn
		\alpha_{N,i} \left( (\del_i^{\back} Y)|_{[0,i]_{q_N}}\right).
	\eqnd 

4. In the definition of $B_i$, the only portions that contribute to the face $q_N = N$ are (\ref{eqn.B_i}) and (\ref{eqn.B4}). The former is collared by $\del_N Y$ at $q_N = N$ by definition, and the latter likewise extends $\del_i Y$ by the zero section in the $q_i$ direction. This concludes the proof.
\end{proof}

\begin{corollary}
Assume $Y$ is collared. Then
	\enum
		\item For all $k < N$, 
			\eqn\label{eqn.cor-front}
				d_k^{\front} B'_i(Y)
				\cong
				\begin{cases}
					B'_{i-1}(d_k^{\front} Y) & i > k\\
					d_i^{\front} Y & i = k \\
					B'_{i}(d_k^{\front} Y) & i <k
				\end{cases}
			\eqnd
		\item For all $k < N$, 
			\eqn\label{eqn.cor-back}
				d_k^{\back} B'_i(Y)
				=
				\begin{cases}
					B'_{i-1}(d_k^{\back} Y) & i \neq k\\
					(d_i^{\back} Y)|_{[i,N]_{q_N}} & i = k
				\end{cases}
			\eqnd
		\item 
			\eqnn
				d_N B'_i(Y) =
				d_N Y \bigcup \alpha_{N,i} \left( (d_i^{\back} Y)|_{[0,i]_{q_N}}\right)
			\eqnd
		\item
			\eqnn
				d_N^{\back} B'_i (Y)
				=
				d_N^{\back} Y.
			\eqnd
	\enumd
\end{corollary}

\begin{lemma}
	\enum
		\item For $0 < k < N$, 
			\eqnn
				d_k B'(Y) = B'(d_k Y).
			\eqnd
	\enumd
\end{lemma}

\begin{proof}

	\begin{align}
		d_k B'(Y)
		&= d_k \circ B_{N-1} \circ \ldots \circ B_1 (Y) \nonumber \\
		&\cong B_{N-2} \circ \ldots \circ B_k (d_k \circ B_k \circ B_{k-1} \circ \ldots \circ B_1 (Y))  \nonumber\\
		&\cong B_{N-2} \circ \ldots \circ B_k (d_k \circ B_{k-1} \circ \ldots \circ B_1 (Y))  \nonumber\\
		&\cong B_{N-2} \circ \ldots \circ B_1 (d_k Y) \nonumber \\
		&\cong B' (d_k Y).\nonumber
	\end{align}
Every line except the last line is a consequence of (\ref{eqn.cor-front}). (The last line is the definition of $B'$.)
\end{proof}

\bibliographystyle{amsalpha}
\bibliography{biblio}

\end{document}